\documentclass[10pt]{amsart}

\usepackage{amssymb}
\usepackage{amsmath}
\usepackage{amsthm}
\usepackage{color}
\usepackage{euscript}
\usepackage[linktoc=page,colorlinks,citecolor = green!50!black,linkcolor= blue,urlcolor = blue]{hyperref}
\usepackage{enumitem}
\usepackage{tikz}
\usepackage{cleveref}
\usepackage{mathrsfs}
\usepackage[all]{xy}
\usepackage{verbatim}
\usepackage{comment}
\usepackage{bm}
\usepackage[margin=1in]{geometry}

\theoremstyle{plain}
\newtheorem{thm}{Theorem}[section]
\newtheorem{lem}[thm]{Lemma}

\newtheorem{cor}[thm]{Corollary}
\newtheorem{prop}[thm]{Proposition}

\newtheorem{assump}[thm]{Assumption}

\theoremstyle{definition}
\newtheorem{defn}[thm]{Definition}
\newtheorem{exmp}[thm]{Example}

\theoremstyle{definition}
\newtheorem{note}[thm]{Note}
\newtheorem{rk}[thm]{Remark}
\newtheorem*{notation}{Notation}

\newcommand{\bC}{{\mathbb C}}
\newcommand{\bD}{{\mathbb D}}

\newcommand{\bN}{{\mathbb N}}

\newcommand{\bQ}{{\mathbb Q}}
\newcommand{\bR}{{\mathbb R}}

\newcommand{\bZ}{{\mathbb Z}}

\newcommand{\cB}{{\mathcal B}}
\newcommand{\cC}{{\mathcal C}}

\newcommand{\cE}{{\mathcal E}}
\newcommand{\cF}{{\mathcal F}}

\newcommand{\cK}{{\mathcal K}}
\newcommand{\cL}{{\mathcal L}}
\newcommand{\cM}{{\mathcal M}}
\newcommand{\cN}{{\mathcal N}}
\newcommand{\cO}{{\mathcal O}}
\newcommand{\cP}{{\mathcal P}}

\newcommand{\cU}{{\mathcal U}}

\newcommand{\cY}{{\mathcal Y}}

\newcommand{\Z}{\mathbb{Z}}
\newcommand{\Q}{\mathbb{Q}}

\newcommand{\B}{\mathcal{B}}

\newcommand{\fM}{\mathfrak M}
\newcommand{\fN}{\mathfrak N}
\newcommand{\faM}[2]{\mathfrak M_\alpha^{#2}|_{z=T^{#1}}}
\newcommand{\fh}{\mathfrak{h}}

\newcommand{\relotimes}{\overset{rel}{\pmb{\pmb{\otimes}}}}
\newcommand{\famotimes}{{\pmb{\pmb{\otimes}}}}

\newcommand{\red}[1]{\textcolor{red}{#1}}

\newcommand{\sheafhom}{\mathcal{H} \kern -.5pt \mathit{om}}

\numberwithin{equation}{section}
\numberwithin{figure}{section}
\title{Iterations of symplectomorphisms and $p$-adic analytic actions on the Fukaya category}
\author{Yusuf Bar{\i}\c{s} Kartal}

\address{University of Edinburgh, Edinburgh, UK}
\email{ykartal@ed.ac.uk}
\date{}
\begin{document}
\begin{abstract}
Inspired by the work of Bell on the dynamical Mordell-Lang conjecture, and by family Floer cohomology, we construct $p$-adic analytic families of bimodules on the Fukaya category of a monotone or negatively monotone symplectic manifold, interpolating the bimodules corresponding to iterates of a symplectomorphism $\phi$ isotopic to the identity. This family can be thought of as a $p$-adic analytic action on the Fukaya category. Using this, we deduce that the ranks of the Floer cohomology groups $HF(\phi^k(L),L';\Lambda)$ are constant in $k\in\bZ$, with finitely many possible exceptions. We also prove an analogous result without the monotonicity assumption for generic $\phi$ isotopic to the identity by showing how to construct a $p$-adic analytic action in this case. We give applications to categorical entropy and a conjecture of Seidel. 
\end{abstract}
\keywords{categorical dynamics, Fukaya category, dynamical Mordell-Lang, p-adic analytic action}
	\maketitle
\setcounter{tocdepth}{1}
\tableofcontents
\parskip1em
\parindent0em	

\section{Introduction}\label{sec:intro}
\subsection{Motivation and the main results}
In \cite[Theorem 1.3]{skolemmahler}, Bell proves the following theorem: let $Y$ be an affine variety over a field of characteristic $0$ and $f:Y\to Y$ be an automorphism. Consider a subvariety $X\subset Y$ and a point $x\in Y$. Then, the set $\{n\in\bN: f^n(x)\in X \}$ is a union of finitely many arithmetic progressions and a set of finitely many numbers. In \cite{seideliterates}, Seidel conjectures a symplectic version of this statement. Namely, given a symplectic manifold $(M,\omega_M)$, a symplectomorphism $\phi$, and two Lagrangians $L,L'\subset M$ (for which Floer cohomology is well-defined), the set 
\begin{equation}
\{n\in\bN:\phi^n(L) \text{ and }L'\text{ are Floer theoretically isomorphic} \}
\end{equation}
forms a union of finitely many arithmetic progressions and a set of finitely many numbers. Recall that $L$ and $L'$ are Floer theoretically isomorphic if there are elements $\alpha\in HF(L,L')$, $\alpha'\in HF(L',L)$ such that $\mu^2(\alpha',\alpha)=1_L$ and $\mu^2(\alpha,\alpha')=1_{L'}$. This is equivalent to the existence of isomorphisms $HF(K,L)\cong HF(K,L')$ that is natural in $K$, i.e. $L$ and $L'$ cannot be distinguished by their Floer theory. 

The purpose of this paper is to prove a version of this statement when $\phi\in Symp^0(M,\omega_M)$, i.e. $\phi$ is isotopic to the identity through symplectomorphisms. Our main theorem holds when $(M,\omega_M)$ is monotone or negatively monotone. Namely: 
\begin{thm}\label{thm:mainthm}
Assume that $(M,\omega_M)$ is a monotone or negatively monotone symplectic manifold and $L,L'\subset M$ are Lagrangians such that Assumption \ref{assumption:monotoneplus} is satisfied. Then, the rank of $HF(\phi^k(L),L';\Lambda)$ is constant in $k\in\bZ$ except for finitely many $k$.
\end{thm}
Here, $HF(L,L';\Lambda)$ denotes the Lagrangian Floer cohomology group defined with coefficients in the Novikov field \begin{equation}\label{eq:novikovfield}
	\Lambda=\bQ((T^\bR))=\bigg\{\sum_{i=0}^\infty a_iT^{r_i}: a_i\in\bQ, r_i\in\bR,r_i\to\infty \bigg\}.
\end{equation} We will sometimes omit $\Lambda$ from the notation and denote this group by $HF(L,L')$. 
\begin{assump}\label{assumption:monotoneplus}
The monotone Fukaya category $\cF(M;\Lambda)$ is smooth (\Cref{defn:homolsmooth}), and it is split generated by a set of Lagrangians $L_1,\dots ,L_m$ with minimal Maslov number at least $3$ such that one of the following is satisfied:
\begin{enumerate}
	\item\label{cond:torsion} each $L_i$ is monotone and the image of $\pi_1(L_i)\to\pi_1(M)$ is torsion for all $i$, or
	\item\label{cond:balanced} $L_i$ is Bohr-Sommerfeld monotone for all $i$.
\end{enumerate}
Also, $L,L'$ are tautologically unobstructed and have minimal Maslov number at least $3$. 
\end{assump}
We initially assume that either (i) $L,L'$ are also monotone with torsion image in the fundamental group $\pi_1(M)$ (when \eqref{cond:torsion} is satisfied for $L_i$), or (ii) they are Bohr-Sommerfeld monotone (when \eqref{cond:balanced} is satisfied for $L_i$). The theorem is valid essentially whenever $L$ and $L'$ define objects of the Fukaya category. However, we start by explaining the proof under these stronger assumptions on $L$ and $L'$ as the proof is more geometric in this case. At the end of Section \ref{sec:comparisonandmaintheorem}, we will explain how to drop the assumptions on $L$ and $L'$. 

Note that in the examples we have, the Lagrangians $L_i$ satisfy the latter assumption and not the former. 
\begin{exmp}\label{exmp:genus2exmp}
One can let $M=\Sigma_g$ be a surface of genus greater than or equal to $2$. Finite generation of $\cF(\Sigma_g,\Lambda)$ is shown in \cite{seidelgenustwo} and \cite{efimovhighergenus}, and that this category is homologically smooth follows from the fact that matrix factorization categories are homologically smooth by Dyckerhoff \cite{dyckerhoffcompactgenmf}, or alternatively by Auroux-Smith \cite[Lemma 2.18]{smithauroux}. 
That one can let the split generators be Bohr-Sommerfeld monotone follows from the fact that every non-separating curve has such a representative in its isotopy class by Seidel \cite{seidelgenustwo} (note that loc.\ cit.\ uses the term balanced for Bohr-Sommerfeld monotone). 
Let $\ell_1$, $\ell_2$, $\ell_3$ be as in \Cref{figure:exmpsurfaces}. More precisely, let $\ell_1\subset \Sigma_g$ be a non-separating simple closed curve with primitive homology class in $\Sigma_g$ (in particular it is not null-homologous). One can let $\ell_1$ be one of the meridians in a decomposition $\Sigma_g=(T^2)^{\#g}$. Let $\phi$ be a symplectomorphism with small flux that disjoins $\ell_1$ from itself and let $\ell_2=\phi(\ell_1)$, $\ell_3=\phi^2(\ell_1)$ (also assume that $\ell_1\cap \ell_3=\emptyset$). Let $L=\ell_1$ and $L'=\ell_2\cup \ell_3=\ell_2\oplus \ell_3$ equipped with Spin structures. Applying Theorem \ref{thm:mainthm}, we see that the rank of $HF(\phi^k(L),L';\Lambda)$ is constant except for finitely many $k$. In this example, the finite exceptional set is $k=1,2$. Note that $\ell_2\cup\ell_3$ cannot be Bohr-Sommerfeld monotone; however, we will explain how to drop this assumption at the end of Section \ref{sec:comparisonandmaintheorem}. 
\end{exmp}
\begin{exmp}
\Cref{thm:mainthm} is valid for any pair of objects of the Fukaya category. Let $\ell$ be a non-separating curve $\ell\subset \Sigma_2$ with primitive homology class, that is fixed by $\phi$, and that intersects each of $\ell_1$, $\ell_2$ and $\ell_3$ exactly at one point as in \Cref{figure:exmpsurfaces}. Let $L=\ell\oplus\ell_1$ and $L'=\ell_2\oplus \ell_3$. Then, the dimension $HF(\phi^k(L),L';\Lambda)$ is given by $2,4,4,2,2,2,2,2,2,\dots$ for $k=0,1,2,\dots$. Note that the eventual dimension always has to be the minimum of the sequence. 
If instead, one lets $L=\ell^{\oplus 2}\oplus\ell_1$ and $L'=\ell_2\oplus \ell_3$, then, the dimension $HF(\phi^k(L),L';\Lambda)$ is given by $4,6,6,4,4,4,4,\dots$ for $k=0,1,2,\dots$. 
\end{exmp}
\begin{figure}\centering
	\includegraphics[height=4 cm]{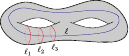}
	\caption{Some Lagrangians on $\Sigma_2$}
	\label{figure:exmpsurfaces}
\end{figure}
\begin{exmp}
To give a more geometric example, let $M=\Sigma_2\times \Sigma_2$. Let $\ell_1,\ell_2,\ell_3,\ell$ be as in \Cref{exmp:genus2exmp}, and let $\phi'=\phi\times\phi$. Consider the Lagrangian tori $\ell_1\times\ell$ and $\ell\times\ell_1$. These tori intersect at one point, defining a morphism $\ell_1\times\ell\to\ell\times\ell_1$ (of non-zero degree possibly). Let $L$ be a Lagrangian representing the cone of this morphism, which can be obtained by Lagrangian surgery. We let $L'=\ell_2\times\ell_3$. Clearly, $\phi'^{-k}(L')$ has non-vanishing Floer cohomology with $\ell_1\times\ell$ only when $k=1$ and with $\ell\times\ell_1$ only when $k=2$. Therefore, $HF(\phi'^k(L),L';\Lambda)$ is non-zero only at $k=1,2$. One can produce examples with more sophisticated finite sets of exceptional $k$ in this way. 
\end{exmp}
\begin{exmp}
For another example of the same nature, let $L=cone(\ell_1\times \ell_2\to \ell\times\ell)$, the cone of the morphism corresponding to the intersection point of $\ell_1\times \ell_2$ and $\ell\times\ell$. One can apply surgery again to represent $L$ by an embedded Lagrangian. Let $L'=\ell_2\times\ell_3$. Applying $\phi'=\phi\times\phi$ iteratively, we find that
\begin{equation}\label{eq:conelagrexmp}
	HF(\phi'^k(L),L')=cone (HF(\phi^k(\ell_1)\times\phi^k(\ell_2),\ell_2\times\ell_3 )\to HF(\ell\times\ell,\ell_2\times\ell_3 ))
\end{equation}
can jump only at $k=1$, as the first term of the cone otherwise vanishes. Hence, the dimension of \eqref{eq:conelagrexmp} is equal to $1$, except $k=1$. At $k=1$, the left-hand side becomes $HF(\ell_2\times\ell_3,\ell_2\times\ell_3 )$, which has dimension $4$, and the morphism is non-zero on $id_{\ell_2\times\ell_3}$. This is sufficient to conclude that the total dimension of \eqref{eq:conelagrexmp} is $3$ for $k=1$. As a result, the dimension of \eqref{eq:conelagrexmp} is given by $1,3,1,1,1,1,1,\dots$ for $k=0,1,2,\dots$. If one replaces $L$, resp. $L'$ by $L\times \ell\subset\Sigma_2^{\times 3}$, resp. $L'\times \ell\subset\Sigma_2^{ \times 3}$ and $\phi'$ by $\phi\times\phi\times\phi$ (or $\phi\times\phi\times id_{\Sigma_2}$), one obtains the dimension sequence $2,6,2,2,2,2,\dots$ for $k=0,1,2,3,\dots$. 
\end{exmp}
The jump locus can be made bigger and the minimum rank can be increased. Indeed, we would hypothesize that for a sequence of natural numbers $(a_k)_k$ to be realized as such a dimension sequence, the only constraint would be $a_k=\min_k a_k$ for $k\gg 0$. This condition is needed as the proof implies that the rank jumps up at the exceptional $k$. 
\begin{exmp}
Let $(a_k)_k$ be a sequence of natural numbers and assume that (i) $a_k=\min_k a_k:=a$ for $k\gg 0$, (ii) the parity of $a_k$ is constant, i.e. $a_k-a$ is always even. In other words, the sequence is of the form $a_k=a+2b_k$ for $a,b_k\in\bN$ such that $b_k=0$ for $k\gg 0$. Let $L=\ell_1$ and $L':=\ell^{\oplus a}\oplus \bigoplus_k \big(\phi^k(\ell_1)\big)^{\oplus b_k}$. Using (i) $\ell$ is fixed by $\phi$, (ii) $HF(\ell_1, \ell)$ is one dimensional, (ii) $HF(\phi^{i}(\ell_1),\phi^j(\ell_1))=0$ unless $i=j$; we calculate that $dim(HF(\phi^k(L),L'))=a+2b_k=a_k$. 
\end{exmp}
\begin{exmp}\label{exmp:nonmonot}
The monotonicity assumption is essential. Indeed, let $M=T^2=\bR^2/\bZ^2$, $\ell\subset T^2$ be given by $x=0$, and $\phi$ be given by $(x,y)\mapsto (x+1/3,y)$. 
Then, the dimension of $HF(\phi^k(\ell),\ell;\Lambda)$ is periodic of period $3$, and is given by $2,0,0,2,0,0,2,0,0,\dots$ for $k=0,1,2,3,,\dots$.
\end{exmp}
After the proof of Theorem \ref{thm:mainthm}, we seek ways to drop the assumption of monotonicity. As illustrated in \Cref{exmp:nonmonot}, it is not valid without this assumption. However, we still prove the following:
\begin{thm}\label{thm:generic}
Assume that $(M,\omega_M)$ is a symplectic manifold with two Lagrangian branes $L,L'\subset \Lambda$ satisfying Assumption	\ref{assumption:generic}. Given generic $\phi\in Symp^0(M,\omega_M)$, the rank of $HF(\phi^k(L),L';\Lambda )$ is constant in $k\in\bZ$ with finitely many possible exceptions. 
\end{thm}
By generic, we mean that the flux of an isotopy from $1$ to $\phi$ is generic. This will be explained further in this section and in Section \ref{sec:generic}. 
\begin{assump}\label{assumption:generic}
The group $\pi_2(M)$ vanishes, and $\cF(M,\Lambda)$ is an (uncurved, $\bZ$ or $\bZ/2\bZ$-graded) $A_\infty$-category over $\Lambda$, it is smooth and proper (\Cref{defn:homolsmooth}), split generated by tautologically unobstructed Lagrangians $L_1,\dots, L_m$. Also, $L,L'$ are tautologically unobstructed Lagrangians with brane structures that bound no discs of Maslov index $2$ or less (so they define objects of the Fukaya category).
\end{assump}
We let $\cF(M,\Lambda)$ denote the category spanned by the split generators $L_1,\dots, L_m$, as before. Theorem \ref{thm:generic} is valid in great generality. As examples of such $M$, one can consider $T^2$ and $T^2\times T^2$.

Next, we discuss applications. The first application is to categorical entropy, as defined in \cite{dynsysandcats}: 
\begin{cor}\label{cor:entropy}
Let $M$ and $\phi$ be as in either one of \Cref{thm:mainthm} or \Cref{thm:generic}. Consider the action of $\phi$ on the derived Fukaya category $D^\pi\cF(M,\Lambda)$, i.e. the split closed triangulated envelope of $\cF(M,\Lambda)$. Then the entropy of the induced equivalence is $0$. 
\end{cor}
\Cref{thm:mainthm} and \Cref{thm:generic} can be seen as analogues of the dynamical Mordell--Lang conjecture for coherent sheaves proven in \cite{dynmordellforcoherent}, rather than analogues of the theorem we cited at the beginning (\cite[Theorem 1.3]{skolemmahler}). However, we can conclude the symplectic analogue of \cite[Theorem 1.3]{skolemmahler} from \Cref{thm:mainthm} and \Cref{thm:generic}. The second application is a stronger version of Seidel's conjecture (\cite{seideliterates}):
\begin{cor}\label{cor:seidelsconj}
Let $M, L,L',\phi$ be as in the setting of either of \Cref{thm:mainthm} or \Cref{thm:generic}. Further, assume that $L$ is connected, and $\pi_1(L)$ is abelian. 
Then, the set 
\begin{equation}
	\{k\in\bN:\phi^k(L) \text{ and }L'\text{ are Floer theoretically isomorphic} \}
\end{equation}
is either empty, a singleton or the entire $\bN$. In other words, if $\phi^k(L)$ and $L'$ are isomorphic for two different $k\in\bN$, then they are isomorphic for all $k\in\bN$. 
\end{cor}

\subsection{Idea of the proof}
To prove the main theorems, we use an analogue of the $p$-adic reduction trick in algebraic geometry. Recall that the field of $p$-adic numbers $(\bQ_p,|\cdot|_p)$ is the completion of $\bQ$ with respect to the $p$-adic topology. More precisely, $\bQ$ carries a non-Archimedean norm $|a|_p:=p^{-val_p(a)}$, where $val_p(a)\in\bZ$ is the power of $p$ that ``divides'' $a$. As a result, two numbers are $p$-adically close if their difference can be divided by a high power of $p$. The norm $|\cdot|_p$ on the completion $\bQ_p$ is non-Archimedean, i.e. $|x+y|_p\leq \max\{|x|_p,|y|_p\}$. Analogous to the normed Archimedean field $(\bC,|\cdot|)$, it is possible to do analytic geometry over $\bQ_p$. If we let $\bD$ denote the unit disc $\{a\in\bQ_p:|a|_p\leq 1\}=:\bZ_{p}$, then the ring of analytic functions on $\bD$ are given by the series $\sum_{i=0}^\infty a_it^i$ that converge for $t\in \bD$. This ring is equal to the \emph{Tate algebra} defined by 
\begin{equation}
	\bQ_p\langle t\rangle:=\Bigg\{\sum_{i=0}^\infty a_it^i:|a_i|_p\xrightarrow{i\to\infty} 0 \Bigg\}.
\end{equation}
While having similarities to the Archimedean geometry, the analytic geometry over $\bQ_p$ has the following properties
\begin{enumerate}
	\item any $f(t)\neq 0\in \bQ_p\langle t\rangle$ has finitely many zeroes in $\bD$
	\item a disc $\bD_{\rho_0}:=\{a\in\bQ_p:|a|_p\leq \rho_0 \}$ centered at $0$ is an additive subgroup of $\bQ_p$
\end{enumerate}
The second claim, which follows from the non-Archimedean property is in sharp contrast with the complex analytic geometry. 

To motivate the proof method, we start by explaining how the proof of \cite[Theorem 1.3]{skolemmahler} goes: as before, let $Y$ be a complex affine variety, $X\subset Y$ be a subvariety (for simplicity, assume that $X$ is a hypersurface defined by $g=0$), $x\in Y$ be a complex point and $f:Y\to Y$ be an automorphism. The data defining $(Y,X=\{g=0\},x,f)$ is finite; therefore, the field of definition can be reduced to a finitely generated extension of $\bQ$. Denote the reduction by $(Y,X=\{g=0\},x,f)$ as well. Bell proves that this extension can be embedded into $\bQ_p$, for an infinite set of primes $p$, in such a way that the action of an iterate $f^m$ on the base change $Y_{\bQ_p}$ is interpolated by an action of the $p$-adic unit disc $\bD$. In other words, there exists an analytic action of the additive group $\bD$ on $Y_{\bQ_p}$ such that $t=1$ acts as $f^m$. Then, Bell considers the set 
\begin{equation}
	\{k\in\bN:f^{mk}(x)\in X_{\bQ_p} \}=\{k\in\bN:g\circ f^{mk}(x)=0 \}.
\end{equation}
It is the intersection of $\{t\in\bD=\bZ_p:g(t.x)=0 \}$ with $m\bN$, and as $t\mapsto g(t.x)$ is analytic, it either vanishes completely, or has finitely many zeroes only. In other words, the orbit of $x$ under $f^{mk}$ is either contained entirely in $X_{\bQ_p}$, or intersects it at finitely many points. Applying the same to $f(x), f^2(x),\dots, f^{m-1}(x)$ in place of $x$ leads to the conclusion, i.e. the set $\{k\in\bN:g\circ f^{k}(x)=0 \}$ is periodic with finitely many possible extensions (equivalently it differs from a union of finitely many arithmetic progressions by a finite set). 

Our proofs of \Cref{thm:mainthm} and \Cref{thm:generic} are similar in the sense that we interpolate the action of (an iterate of) $\phi$ on the Lagrangians by a $p$-adic analytic action. In other words, we construct a family of transformations parametrized by the $p$-adic unit disc $\bD$ that restricts to the identity at $t=0$. We prove the action property only in a small neighborhood $\bD_{p^{-n}}$ of $t=0$, which is an additive subgroup of $\bD$.

Consider the Fukaya category $\cF(M,\Lambda)$ spanned by the generators $\{L_i\}$ (this requires fixing brane structures and perturbation data as defined in \cite{seidelbook}). Every tautologically unobstructed Lagrangian $\tilde L$ with minimal Maslov number at least $3$ can be represented as a right module $h_{\tilde L}$ and a left module $h^{\tilde L}$ over $\cF(M,\Lambda)$. Moreover, the symplectomorphism $\phi$ induces an action on the category of modules over the Fukaya category. As in the classical Morita theory, the transformations of the category of modules are represented by $\cF(M,\Lambda)$-$\cF(M,\Lambda)$ bimodules (see e.g. \cite{kellerdg} for a treatment of Morita theory for dg categories, and \cite{kellerainfty,categoricaldynamics} for $A_\infty$-bimodules). The composition of $\cM_1$ and $\cM_2$ corresponds to the convolution $\cM_1\otimes_{\cF(M,\Lambda)}\cM_2$. 

For the rest of the introduction, we focus on \Cref{thm:mainthm}. The first analogous step is the reduction of the field of coefficients of $\cF(M,\Lambda)$. The Bohr--Sommerfeld monotonicity in \Cref{assumption:monotoneplus} can be thought of as an analogue of the exactness of a Lagrangian in a Liouville manifold, and it guarantees that the disc counts that we use to define the Fukaya category are finite (\cite{wehrheimwoodwardquilted1,ohmonotone}). Note that this condition depends on a choice of data (similar to the exactness of Lagrangians depending on a choice of primitive for the symplectic form). 
The implication is that the Fukaya category can be defined over the finitely generated extension
\begin{equation}
	\bQ(T^{E_1},\dots,T^{E_m})\subset \Lambda
\end{equation}
for some $E_1,\dots, E_m$ (\Cref{sec:definability}). Up to enlarging this subfield to a slightly bigger finitely generated extension, one can also guarantee that the modules $h_{L},h^{L'}$ are defined over the subfield (\Cref{lem:defnhlhl}). This field can be embedded into any $\bQ_p$ (cf. \Cref{lem:fieldmapqp}). As a result of the reduction and base change into $\bQ_p$, we obtain a category $\cF(M,\bQ_p)$ over the $p$-adic numbers, which should be thought of as a non-canonical Fukaya category with $p$-adic coefficients. 

The second major step is the reduction of the bimodule corresponding to $\phi$ such that after extension to the $p$-adic coefficient field, the action of $\phi$ is interpolated by an action of the $p$-adic unit disc. 

Fix a path from $1_M$ to $\phi$ through symplectomorphisms and assume that the flux of the path is $\alpha$ where $\alpha$ is a closed $1$-form. Every closed $1$-form generates a symplectic isotopy $\phi_\alpha^t$ as the flow of the vector field $X_\alpha$ satisfying $\omega(\cdot,X_\alpha)=\alpha$ and $\phi^1_\alpha$ is Hamiltonian isotopic to $\phi$ by \cite[Theorem 10.2.5]{mcduffsalamonsymplectic}. Therefore, the actions of $\phi$ and $\phi^1_\alpha$ on Floer cohomology are the same, and we may without loss of generality assume that $\phi=\phi^1_\alpha$. We can define a family of transformations (aka bimodules) over $\cF(M,\Lambda)$, denoted by $\fM_{\alpha}^\Lambda$. This family is parametrized by a rigid analytic version of the real line. In particular, there is an associated bimodule $\faM{f}{\Lambda}$ for every $f\in \bR$ (the reason for the choice of notation will be clear in the subsequent sections). Note that, we appeal to monotonicity again to ensure the family is defined over the entire ``rigid analytic real line'', rather than for $|f|$ sufficiently small. Fukaya's trick (\cite{abouzaidicm,fukayafamilyfloer}) and other geometric considerations imply that
\begin{enumerate}
	\item\label{item:firstint} for $|f|$ sufficiently small, the corresponding bimodule $\faM{f}{\Lambda}$ represents $\phi_\alpha^f$
	\item\label{item:secondint} for $|f|, |f'|$ sufficiently small, the family behaves like a group-action, i.e. $\faM{f}{\Lambda}\otimes_{\cF(M,\Lambda)}\faM{f'}{\Lambda}\simeq \faM{f+f'}{\Lambda}$
\end{enumerate}
A (slightly weaker) restatement of \eqref{item:firstint} is proven in \Cref{lem:nearbyfloer} and \eqref{item:secondint} is proven in \Cref{lem:grouplikenovikov}. To prove the latter, we write natural transformations (aka bimodule homomorphisms)
\begin{equation}\label{eq:mapintrotransfor}
	\faM{f}{\Lambda}\otimes_{\cF(M,\Lambda)}\faM{f'}{\Lambda}\to \faM{f+f'}{\Lambda}
\end{equation}
that vary ``analytically'' in $f,f'$. It is a standard quasi-isomorphism at $f=f'=0$, and we appeal to a semi-continuity argument to conclude the same for small $|f|,|f'|$. Hence, we obtain a ``local analytic action'' (\Cref{lem:grouplikenovikov}). Unfortunately, this argument does not extend beyond a small neighborhood of $f=0$. 

Analogous to the reduction of the field of definition of $\cF(M,\Lambda)$, we can reduce the field of definition of the family of bimodules $\fM_{\alpha}^{\Lambda}$ to a subfield $K\subset\Lambda$, and obtain a related analytic family $\fM_{\alpha}^{\bQ_p}$ of bimodules over $\cF(M,\bQ_p)$ by using analogous formulae. The parameter space of $\fM_{\alpha}^{\bQ_p}$ is the $p$-adic unit disc. Moreover, there are transformations 
\begin{equation}\label{eq:intropadicmap}
	\fM_{\alpha}^{\bQ_p}|_{t=t_1}\otimes_{\cF(M,\bQ_p)} 	\fM_{\alpha}^{\bQ_p}|_{t=t_2}\to 	\fM_{\alpha}^{\bQ_p}|_{t=t_1+t_2}
\end{equation}
analogous to \eqref{eq:mapintrotransfor}, that vary analytically in $t_1,t_2$ and that induce an isomorphism at $t_1=t_2=0$. A similar semi-continuity argument implies that \eqref{eq:intropadicmap} is an isomorphism for a small $p$-adic disc $\bD_{p^{-n}}=p^n\bZ_{p}$ (\Cref{prop:grouplikepadic}). As we remarked, $\bD_{p^{-n}}=p^n\bZ_{p}$ is an additive subgroup in sharp contrast to the Archimedean case. Therefore, we obtain a $p$-adic analytic action of $\bD_{p^{-n}}=p^n\bZ_{p}$, not just a local action as before. Note that this is one of the major reasons we pass to the $p$-adic numbers. 

As $\faM{f}{\Lambda}$ and $\fM_{\alpha}^{\bQ_p}|_{t=f}$ are related by base change and reduction, the same conclusion holds for $\faM{f}{\Lambda}$ over $\bQ\cap p^n\bZ_{p}=p^n\bZ_{(p)}$. 
In other words, $\faM{f}{\Lambda}\otimes_{\cF(M,\Lambda)}\faM{f'}{\Lambda}\simeq \faM{f+f'}{\Lambda}$ whenever $f,f'\in \bQ$ are small with respect to the $p$-adic topology. 
As a result, the set \begin{equation}
	\{f\in p^n\bZ_{(p)}: \faM{f}{\Lambda}\text{ represents }\phi_\alpha^f \}
\end{equation} is an additive subgroup of $p^n\bZ_{(p)}$ that contains all elements that are sufficiently close to $0$ with respect to the Euclidean topology (by Fukaya's trick). Therefore, for every $f\in p^n\bZ_{(p)}$, $\faM{f}{\Lambda}$ represents $\phi_\alpha^f$. This fact, combined with similar reduction and base change arguments, implies that the dimension of 
\begin{equation}
H^*\big(h_{L'}\otimes_{\cF(M,\Lambda)}\faM{-f}{\Lambda} \otimes_{\cF(M,\Lambda)}h^L\big)\cong HF(\phi_\alpha^f(L),L')
\end{equation}
over $\Lambda$ is the same as the dimension of $H^*(h_{L'}\otimes_{\cF(M,\bQ_p)}\fM^{\bQ_p}_{\alpha}|_{t=-f} \otimes_{\cF(M,\bQ_p)}h^L)$ over $\bQ_p$ (\Cref{prop:compareprop}), and by an application of Strassman's theorem, we show that this dimension can jump only finitely many times over the entire $p^n\bZ_{(p)}$. We use this, together with similar considerations to conclude \Cref{thm:mainthm}.

The proof of \Cref{thm:generic} is similar, with some important differences in the $p$-adic reduction steps. As there may be infinitely many discs, we define the ``$p$-adic Fukaya category'', essentially by replacing the Novikov variable $T$ with a $p$-adically small number. However, this creates obstructions to defining the $p$-adic analytic action using the same method, and we circumvent this problem by assuming $\alpha$ is generic. The key topological property that is guaranteed by the genericity is that the area of holomorphic discs cannot be the same as periods of $\alpha$.  

\subsection{Outline of the paper}
In Section \ref{sec:background}, we recall the basics of Fukaya categories as well as related homological algebra. We also introduce the field $K\subset\Lambda$, over which both the Fukaya category and the family of invertible bimodules are defined. In Section \ref{sec:families}, we recall the notion of families and their homological algebra. We also construct the Novikov and $p$-adic families $\fM_{\alpha}^{\Lambda}$, resp. $\fM_{\alpha}^{\bQ_p}$, and we establish the group-like property of $\fM^{\bQ_p}_{\alpha}$ over $\bQ_p\langle t/p^n\rangle$. Section \ref{sec:comparisonandmaintheorem} is devoted to the comparison of the algebraically constructed bimodule above to the Floer cohomology groups (Proposition \ref{prop:compareprop}). We then use this to conclude the proof of Theorem \ref{thm:mainthm}. We also explain how to the drop the Bohr-Sommerfeld monotonicity assumption on $L$ and $L'$. In Section \ref{sec:generic}, we prove Theorem \ref{thm:generic}, and in \Cref{sec:applications}, we prove \Cref{cor:entropy} and \Cref{cor:seidelsconj}. In \Cref{appendix:tatealgebras}, we recall the basics of Tate algebras. In \Cref{appendix:semicont}, we establish some semi-continuity results and give proof of \Cref{lem:grouplikenovikov}, which is a technical result we postpone into an appendix so as not to interrupt the flow of the paper.

\subsection*{Guide to the notation and the conventions}
In this subsection, we collect some of the notation and conventions in the paper for the convenience of the reader. These notions are discussed in later sections in detail, and we provide references in this subsection. The purpose of this subsection is as a reference guide while reading the paper, rather than introducing the notions below. Therefore, this subsection can be skipped in the first reading, and can be referred to later (except for the next paragraph). 

Recall that a symplectic manifold $(M,\omega_M)$ is called monotone, resp. negatively monotone, if $c_1(M)=\lambda[\omega_M]$ for some $\lambda>0$, resp. $\lambda<0$. Throughout the paper, we use the word \emph{monotone} to refer to both cases, i.e. we drop the prefix ``negatively''. In other words, $(M,\omega_M)$ is monotone if $c_1(M)=\lambda [\omega_M]$, for some $\lambda\neq 0$.

Various fields we use in the paper include the field of $p$-adic numbers $\bQ_p$ (\Cref{appendix:tatealgebras}), and the Novikov field 
\begin{equation}
	\Lambda=\bQ((T^\bR))=\bigg\{\sum_{i=0}^\infty a_iT^{r_i}: a_i\in\bQ, r_i\in\bR,r_i\to\infty \bigg\}.
\end{equation}
For an additive subgroup $G\subset \bR$, let $\bQ((T^G))$ denote the subfield of $\Lambda$ consisting of $\sum_{i=0}^\infty a_iT^{r_i}$ such that $r_i\in G$, and let $\bQ(T^G)$ denote the subfield of $\bQ((T^G))$ generated by $T^g$, $g\in G$. In other words, $\bQ(T^G)$ is the fraction field of the ring of finite sums $\sum_{i=0}^N a_iT^{r_i}$. Later in the paper we work with a specific finitely generated group $G$ (\Cref{defn:G}), and additionally with $G_{(p)}:=\{\frac{g}{n}:g\in G,p\nmid n \}$ (where $p$ is a fixed prime). In the subsequent sections following \Cref{defn:G}, we spare the letter $K$ to denote $\bQ(T^{G_{(p)}})\subset \Lambda$. In \Cref{defn:genericfield}, we define an analogue of the field $K$, denoted by $K_g$ (its definition is more involved). These are the fields over which the Fukaya category is defined in the situation of \Cref{thm:mainthm}, resp. \Cref{thm:generic}. We denote the Fukaya category over these fields by $\cF(M,K)$, resp. $\cF(M,K_g)$. 

We produce field homomorphisms $\kappa:K\to\bQ_p$ (\Cref{lem:fieldmapqp}), resp. $\kappa_g:K_g\to\bQ_p$ (\Cref{defn:genericfield}), and denote the base changes by 
\begin{equation}
	\cF(M,\bQ_p):=\cF(M,K)\otimes_K \bQ_p\text{, resp. }\cF(M,\bQ_p):=\cF(M,K_g)\otimes_{K_g} \bQ_p.
\end{equation}

We recall the $p$-adic numbers and Tate algebras to the extent we need in \Cref{appendix:tatealgebras}. Given $a\in\bQ_p$, let $val_p(a)$ denote the multiplicity of $p$ in $a$. We will frequently use the following Tate algebras:
\begin{align}
	\bQ_p\langle t\rangle:= \Bigg\{\sum_{i=0}^\infty a_it^i:a_i\in\bQ_p, \lim\limits_{i\to\infty} val_p(a_i)=\infty \Bigg\},\\
	\bQ_p\langle t_1,t_2\rangle:= \Bigg\{\sum_{i,j\geq 0} a_{ij}t_1^it_2^j:a_{ij}\in\bQ_p, \lim\limits_{i+j\to\infty} val_p(a_{ij})=\infty \Bigg\}.
\end{align}
As we remarked, $\bQ_p\langle t\rangle$ can be seen as the ring of analytic functions on the $p$-adic unit disc $\bD$. The ring $\bQ_p\langle t/p^n\rangle$ obtained by replacing $t$ with $t/p^n$ can be canonically identified with the ring of functions on the $p$-adic disc $\bD_{p^{-n}}=p^n\bZ_{p}$ of radius $p^{-n}$ (we refer to \eqref{eq:discpadicfnc} for details). Similar statements hold for $\bQ_p\langle t_1,t_2\rangle$. Throughout the paper, we use the notations $\bQ_p\langle t/p^n\rangle$, resp. $\bQ_p\langle t_1/p^n,t_2/p^n\rangle$ to denote the ring obtained from $\bQ_p\langle t\rangle$, resp. $\bQ_p\langle t_1,t_2\rangle$ by replacing $t$ with $t/p^n$, resp. $t_1,t_2$ with $t_1/p^n,t_2/p^n$. Details are given in \Cref{appendix:tatealgebras}. 

In \Cref{sec:families}, we introduce the notion of a family of left/right/bi-modules over an $A_\infty$-category $\cB$ following \cite{flux}. It will have multiple strains, such as a \emph{Novikov family} (\Cref{defn:novikovfamily}) and a \emph{$p$-adic family} (\Cref{defn:padicfamily}). For instance, a $p$-adic family $\fM$ of bimodules over a $\bQ_p$-linear $A_\infty$-category $\cB$ means a bimodule with a $\bQ_p\langle t\rangle$-linear structure that is compatible with the structure maps $\mu_{\fM}$. We refer to $\bQ_p\langle t\rangle$ as the base of the family. 

We denote ordinary modules/bimodules by letters such as $\cM$, $\cN$, $h_L$, $h^L$ (the last two denote right and left Yoneda modules defined as $h_L:=\cB(\cdot, L)$, resp. $h^L:=\cB(L,\cdot)$, \cite{generation,categoricaldynamics} for more details). We use the letters $\fM$, $\fN$ and $\fh$ to denote families. The specific families we define will be denoted by these letters, with the appropriate subscripts and superscripts. If $\fM$ is a $p$-adic family, and $a\in\bZ_{p}=\bD$, we have an evaluation map $\bQ_p\langle t\rangle\to\bQ_p$, $f(t)\mapsto f(a)$, and one can define an ordinary left/right/bi-module 
\begin{equation}
	\fM|_{t=f}:=\fM\otimes_{\bQ_p\langle t\rangle}\bQ_p.  
\end{equation}
Evaluation maps for other types of families are denoted similarly (details are provided later). 

A tensor product with no subscript means it is taken over the base field. Given $A_\infty$-bimodules $\cM$ and $\cN$, their convolution is denoted by $\cM_1\otimes_\cB \cM_2$ (\Cref{defn:ordinaryconv}, \cite{generation}), and it is an $A_\infty$-model for the derived tensor product. If $\fM_1$ is a $p$-adic family and $\cM_2$ is an ordinary bimodule, $\fM_1\otimes_\cB \cM_2$ carries a natural $\bQ_p\langle t\rangle$-linear structure, i.e. it is also a family. To emphasize that it inherits a family structure from $\fM_1$, we denote the convolution by $\fM_1\famotimes_\cB \cM_2$. For different orders of tensor products, and iterated tensor products, we use analogous notation. Note that this does not differ from $\fM_1\otimes_\cB \cM_2$ as a $\cB$-bimodule, and the purpose of the notation is to emphasize the family structure. 

Similarly, if $\fM_1$ and $\fM_2$ are two families of bimodules, one can consider their convolution relative to the base $\bQ_p \langle t\rangle$, which we denote by $\fM_1\relotimes_\cB\fM_2$. It is another family over $\bQ_p\langle t\rangle$, and we elaborate on this in \Cref{defn:reltensor}. For example, if $\cB=\bQ_p$, a $p$-adic family means a dg module over $\bQ_p\langle t\rangle$, and in this case $\fM_1\relotimes_{\cB}\fM_2$ is simply the tensor product over $\bQ_p\langle t\rangle$. 
%
%
%

\section*{Acknowledgments}
We would like to thank John Pardon for suggesting the semi-continuity argument in Proposition \ref{prop:grouplikepadic}, Umut Varolg\"une\c{s} and Mohammed Abouzaid for suggesting to consider negatively monotone and monotone symplectic manifolds, and Ivan Smith and Nick Sheridan for helpful conversations and email correspondence. We would like to thank Sheel Ganatra for pointing out to a reference and some suggestions. Finally, we wish to thank the anonymous referees for several helpful comments. Part of this work is conducted while the author is supported by ERC Starting Grant 850713 (Homological mirror symmetry, Hodge theory, and symplectic topology). 
%
%
%
%
%
%
%
%
%
\section{Background on Fukaya categories}\label{sec:background}
\subsection{Reminders and remarks on Fukaya categories and related homological algebra}\label{subsec:fukayacategory}
In this section, we will recall the basics of Fukaya categories and related homological algebra, and we will explain some of our conventions. Throughout the paper, $\Lambda$ denotes the Novikov field with rational coefficients and real exponents, i.e. $\Lambda=\bQ((T^\bR))$.

Let $L_1,\dots, L_m\subset M$ be monotone Lagrangians with minimal Maslov number at least $3$ that are oriented and equipped with $Spin$-structures. Assume that the Lagrangians $L_i$ are pairwise transverse. To define the Fukaya category whose objects are given by $L_1,\dots, L_m$, one counts marked holomorphic discs. More precisely, define 
\begin{equation}\label{eq:remindfloerdiff}
hom(L_i,L_j)=CF(L_i,L_j;\Lambda)=\Lambda\langle L_i\cap L_j\rangle  \text{ if }i\neq j.
\end{equation}
A generic choice of almost complex structure allows one to endow $hom(L_i,L_j)$ with a differential, defined by the formula $\mu^1(x)=\sum\pm T^{E(u)}.y$, where the sum runs over pseudo-holomorphic strips with boundary on $L_i\cup L_j$ and asymptotic to $x$ and $y$. Here, $E(u)$ denotes the symplectic area of the strip. 

More generally, given $(L_{i_0},\dots L_{i_q})$ that are pairwise distinct and $x_j\in L_{i_{j-1}}\cap L_{i_{j}}$, define 
\begin{equation}\label{eq:remindfukayastructure}
\mu^q(x_q,\dots ,x_1)=\sum\pm T^{E(u)} .y,
\end{equation}
where $y$ runs over intersection points $y\in L_{i_{0}}\cap L_{i_q}$ and $u$ runs over rigid marked pseudo-holomorphic discs with boundary on $\bigcup_j L_{i_j}$, and asymptotic to $\{x_j\}$ and $y$ near the markings. The $spin$ structures on the Lagrangians allow one to orient the moduli of such discs, determining the signs in (\ref{eq:remindfloerdiff}) and (\ref{eq:remindfukayastructure}) (cf. Remark \ref{note:signs}). By standard gluing and compactness arguments, this defines a $\bZ/2\bZ$-graded $A_\infty$-structure over $\Lambda$. The condition that the Maslov numbers of $L_i$ are at least $3$ implies that $A_\infty$ structure has no curvature. 

To include $hom(L_i,L_i)$ (as well as the situation when two of the Lagrangians $L_{i_q}$ and $L_{i_r}$ coincide), one can follow different options: the one that we take here is the approach via the count of pearly trees. Our main references are \cite[Section 7]{seidelgenustwo} and \cite[Section 4]{sheridanpopants}. Fix Morse-Smale pairs $(f_i,g_i)$ on each $L_i$, and define 
\begin{equation}
hom(L_i,L_i)=CM(f_i;\Lambda)=\Lambda \langle crit(f_i)\rangle \text{ if }i=j.
\end{equation}
The differential on $hom(L_i,L_i)$ is defined via the count of Morse trajectories. To define more general structure maps, one has to consider ``holomorphic pearly trees'', i.e. Morse flow lines connected by pseudo-holomorphic ``pearls'', \cite[Section 4]{sheridanpopants}. For simplicity, we assume that the Floer datum corresponding to the pair $(L_i,L_j)$ has vanishing Hamiltonian term, whenever $i\neq j$, while we still use perturbation data with non-vanishing Hamiltonians of the form $H\gamma$ (where $H\in C^\infty ([0,1]\times M)$ is a Hamiltonian function, $\gamma$ is a subclosed $1$-form on the surface that vanishes in the tangent directions to the boundary, and the Floer equation is $(du-X_H\otimes\gamma)^{0,1}=0$). 
The energy of a holomorphic pearly tree is defined to be the sum of the energy of all holomorphic pearls. We also refer the reader to \cite{clusterhomology} and \cite{lagrangianquantumhomology}. We will abuse the notation and denote the hom-sets by $CF(L_i,L_j;\Lambda)$ even when $L_i=L_j$. 

One could replace $T^{E(u)}.y$ terms in the definition of $A_\infty$-maps by $T^{E^{top}(u)}.y$, where $E^{top}(u)$ denotes the topological energy. For a Riemann surface $S$, and a map $u:S\to M$, the topological energy for a perturbation datum with Hamiltonian term $H\gamma$ (where $\gamma$ is a closed $1$-form on $S$) is defined by 
\begin{equation}
	E^{top}(u)=\int_Su^*\omega_M-d(u^*H \gamma).
\end{equation}
However, as $H\gamma$ vanishes on the strip-like ends corresponding to the intersection points of different Lagrangians (as the corresponding Floer datum has no Hamiltonian term, and the perturbation data is chosen consistently), and as $\gamma$ vanishes in directions tangent to the boundary of the surface, we have 
\begin{equation}
	\int_Sd(u^*H \gamma)=\int_{\partial S} u^*H\gamma=0.
\end{equation}
In other words, topological energy coincides with the symplectic area. We assume that the choice of $(H,J)$ is made so that the topological energy is larger than the geometric energy; therefore, positive (for instance, let $H\geq 0$, $d\gamma\leq 0$, cf. \cite[\S7.2]{abousei}). 

The reason we prefer this model of Fukaya categories over the one in \cite{seidelbook} is that it gives us better control over the topological energy of the discs. Namely, the topological energy of the discs all belong to the finitely generated group $\omega_M(H_2(M,\bigcup L_i;\bZ))$. Another reason we use this model is the convenience in applying Fukaya's trick (Lemma \ref{lem:halg=h}). 

Even though formally we are counting pseudo-holomorphic pearly trees, we will refer to them as ``pseudo-holomorphic discs'' throughout the paper, by abuse of terminology.  Similarly, to avoid confusion our figures will present discs, rather than pearly trees. 
\begin{note}
We must warn that in the upcoming figures such as Figure \ref{figure:novikovfamily} or Figure \ref{figure:groupquasi}, we use wavy lines going through the disc. This has nothing to do with the Morse trajectories of the pearly trees, rather they represent the homotopy class of a path in $M$ going from one input to the output. The meaning of this path is also clear for pearly trees. 
\end{note}
\begin{rk}\label{rk:seidelmodel}
	For more general Hamiltonian terms, the topological energy and the symplectic area may be different, but the difference depends on a quantity associated to each generator (namely the integral of the Hamiltonian along the chord, cf. \cite[(7.9)]{abousei}). Therefore, the $A_\infty$-structure defined using $T^{E^{top}(u)}.y$ in place of $T^{E(u)}.y$ is related to the latter by a rescaling of each generator. In particular, one can alternatively use the model presented in \cite{seidelbook}. Up to rescaling of the generators, the crucial property that the energy of discs lies in a finitely generated subgroup of $\bR$ holds. One also needs \Cref{lem:halg=h} to hold for this model, which we will elaborate on in \Cref{rk:seidelmodelftrick}.
\end{rk}

Let $\tilde L\subset M$ be another oriented, tautologically unobstructed 
Lagrangian brane of minimal Maslov number at least $3$ and equipped with a $Spin$-structure. Assume that $\tilde L\pitchfork L_i$ for all $i$. Then there exists a right, resp. left, $A_\infty$-module $h_{\tilde L}$, resp. $h^{\tilde L}$ such that 
\begin{equation}
h_{\tilde L}(L_i)=CF(L_i,\tilde L;\Lambda)\text{, resp. } h^{\tilde L}=CF(\tilde L,L_i;\Lambda),
\end{equation}
where the structure maps are defined analogously. These modules are defined as follows: we extend $\cF(M,\Lambda)$ by adding $\tilde L$, and these are the corresponding Yoneda modules. We denote the restriction of right and left Yoneda modules corresponding to $\tilde L$ to $\cF(M,\Lambda)$ by $h_{\tilde L}$ and $h^{\tilde L}$ respectively. If $\tilde L$ is not transverse to all $L_i$, one can apply a small Hamiltonian perturbation. Different Hamiltonian perturbations give rise to quasi-isomorphic modules over $\cF(M,\Lambda)$.
\begin{note}\label{note:signs}
Throughout the paper, we will omit the signs and write $\pm$, as they are standard (similar to above where we wrote $\sum \pm T^{E(u)}$ for the coefficients of the $A_\infty$-structure maps). Most of the sums we have are merely deformations of standard formulas and the signs do not change. For example, in the sums (\ref{eq:structurenovikov}), (\ref{eq:padicstructuremaps}), (\ref{eq:structuregenericnovikov}) and (\ref{eq:padicgenericstructuremaps}), one obtains the diagonal bimodule by putting $z=1$ (or $t=0$), and the signs are the same as those of the diagonal bimodule (and those of the $A_\infty$-structure coefficients). Similarly, the sums (\ref{eq:structurealgyoneda}), and (\ref{eq:structurealgyonedafamily}) share the same signs as the formulas defining the right Yoneda module. 
\end{note}
Throughout the paper, we will work with smaller fields of definition for the Fukaya category. In other words, if $K\subset \Lambda$ is a subfield containing all the coefficients $\sum\pm T^{E(u)}$ defining the $A_\infty$-structure, then one could as well follow the definition above to obtain a $K$-linear $A_\infty$-category, which we denote by $\cF(M,K)$. By base change along the inclusion map $K\to \Lambda$, one obtains the original category $\cF(M,\Lambda)$ (in other words, $\cF(M,\Lambda)=\cF(M,K)\otimes_K\Lambda$). If $K$ is a smaller field of definition and $\kappa:K\to Q$ is an arbitrary field extension, one obtains a category via base change, and we denote this category by $\cF(M,Q):=\cF(M,K)\otimes_K Q$, omitting $\kappa$ from the notation (we will later specialize to the case where $Q$ is the field of $p$-adics for some prime $p$ or a finite extension of it).

Similarly, if the coefficients defining the modules $h_{\tilde L}$ and $h^{\tilde L}$ belong to $K$, one can define right, resp. left modules over $\cF(M,K)$. Via base change along $\kappa:K\to Q$, one obtains modules over $\cF(M,Q)$. We keep the notation $h_{\tilde L}$ and $h^{\tilde L}$ for these modules. It is not necessarily true that these modules are invariant under Hamiltonian perturbations of $\tilde L$. The coefficients defining the continuation morphisms may not belong to the smaller subfield $K$.

For later use, fix a base point on $M$. In addition, given any generator of any morphism complex between a pair of Lagrangians that define objects of the Fukaya category fix a relative homotopy class of paths on $M$ from the base point to the generator. Similarly, for any such object $L_0$, corresponding to any generator of the complexes $h_{L'}(L_0)$ and $h^L(L_0)$, fix a relative homotopy class of paths on $M$ from the base point to the generator. 

Since $\cF(M,\Lambda)$ is not built to contain all Lagrangians, the following clarification is needed:
\begin{defn}\label{defn:splitgen}
We say \emph{$\{L_i\}$ split generates $\tilde L$}, if $\tilde L$, as an object of the Fukaya category $\cF'$ with objects $\{L_i\}\cup\{\tilde L\}$ is quasi-isomorphic to an element of the subcategory $tw^\pi(\cF)$ of $tw^\pi(\cF')$, where $\cF$ denotes the Fukaya category with objects $\{L_i\}$. We say \emph{$\{L_i\}$ split generates the Fukaya category}, if this holds for any $\tilde L$ as above.
\end{defn}
Recall that $tw(\cF)$ denotes the category of twisted complexes over $\cF$, and $tw^\pi(\cF)$ denotes its idempotent closure. 
\begin{notation}
Throughout the paper, $\cF(M,\Lambda)$ will always consist of objects $\{L_i\}$ split generating the Fukaya category. 
\end{notation}
\begin{rk}
If $\{L_i\}$ split generate $\tilde L$, then the modules $h_{\tilde L}$ and $h^{\tilde L}$ are perfect, i.e. they can be represented as a summand of a complex of Yoneda modules of $\cF(M,\Lambda)$. Equivalently, any closed module homomorphism from $h_{\tilde L}$, resp. $h^{\tilde L}$, to a direct sum of right, resp. left, $A_\infty$-modules factor through a finite sum (in cohomology, i.e. up to an exact module homomorphism). 
\end{rk}
One way to ensure split generation of the Fukaya category is the \emph{non-degeneracy} of $M$, i.e. $\{L_i\}$ split generate the Fukaya category if the open-closed map from the Hochschild homology of the category spanned by $\{L_i\}$ hits the unit in the quantum cohomology by Abouzaid \cite{generation} (loc.\ cit.\ proves this claim for wrapped Fukaya categories, but the proof goes through without change in the compact case).  
Another implication of non-degeneracy is (homological) smoothness:
\begin{defn}\label{defn:homolsmooth}
An $A_\infty$-category is called \emph{homologically smooth} (or just \emph{smooth}), if its diagonal bimodule is perfect, or equivalently if the diagonal bimodule can be represented as a direct summand of a twisted complex of Yoneda bimodules. An $A_\infty$-category is called \emph{proper}, if the hom-complexes have finite dimensional cohomology. Similarly, an $A_\infty$-module is called \emph{proper} if the complexes associated to every object have finite dimensional cohomology.
\end{defn}
In other words, if $M$ is non-degenerate, then $\cF(M,\Lambda)$ is homologically smooth by Ganatra \cite[Theorem 1.2]{sheelthesis}. It is also proper by definition. Similarly, the modules $h_{\tilde L}$ and $h^{\tilde L}$ are proper.
\begin{rk}
The smoothness of a category implies that the category is split generated by finitely many objects. Together with properness, it also implies that proper modules over the category are perfect, i.e. they can be represented as a direct summand of a complex of Yoneda modules (Lemma \ref{lem:propermodule}).
\end{rk}
We will make frequent use of the following lemma:
\begin{lem}\label{lem:basechange}
Given a smooth and proper $A_\infty$-category $\cB$ over a field $K$, 
and given proper (left/right/bi-) modules $\cN_1,\cN_2$, if $hom_{\cB^{mod}}(\cN_1,\cN_2)$ has finite dimensional cohomology, then this dimension does not change under the base change under a field extension $K\subset Q$. Moreover, the cohomologies are related by ordinary base change under $K\subset Q$.
\end{lem}
\begin{proof}
As we will see later in Lemma \ref{lem:propermodule}, the properness of $\cN_1$ and $\cN_2$ implies that they are actually perfect, i.e. they can be represented as Yoneda modules corresponding to twisted complexes with idempotents. Call these $X_1, X_2\in tw^\pi(\cB)$. As a result of Yoneda Lemma, 
\begin{equation}
dim_KH^*(hom_{\cB^{mod}}(\cN_1,\cN_2))=dim_KH^*(hom_{tw^\pi(\cB)}( X_1, X_2)).
\end{equation}
The latter clearly remains the same under base change. 
\end{proof}
As mentioned, we will often work with smaller fields of definition; however, smoothness does not depend on the coefficient field as long as the category is also proper. In other words:
\begin{lem}
Let $\cB$ be a proper $A_\infty$-category over a field $K\subset \Lambda$. If the base change $\cB_\Lambda:=\cB\otimes_K\Lambda $ is smooth, then so is $\cB$.
\end{lem}
\begin{proof}
The smoothness is equivalent to the perfectness of the diagonal bimodule, which is the same as the diagonal bimodule being a compact object, i.e. $RHom(\cB,\cdot)$ commutes with filtered colimits, where the $RHom$ is taken within the category of bimodules. The category of bimodules is compactly generated by Yoneda bimodules over $\cB$ (this follows from the Yoneda lemma and the definition of compact generation, \cite{kellerdg}), and it is closed under arbitrary colimits. Therefore, every bimodule itself is a filtered colimit of compact objects (\cite[Proposition 2.2]{toenvaquiemoduli}), which can also be seen explicitly as follows: consider a bimodule $\cM$ and the colimit $\cM^c$ of compact objects $\cY$ over the diagram given by pairs $(\cY,\cY\to\cM)$, where $\cY$ is a compact object. It is easy to check that the cone of $\cM^c\to\cM$ is right orthogonal to every compact object; therefore, the cone vanishes by compact generation and $\cM^c\simeq \cM$. 
Note that the compact objects are the same as perfect bimodules, i.e. direct summands of twisted complexes of Yoneda bimodules (\cite[Corollary 3.7]{kellerdg}). Therefore, it suffices to show that $RHom(\cB,\cdot)$ commutes with filtered colimits of perfect bimodules. On the other hand, every colimit can be expressed as a co-equalizer of two arbitrary direct sums, and $RHom(\cB,\cdot)$ commutes with co-equalizers (as a co-equalizer is equivalent to an equalizer up to shift in the category of bimodules). Therefore, it suffices to show that $RHom(\cB,\cdot)$ commutes with arbitrary direct sums of compact objects, i.e. if one shows that every closed morphism of bimodules
\begin{equation}\label{eq:diagtosum}
f:\cB\to \bigoplus_\eta \cY_\eta,
\end{equation}where $\{\cY_\eta\}$ is a collection of direct summands of twisted complexes of Yoneda bimodules over $\cB$, factors in cohomology through a finite direct sum. 
We know this holds after the base change to $\Lambda$, due to the smoothness of $\cB_\Lambda$. In other words, a projection $f'$ of (\ref{eq:diagtosum}) to a cofinite sub-sum $\bigoplus_{\eta'}\cY_{\eta'}$ vanishes in cohomology after the base change to $\Lambda$ (i.e. it becomes exact as we work over a field). Assume that $f'\otimes_K 1=d(h)$. Choose a basis for $\Lambda$ over $K$ that includes $1$, and write every component of $h$ in this basis. If we throw away the parts with basis elements other than $1$, we obtain a morphism $\tilde h\otimes_K 1$, whose differential is still equal to $f'\otimes_K 1$. Therefore, there is a morphism $\tilde h:\cB\to \bigoplus_{\eta'}\cY_{\eta'}$, whose differential is $f'$.
\end{proof}
Therefore, the category $\cF(M,K)$ is smooth, whenever it is defined (i.e. the $A_\infty$-structure maps lie in $K$). More generally:
\begin{cor}
If $\cB$ is smooth, proper, and $L_1,\dots,L_m$ is a set of objects of $\cB$ that split generate $\cB_\Lambda$, then they split generate $\cB$.
\end{cor}
\begin{proof}
Consider the part of the bar resolution of the diagonal bimodule of $\cB$ only involving objects $L_i$. This resolution can be filtered by finite twisted complexes, i.e. there exists an infinite sequence $\cY_0\subset\cY_1\subset\dots$ obtained by (stupid truncations) of the bar resolution. Let $f_k:\cY_k\to \cB$ denote the restriction of the resolution map to $\cY_k$. Then, $L_i$ split generate $\cB$ if and only if $f_k$ is split for some $k$, i.e.
\begin{equation}\label{eq:colimto}
	RHom(\cB,\cY_k)\xrightarrow{f_k\circ } RHom(\cB,\cB)
\end{equation}
is surjective. Indeed, if $L_i$ split generate $\cB$, then the corresponding bar complex with objects $L_i$ is a resolution, and $\cB\simeq \operatorname{colim}_k \cY_k$. As $\cB$ is compact, $\operatorname{colim}_k	RHom(\cB,\cY_k)\simeq  RHom(\cB,\cB)$; thus, \eqref{eq:colimto} hits the unit for a large $k$, i.e. $f_k$ is split (which also implies the surjectivity). On the other hand, if \eqref{eq:colimto} is surjective, it implies that $\cB$ is a direct summand of some $\cY_k$. Therefore, for any right $\cB$-module $\cN$, $\cN\simeq \cN\otimes_\cB \cB$ is a direct summand of $\cN\otimes_\cB \cY_k$, and the latter is equivalent to a complex of Yoneda modules of $L_i$ (cf. this argument with \cite[Appendix A]{generation}). 

The surjectivity holds after extending the coefficients to $\Lambda$; therefore, it holds over $K$ as well, by Lemma \ref{lem:basechange}.
\end{proof}
Going back to Fukaya categories, we have the following simple observation, that will be used regularly:
\begin{lem}\label{lem:rightconvlefteqhf}
	Assume that $\tilde L$ and $\tilde L'$ are split generated by $\{L_i\}$. Then, $H^*(h_{\tilde L'}\otimes_{\cF(M,\Lambda)} h^{\tilde L})\cong HF(\tilde L,\tilde L') $.
\end{lem}
\begin{proof}
	By Hamiltonian perturbations, one can ensure $\tilde L$ and $\tilde L'$ are transverse to each other and to all $L_i$. Then, it is possible to extend $\cF(M,\Lambda)$ by adding these. Denote this extension by $\tilde \cF(M,\Lambda)$. Standard homological algebra shows that 
	\begin{equation}
	HF(\tilde L, \tilde L';\Lambda)\cong H^*(hom_{\tilde \cF(M,\Lambda)}(\tilde L, \tilde L'))\cong H^*(h_{\tilde L'}\otimes_{\tilde \cF(M,\Lambda)} h^{\tilde L}).
	\end{equation}
	But by the split generation statement $tw^\pi(\tilde \cF(M,\Lambda))=tw^\pi(\cF(M,\Lambda))$, i.e. $\cF(M,\Lambda)$ and $\cF(M,\Lambda)$ are Morita equivalent and \begin{equation}
	H^*(h_{\tilde L'}\otimes_{\tilde \cF(M,\Lambda)} h^{\tilde L})\cong H^*(h_{\tilde L'}\otimes_{\cF(M,\Lambda)} h^{\tilde L}).
	\end{equation}
\end{proof}
\begin{rk}\label{rk:abstractextension}
Under the assumptions of the lemma, if $\cM$ is a $\cF(M,\Lambda)$-bimodule, then it is the restriction of a bimodule $\widetilde{\cM}$ over the larger category $\tilde \cF(M,\Lambda)$, and $\widetilde{\cM}(\tilde L,\tilde L')\simeq h_{\tilde L'}\otimes_{\tilde \cF(M,\Lambda)} \cM \otimes_{\tilde \cF(M,\Lambda)} h^{\tilde L}$. We will prefer to work with $h_{\tilde L'}\otimes_{\tilde \cF(M,\Lambda)} \cM \otimes_{\tilde \cF(M,\Lambda)} h^{\tilde L}$, since in general the extension of $\cM$ is abstract, and should not be confused with the concrete constructions we are going to make.	
\end{rk}
We now explain the notion of Bohr-Sommerfeld monotonicity which appeared in Assumption \ref{assumption:monotoneplus}. We borrow the definition of this notion from \cite[Remark 4.1.4]{wehrheimwoodwardquilted1}. To define this notion, we need to assume that $\omega_M$ is rational, i.e. the monotonicity constant is rational. First, let $[\omega_M]=c_1(M)$ for simplicity. Then there exists a (negative) pre-quantum bundle, i.e. a line bundle $\cL$ with a unitary connection $\nabla$ whose curvature is equal to $-2\pi i\omega_M$, and a bundle isomorphism $\cL\cong \cK^{-1}$, where the latter denotes the anti-canonical bundle. The restriction of $(\cL,\nabla)$ to a Lagrangian $\tilde L$ is flat, and $\cK^{-1}|_{\tilde L}$ carries a natural non-vanishing ``Maslov section''. We call a Lagrangian $\tilde L$ \emph{Bohr-Sommerfeld monotone} if 
\begin{itemize}
	\item 	$(\cL,\nabla)|_{\tilde L}$ has trivial monodromy, 
	\item under the induced identification $\cL|_{\tilde L}\cong \cK^{-1}|_{\tilde L}$, the Maslov section is homotopic to a non-vanishing flat section.
\end{itemize}
More generally, if $[\omega_M]=\frac{k}{l}c_1(M)$, for $k,l\in\bZ\setminus\{0\}$, there exists $(\cL,\nabla)$ with curvature $-2\pi i l\omega_M$ and isomorphism $\cL\cong (\cK^{-1})^{\otimes k}$. We call $\tilde L$ \emph{Bohr-Sommerfeld monotone} if the analogous conditions hold, i.e. $(\cL,\nabla)|_{\tilde L}$ has trivial monodromy and the identification $\cL|_{\tilde L}\cong (\cK^{-1})^{\otimes k}|_{\tilde L}$ carries the ($k^{th}$ power of) the Maslov section to a section homotopic to a flat one. Observe that this is meaningful for a negative $k$ as well. Note a slight notational difference from \cite{wehrheimwoodwardquilted1}, where $\omega_M$ is assumed to be integral (instead of rational), and $(\cL,\nabla)$ is taken to be a negative pre-quantum bundle. Assuming $\omega_M$ is rational is not a loss of generality, as this can be achieved by replacing $\omega_M$ with a positive constant multiple of itself thanks to our (negative) monotonicity assumption. 

We emphasize that the Bohr--Sommerfeld monotonicity is a condition that depends on the choice of $\cL$, $\nabla$, as well as the identification with the tensor power of the canonical bundle. In other words, different Lagrangians can satisfy this condition for different choices, yet they may still fail to satisfy this simultaneously for the same choice of data (in other words, this is a simultaneous condition for a collection of Lagrangians). We omit the dependency on $\cL$, $\nabla$, and the identification from the notation.

For us the crucial implication of this condition is the following lemma that follows from \cite[Lemma 4.1.5]{wehrheimwoodwardquilted1} and the Gromov compactness:
\begin{lem}\label{lem:finitesum}
	If each of $L_i,L$ and 	$L'$ is Bohr-Sommerfeld monotone, then there are only finitely many pseudo-holomorphic marked discs (in the $0$-dimensional moduli) with boundary on these Lagrangians and with fixed asymptotic conditions at the markings.
\end{lem}
The conclusion of the lemma also follows from the first option in Assumption \ref{assumption:monotoneplus}, by \cite[Lemma 4.1.3]{wehrheimwoodwardquilted1}.


%

\subsection{Energy spectrum and definability of Fukaya category over smaller subfields}\label{sec:definability}
As remarked, the monotonicity of the split generators $L_i$ implies that the coefficients of the $A_\infty$-structure are finite, i.e. the Fukaya category can be defined over the field $\bQ(T^\bR)$, i.e. the fraction field of 
\begin{equation}
\bQ[T^\bR]:=\Bigg\{\sum_{i=0}^Na_iT^{r_i}: a_i\in\bQ, r_i\in\bR\Bigg\}.
\end{equation}
The assumptions on $L,L'$ imply that the Yoneda modules $h_{L'}$ and $h^L$ are also defined over $\bQ(T^\bR)$. The purpose of this section is to find smaller fields of definition for the Fukaya category.

Since the boundary of marked discs used to define $\cF(M,\Lambda)$ are all on various $L_i$, the energy of such discs would take values in the image of $\omega_M:H_2(M,\bigcup_i L_i;\bZ)\to\bR$. In other words, as we construct the Fukaya category using the pearl complex for the immersed Lagrangian $\bigcup_i L_i$, we see that the energies of all discs involved lie in the finitely generated group $\omega(H_2(M,\bigcup_i L_i;\bZ))$. 
Hence, there exists a finitely generated additive subgroup $G_{pre}\subset\bR$ that contains all possible energies. In the statement of Theorem \ref{thm:mainthm}, we used two other 
Lagrangians denoted by $L$ and $L'$.
Without loss of generality, assume that the discs with boundary conditions on $L,L'$ in addition to $L_i$ also have topological energy inside $G_{pre}$.

In this case, the Fukaya category of $M$ with objects $L_i$ is defined over $\bQ(T^g:g\in G_{pre})\subset \Lambda$. As we assume that the discs with boundary conditions on $L,L'$ in addition to $L_i$ also have topological energy inside $G_{pre}$, the left/right modules corresponding to $L,L'$ are also defined over $\bQ(T^g:g\in G_{pre})$. 
\begin{rk}\label{rk:noninvariance}
$G_{pre}$ is not invariant under Hamiltonian perturbations. Hence, the invariance of the Fukaya category holds only after a base change to a larger field.
\end{rk}
We will need to evaluate at $z=T^f$ for arbitrarily small rational numbers $f\in\bQ$. Hence, we define
\begin{defn}\label{defn:G}	Let $G\subset \bR$ be the additive subgroup spanned by $G_{pre}$ and $\alpha(C)$ where $C$ is an integral $1$-cycle in $M$, and $\alpha$ is the closed $1$-form fixed in Section \ref{sec:intro} satisfying $\phi=\phi^1_\alpha$. Given prime $p$, let $G_{(p)}$ denote the set $\{\frac{g}{m}:g\in G, m\in \bZ,p\nmid m \}$. 
\end{defn}
Since $G$ is finitely generated and torsion free, one can find a basis of $G$ over $\bZ$. This basis induces a basis of $G_{(p)}$ over $\bZ_{(p)}=\{\frac{n}{m}:n,m\in \bZ, p\nmid m \}$, and $G_{(p)}$ is a free $\bZ_{(p)}$-module. 

Since $G\subset G_{(p)}$ are ordered groups, $\bQ((T^G))\subset\bQ((T^{G_{(p)}}))$ are defined in the standard way, i.e. they are Novikov series that involve only $T^g$-terms such that $g\in G$, resp. $g\in G_{(p)}$. 

Fix the following notation:
\begin{notation}
	Let $K=\bQ(T^{G_{(p)}})$ be the field of rational functions in $T^g,g\in G_{(p)}$. 
\end{notation}
The field $K$ is not finitely generated over $\bQ$ but it can be obtained by adding roots to finitely generated field $\bQ(T^G)$. Observe,
\begin{lem}\label{lem:defnhlhl}
The coefficients of the structure maps of $\cF(M,\Lambda)$ as well as $h^L,h_{L'}$ are in $K$.	
\end{lem}
As a result, the Fukaya category, as well as the modules $h^L,h_{L'}$ are defined over $K$. In other words, we have a proper $A_\infty$-category $\cF(M,K)$ over $K$ such that $\cF(M,\Lambda)=\cF(M,K)\otimes_K\Lambda$. Similarly, we have proper $A_\infty$-modules over $\cF(M,K)$ still denoted by $h^L,h_{L'}$. 
\begin{proof}
This follows from the observations on the energy of discs defining the Fukaya category and the Yoneda modules $h^L,h_{L'}$. More precisely, the group $G$ is constructed to include all possible energies, and as a result, $T^{E(u)}\in \bQ(T^G)\subset K=\bQ(T^{G_{(p)}})$. \Cref{assumption:monotoneplus} implies the coefficients of the $A_\infty$ structure maps are finite sums in $T^{E(u)}$, and the result follows.  	
\end{proof}
%
\begin{rk}\label{rk:nonisomreduction}
By Remark \ref{rk:noninvariance}, $\cF(M,K)$ is not invariant under Hamiltonian perturbations either: the continuation maps are defined only after a base change to a slightly larger field that depends on the continuation data.	In particular, if we reduce two Hamiltonian isotopic Lagrangians to the smaller field $K$, the reductions are not necessarily quasi-isomorphic. Even a single Yoneda module can have multiple reductions, and this is sufficient for our purposes.
\end{rk}
\section{Families of bimodules and symplectomorphisms}\label{sec:families}
\subsection{Family of bimodules over the Novikov field}
Recall that $\alpha$ is a fixed closed $1$-form on $M$ such that $\phi_\alpha^1=\phi$, where $\phi_\alpha^t$ denotes the flow of $X_\alpha$, which is the vector field satisfying $\omega(\cdot,X_\alpha)=\alpha$. 
One can see $\phi_\alpha^t$ as a family of symplectomorphism and up to some technicalities it defines a class of bimodules by the rule 
\begin{equation}
(L_i,L_j)\mapsto HF( L_i,\phi^t_\alpha(L_j)).
\end{equation}
Our first goal in this section is to give another description of this family inspired by family Floer cohomology and quilted Floer cohomology (\cite{quiltedstrip}, \cite{sheelthesis}) for small $t$. The notion of family we use is essentially due to Seidel \cite{flux}. He allows affine curves as the parameter space of the family. For our purposes, this is insufficient. A natural ``space'' one can work with has the ring of functions
\begin{equation}\label{eq:intervalring}
\Lambda\{z^\bR\}_{[a,b]}:=\bigg\{\sum a_r z^r:\text{, where }r\in\bR, a_r\in\Lambda \bigg\}.
\end{equation}
The series are required to satisfy the convergence condition $val_T (a_r)+r\nu \to\infty$ for all $\nu\in[a,b]$. Recall that the valuation $val_T:\Lambda\setminus\{0 \}\to\bR$ is defined by $val_T(\sum_{i=0}^\infty a_iT^{r_i})=\min\limits_{i\in\bN, a_i\neq 0} r_i$ (one can extend to $\Lambda\to\bR\cup\{\infty\}$ by defining $val_T(0)=\infty$). 

The isomorphism type of \eqref{eq:intervalring} is independent of $a$ and $b$ as long as $a<b$, and we consider it as a non-Archimedean analogue of the interval $[a,b]$. For instance, given $f\in[a,b]$, there exists an evaluation map $\Lambda\{z^\bR\}_{[a,b]}\to \Lambda$ given by $z\mapsto T^f$ (and $z^r\mapsto T^{fr}$, which does not follow automatically). More will be explained in \Cref{appendix:semicont}, but we note that we often omit $a$ and $b$ from the notation, and use $\Lambda\{z^\bR\}$ to denote $\Lambda\{z^\bR\}_{[a,b]}$ for some $a<0<b$ (therefore, ``$z=1$'' can be thought of as a point of the heuristic parameter space). 

On the other hand, by monotonicity, we will only need finite series, except for some semi-continuity statements. Therefore, until Section \ref{sec:generic}, where the monotonicity assumption is dropped, we instead consider the rings
\begin{equation}
\Lambda[z^\bR]=\{\text{finite sums }\sum a_rz^r:\text{ where }r\in\bR, a_r\in \Lambda \}, \atop
K[z^G]=\{\text{finite sums }\sum a_rz^r:\text{ where }r\in G\subset\bR, a_r\in K \}
\end{equation} 
as the heuristic ring of functions of our parameter space. Here, $G$ is the finitely generated additive subgroup of $\bR$ defined in the previous section. For our purposes, it would suffice to only allow monomials of the form $z^{\alpha(C)}$ as well, where $C$ is a $1$-cycle. Note that we will not attempt to associate a geometric spectrum to these rings, and we often refer to them as the parameter space, by abuse of terminology.
\begin{defn}\label{defn:novikovfamily}
Let $\cB$ be a smooth and proper $A_\infty$-category over $\Lambda$. A \emph{Novikov family $\fM$ of bimodules over $\B$} is an assignment of a free ($\bZ/2\bZ$)-graded $\Lambda[z^\bR]$-module, $\fM(\tilde L,\tilde L')$ to every pair of objects together with $\Lambda[z^\bR]$-linear, structure maps 
\begin{equation}\label{eq:strfam}
\cB (L'_1,L'_{0})\otimes \dots \cB (L'_m,L'_{m-1})\otimes \fM(L_n,L_m')\otimes \dots \cB (L_0,L_1)\atop \to \fM(L_0,L_0')[1-m-n]
\end{equation} satisfying the standard $A_\infty$-bimodule equations (recall that tensor products without subscripts are taken over the base field, which is $\Lambda$ in this case). A \emph{(pre-)morphism of two families $\fM$ and $\fM'$} is a collection of $\Lambda[z^\bR]$-linear, maps
\begin{equation}\label{eq:premorph}
f^{m|1|n}:\cB (L'_1,L'_{0})\otimes \dots \cB (L'_m,L'_{m-1})\otimes \fM(L_n,L_m')\otimes \dots \cB (L_0,L_1)\atop \to \fM'(L_0,L_0')[-m-n]
\end{equation} The Novikov families form a $\Lambda[z^\bR]$-linear, pre-triangulated dg category, where the differential and composition are given by standard formulas for bimodules. A \emph{morphism} of families means a closed pre-morphism. The cone of a morphism is defined as the cone of underlying bimodules, equipped with the obvious family structure (i.e. $\Lambda[z^\bR]$-linear, structure) itself. 

Similarly, if $\cB$ is smooth and proper over $K$, we define a \emph{Novikov family} to be an assignment of a free ($\bZ/2\bZ$)-graded $K[z^G]$-module $(L,L')\mapsto \fM(\tilde L,\tilde L')$ with $K[z^G]$-linear structure maps similar to \eqref{eq:strfam} satisfying the $A_\infty$-bimodule equations. The pre-morphisms, and the category of Novikov families are defined analogously, by replacing the prefix ``$\Lambda[z^\bR]$-linear'' with ``$K[z^G]$-linear''. 
%
\end{defn}
\begin{figure}\centering
	\includegraphics[height=4 cm]{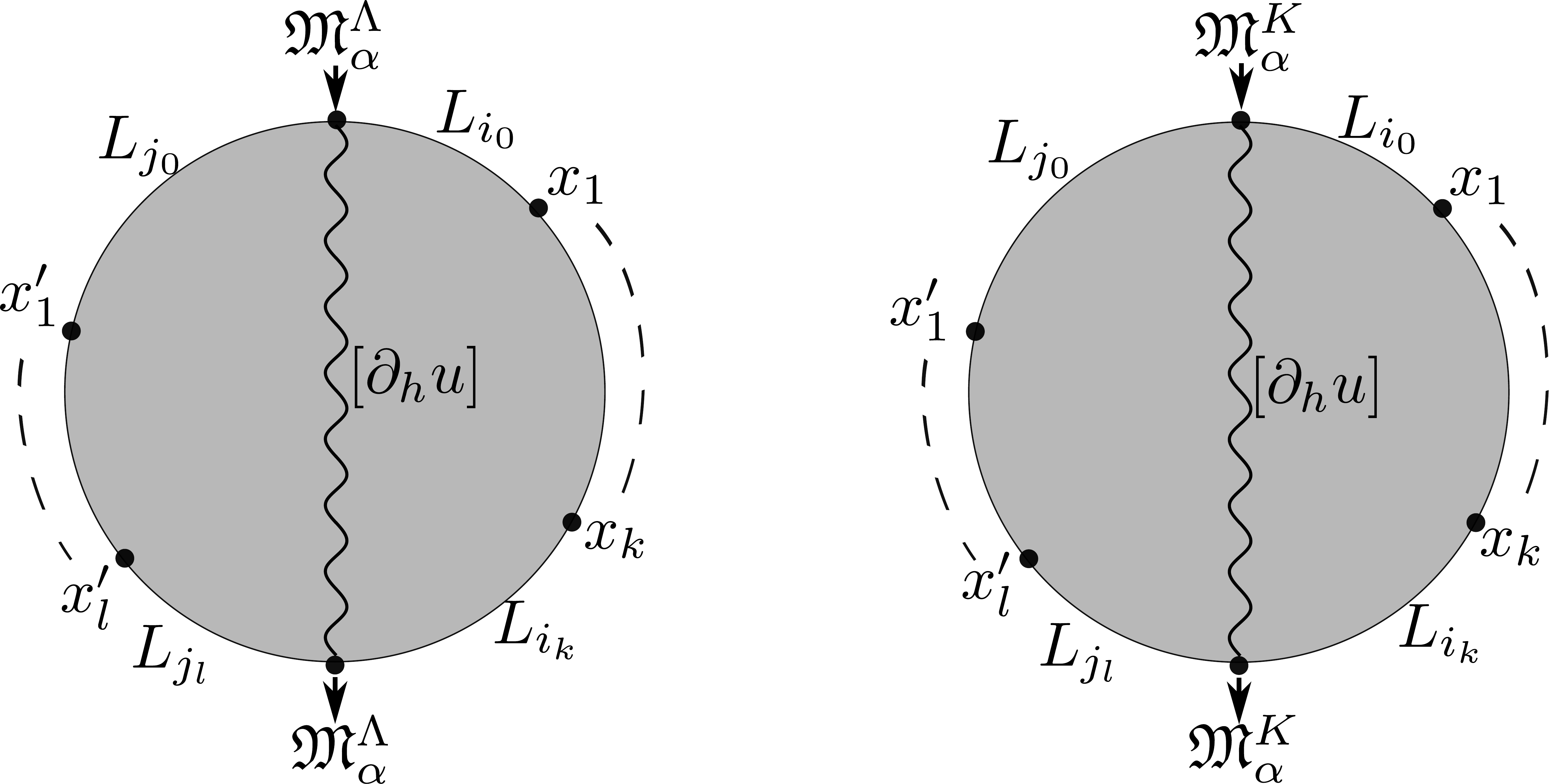}
	\caption{The counts defining $\fM_{\alpha}^\Lambda$ and $\fM^K_\alpha$}
	\label{figure:novikovfamily}
\end{figure}
We use the term Novikov family for both $K$ and $\Lambda$. Which type of family we are working with will be clear from the context.

\begin{defn}\label{defn:Kfukayafamily}
We define the family $\fM_\alpha^K$ of $\cF(M,K)$-bimodules via 
\begin{equation}\label{eq:family1chain}
(L_i,L_j)\mapsto \fM^K_\alpha (L_i,L_j)=CF(L_i,L_j;K)\otimes_K K[z^G] ,
\end{equation}
where $CF(L_i,L_j;K)=K\langle L_i\cap L_j\rangle$. To define the differential, consider the pseudo-holomorphic strips with boundary on $L_i$ and $L_j$ defining the Floer differential. Recall that we chose a base point on $M$ and relative homotopy classes of paths from this point to generators of $CF(L_i,L_j;\Lambda)$. Concatenate the following three paths: (i) the chosen path from the base point to the input chord of the strip, (ii) any path along the strip from the input to the output, and 
(iii) the reverse of the path from the base point to output. Denote this class by $[\partial_h u]$, where $u$ is the Floer strip. 
Then, define the differential for (\ref{eq:family1chain}) via the formula 
\begin{equation}
\mu^1(x)=\sum \pm T^{E(u)}z^{\alpha([\partial_h u])}.y,
\end{equation}
where $x$ and $y$ are generators of $CF(L_i,L_j;\Lambda)$ and $u$ ranges over the Floer strips with given boundary conditions, with input $x$ and output $y$. We obtain the more general structure maps of the family of bimodules by deforming the structure maps for the diagonal bimodule. Namely, the structure maps for diagonal bimodule send the tuple $(x_k,\dots,x_1|x|x_1',\dots x_l')$ to the signed sum
\begin{equation}
\sum \pm T^{E(u)}.y,
\end{equation}
where the sum ranges over the discs with input $x_l'\dots x,\dots, x_k$ and with output $y$. We define the structure maps for the family $\fM_\alpha^K$ (in the sense of \Cref{defn:novikovfamily}) via the formula
\begin{equation}\label{eq:structurenovikov}
\sum \pm T^{E(u)}z^{\alpha([\partial_h u])}.y ,
\end{equation}
where $[\partial_h u]$ denotes the class obtained by concatenating (i) the chosen path from the base point to $x$, (ii) $u\circ \gamma$ where $\gamma$ is a path in the marked disc from the marked point corresponding to $x$ to the marked point corresponding to $y$, and (iii) the reverse of the chosen path from the base point to $y$. In other words, $[\partial_h u]$ is obtained by concatenating the wavy line in Figure \ref{figure:novikovfamily} with paths to the base point. 
\end{defn}
\begin{defn}
Let $\fM_{\alpha}^\Lambda$ denote the $\Lambda[z^\bR]$-linear Novikov family of $\cF(M,\Lambda)$-bimodules that is obtained by replacing $K$ with $\Lambda$ in Definition \ref{defn:Kfukayafamily}. Equivalently, this family can be obtained by extension of the coefficients of $\fM^K_{\alpha}$ along the inclusion map $K\to\Lambda$.
\end{defn}
\begin{note}
Let $\tilde L$ and $\tilde L'$ be two Lagrangians that satisfy the conditions on $L$ and $L'$ in Assumption \ref{assumption:monotoneplus}. As mentioned in Remark \ref{rk:abstractextension}, one can abstractly extend the Fukaya category  to include $\tilde L$ and $\tilde L'$ and the bimodule $\fM^K_\alpha$ to $\widetilde{\fM^K_\alpha}$. However, this extension is by abstract means, whereas the notation $\widetilde{\fM_\alpha^K}(\tilde L,\tilde L')$ suggests the same concrete definition as (\ref{eq:family1chain}). Therefore, we will use the complex 
\begin{equation}
h_{\tilde L'}\famotimes_{\cF(M,K)}  \fM_\alpha^K  \famotimes_{\cF(M,K)} h^{\tilde L}
\end{equation}
instead, which is well-defined whenever the modules $h_{\tilde L'}$ and $h^{\tilde L}$ are defined over $\cF(M,K)$.
\end{note}
Consider $\faM{f}{\Lambda}$ and $\faM{f}{K}$, i.e. the base change of the respective family under the map $\Lambda[z^\bR]\to\Lambda$, resp. $K[z^G]\to K$ that sends $z^r$ to $T^{fr}$. The latter makes sense only for $f\in G_{(p)}\subset \bR$. These are bimodules over $\cF(M,\Lambda)$, resp. $\cF(M,K)$.

Later we will show $\faM{-f}{K}$, together with the left module corresponding to $L$ and right module corresponding to $L'$ can be used to recover the groups $HF(\phi^f_\alpha L,L')$ for some $f$, via the formula
\begin{equation}
HF(\phi^f_\alpha (L),L')\cong H^*(h_{L'}\otimes_{\cF(M,K)} \faM{-f}{K} \otimes_{\cF(M,K)} h^L)\otimes_K \Lambda,
\end{equation}
where $h_{L'}$ denotes the right Yoneda module corresponding to $L'$ and $h^L$ denotes the left Yoneda module corresponding to $L$. 

We also note:
\begin{lem}\label{lem:grouplikenovikov}
For $f,f'\in\bR$ such that $|f|,|f'|$ are small, $\faM{f+f'}{\Lambda}\simeq \faM{f}{\Lambda}\otimes_{\cF(M,\Lambda)} \faM{f'}{\Lambda}$. The same statement holds for $\Lambda$ replaced by $K$ if $f,f'\in \bZ_{(p)}$.
\end{lem}
We prove \Cref{lem:grouplikenovikov} in \Cref{appendix:semicont}. To this end, we write a map
\begin{equation}\label{eq:novconvto}
\faM{f}{\Lambda}\otimes_{\cF(M,\Lambda)} \faM{f'}{\Lambda}\to \faM{f+f'}{\Lambda}
\end{equation}
that varies continuously in $f$ and $f'$, and that restricts to a quasi-isomorphism at $f=f'=0$. 

We also prove a $p$-adic version of \Cref{prop:grouplikepadic} of \Cref{lem:grouplikenovikov}. 
\begin{rk}
In \Cref{cor:groupovernovikov}, we give another proof of \Cref{lem:grouplikenovikov} for $f,f'\in\bZ_{(p)}$ with a small $p$-adic absolute value.
\end{rk}
\subsection{p-adic arcs in Floer cohomology}\label{subsec:padicarcs}
Let $p>2$ be a prime number. We start by constructing an embedding $\bQ(T^G)\hookrightarrow \bQ_p$ such that elements of the form $T^g$ map to elements of $1+p\bZ_p$. More precisely, fix an integral basis $g_1,\dots ,g_k$ of the group $G$. Let $\kappa_1,\dots,\kappa_k\in\bZ_p$ be algebraically independent over $\bQ$. Define a map $g_i\mapsto 1+p\kappa_i$ from $\bQ(T^G)\to \bQ_p$. This is well defined since $T^{g_i}$ are algebraically independent. We will extend it to a map $\bQ(T^{G_{(p)}})\to\bQ_p$ using Definition \ref{defn:padicinterp}. Recall that $\bQ_p\langle t\rangle=\{\sum_n a_nt^n:n\in\bN, a_n\in\bQ_p, val_p(a_n)\to\infty \}$ is the Tate algebra over $\bQ_p$ with one variable, and it can be thought of as the set of analytic functions on the $p$-adic unit disc $\bZ_p$:
\begin{defn}\cite{padicinterp,dynmordellforcoherent}\label{defn:padicinterp}
Let $v=1+\nu\in1+p\bZ_p$. Define $v^t\in \bQ_p\langle t\rangle$ to be the function \begin{equation}\label{eq:padicinterp}
(1+(v-1))^t=(1+\nu)^t:=\sum_{i=0}^{\infty} {t\choose i}\nu^i.
\end{equation}
\end{defn}
Recall that ${t\choose i}$ denotes the polynomial $\frac{t(t-1)\dots(t-i+1)}{i!}$. 
The convergence of (\ref{eq:padicinterp}) on $\bZ_p$ is clear (cf. \cite[Proposition 2.1]{dynmordellforcoherent}). Let us list some properties, mainly inspired by \cite{dynmordellforcoherent}:
\begin{enumerate}
	\item\label{propexp1} $v^t$ is the $n^{th}$ power of $v$ when $t=n\in\bN$,
	\item\label{propexp2} $v^{t+t'}=v^tv^{t'}\in\bQ_p\langle t,t'\rangle$,
	\item\label{propexp3} $(v_1v_2)^t=v_1^tv_2^t$.
\end{enumerate}
\begin{proof}
The claim (\ref{propexp1}) follows from the binomial theorem. To see (\ref{propexp3}), check it first on $\bN\subset\bZ_p$ using (\ref{propexp1}). A functional equation that holds on a dense (or just infinite) subset of $\bZ_p$ holds over $\bQ_p\langle t\rangle$ by Strassman's theorem. More precisely, Strassman's theorem states that a non-zero function $f(t)\in\bQ_p\langle t\rangle$ vanishes only at finitely many elements of $\bQ_p$ (\cite{strassmannuber}, \cite[p. 62, Theorem 4.1]{casselslocal}, \cite[Theorem 3.38]{padicanalysiskatok}). We apply this statement to $f(t)=(v_1v_2)^t-v_1^tv_2^t$.

Similarly, to see (\ref{propexp2}), check it first on $\bN\times\bN\subset\bZ_p\times\bZ_p$ using (\ref{propexp1}), and conclude by the density that the equation holds on $\bZ_p\times\bZ_p$. To obtain (\ref{propexp2}), we apply Strassman's theorem iteratively. More precisely, let $f(t,t')=v^{t+t'}-v^tv^{t'}=\sum g_k(t')t^k$. We noted that $f(a,a')=0$ for $a,a'\in\bZ_p$. Therefore, for any fixed $a'$, $\sum g_k(a')t^k\in\bQ_p\langle t\rangle$ is $0$ at infinitely many $t$, and Strassman's theorem implies that $g_k(a')=0$ for all $k$. As $a'$ was arbitrary, a second application of Strassman's theorem would let us conclude that $g_k(t')=0$; therefore, $f(t,t')=0$. 
%
\end{proof}
\begin{lem}\label{lem:fieldmapqp}
The map $\bQ(T^{G})\to \bQ_p$ that sends $T^{g_i}$ to $1+p\kappa_i$ extends uniquely to a field homomorphism $\kappa:\bQ(T^{G_{(p)}})\to \bQ_p$ such that for a given $g_i$, and given $n$ such that $p\nmid n$, $\kappa(T^{g_i/n})=(1+p\kappa_i)^{1/n}$, i.e. the specialization of $(1+p\kappa_i)^{t}$ to $t=1/n$. 
\end{lem}
Observe that when $p\nmid n$, one can specialize $(1+p\kappa_i)^{t}\in \bQ_p\langle t\rangle$ to $t=1/n$ as $1/n\in\bZ_p$.
\begin{proof}
Observe that, for a given $N\in \bN$ satisfying $p\nmid N$, the elements $x_i:=(1+p\kappa_i)^{1/N}$, $i=1,\dots, k$ are algebraically independent. To see this, first note that $x_i^N=1+p\kappa_i$ by the property \eqref{propexp1}; thus, the elements $x_i^N$ are algebraically independent as $\kappa_i$ are independent. Assume that the elements $x_i$ satisfy an algebraic relation, and without loss of generality, that $x_k$ is algebraic over $\bQ(x_1,\dots, x_{k-1})$. Clearly, $\bQ(x_1,\dots, x_{k-1})$ is an algebraic extension of $\bQ(x_1^N,\dots, x_{k-1}^N)$. Therefore, $x_k$ (and thus $x_k^N$) is algebraic over $\bQ(x_1^N,\dots, x_{k-1}^N)$, which would contradict the independence of the elements $x_i^N$. This proves the independence of the elements $x_i$.

As a result, there exists a unique field homomorphism $\bQ(T^{g_1/N},\dots ,T^{g_k/N})\to \bQ_p$ that sends $T^{g_i/N}$ to $x_i=(1+p\kappa_i)^{1/N}$. Therefore, $T^{-g_i/N}$ maps to the inverse of $(1+p\kappa_i)^{1/N}$, which is equal to $(1+p\kappa_i)^{-1/N}$ ($(1+p\kappa_i)^{1/N}(1+p\kappa_i)^{-1/N}=(1+p\kappa_i)^0=1$ by \eqref{propexp2}).

Let $n\mid N$ and assume that $d:=N/n>0$, then $T^{g_i/n}=T^{g_i/N}\dots T^{g_i/N}$ ($d$ times) maps to 
\begin{equation}
	(1+p\kappa_i)^{1/N}\dots (1+p\kappa_i)^{1/N}\text{ (}d\text{ times)},
\end{equation}
which is equal to $(1+p\kappa_i)^{d/N}=(1+p\kappa_i)^{1/n}$ by \eqref{propexp2}. Similarly, $T^{-g_i/n}=T^{-g_i/N}\dots T^{-g_i/N}$ maps to $(1+p\kappa_i)^{-1/n}=(1+p\kappa_i)^{-1/N}\dots (1+p\kappa_i)^{-1/N}$. 

This proves the unique extension to $\bQ(T^{\frac{1}{N}G})$ satisfying the desired property, where $\frac{1}{N}G:=\{r:Nr\in G \}$. As $\bQ(T^{G_{(p)}})=\bigcup_{p\nmid N}\bQ(T^{\frac{1}{N}G})$, we obtain a unique extension to $\bQ(T^{G_{(p)}})$.
\end{proof}
\begin{notation}
We will denote $\kappa(T^a)$ also by $T_\kappa^a$, where $a\in G_{(p)}$. We denote the category obtained from $\cF(M,K)$ by extending the coefficients through $K=\bQ(T^{G_{(p)}})\to \bQ_p$ by $\cF(M,\bQ_p)$. In other words, $\cF(M,\bQ_p):= \cF(M,K)\otimes_K\bQ_p $.
\end{notation}
%
%
The following is the $p$-adic analogue of Definition \ref{defn:novikovfamily}:
\begin{defn}\label{defn:padicfamily}
For a given smooth and proper $A_\infty$-category $\cB$  over $\bQ_p$, a \emph{$p$-adic family $\fM$ of bimodules over $\B$} is an assignment of a free ($\bZ/2\bZ$)-graded $\bQ_p\langle t\rangle$-module $\fM(L,L')$ to every pair of objects together with $\bQ_p\langle t\rangle$-linear structure maps 
\begin{align}
\cB (L'_1,L'_{0})\otimes \dots \cB (L'_m,L'_{m-1})\otimes \fM(L_n,L_m')\otimes \dots \cB (L_0,L_1)\\ \to \fM(L_0,L_0')[1-m-n]
\end{align} satisfying the standard bimodule equations. A \emph{(pre)-morphism} of two families $\fM$ and $\fM'$ is a collection of $\bQ_p\langle t\rangle$-linear maps
\begin{align}
f^{m|1|n}:\cB (L'_1,L'_{0})\otimes \dots \cB (L'_m,L'_{m-1})\otimes \fM(L_n,L_m')\otimes \dots \cB (L_0,L_1)\\ \to \fM'(L_0,L_0')[-m-n].
\end{align} As before, the $p$-adic families form a $\bQ_p\langle t\rangle$-linear pre-triangulated dg category, where the differential and composition are given by standard formulas for bimodules, and a \emph{morphism} of families means a closed pre-morphism. The cone of a morphism is defined as the cone of the underlying map of bimodules, equipped with the natural $\bQ_p\langle t\rangle$-linear structure.
\end{defn}
Definition \ref{defn:padicfamily} easily generalizes to other non-Archimedean fields extending $\bQ_p$ as well as to the Tate algebras with several variables $\bQ_p\langle t_1,\dots t_n\rangle$. Let $\fM$ be a family and $q\in \bQ_p$ be an element of the $p$-adic unit disc $\bD$. One can consider $q$ as a continuous ring homomorphism $\bQ_p\langle t\rangle\to\bQ_p$ such that $f(t)\mapsto f(q)$. Note that these ring homomorphisms are in correspondence with elements of $\bZ_p$. In other words, given $a\in\bZ_p$, one can define such a continuous ring homomorphism sending $\sum_{k=0}^\infty a_kt^k$ to $\sum_{k=0}^\infty a_ka^k\in\bQ_p$. Conversely, any continuous ring homomorphism is determined in this way by the image of $t$ (by continuity), and for $\sum_{k=0}^\infty a_ka^k$ to converge (i.e. $val_p(a_k)+kval_p(a)\to\infty$) for any $\sum_{k=0}^\infty a_kt^k\in \bQ_p\langle t\rangle$ (i.e. whenever $val_p(a_k)\to \infty $), one needs $val_p(a)\geq 0$. 

Define the restriction $\fM|_{t=q}$ as $\fM\otimes_{\bQ_p\langle t\rangle}\bQ_p$. This is an $A_\infty$-bimodule over $\cB$.
\begin{exmp}\label{exmp:constfams}
For any bimodule $\cM$ over $\cB$, one can define a p-adic family by $\fM(L,L')=\cM(L,L')\otimes_{\bQ_p}{\bQ_p\langle t\rangle}$ with the structure maps obtained by base change. This type of family has the same restrictions at every point. In particular, one can let $\cM$ to be a Yoneda bimodule $h^{L}\boxtimes h_{L'}$ (i.e. the exterior tensor product of left and right Yoneda modules, its structure maps are defined in \cite[(2.83),(2.84)]{sheelthesis}). We call such a family \emph{a constant family of Yoneda bimodules}. Given projective $\bQ_p\langle t\rangle$-module $P$ of finite rank, one can also define a corresponding \emph{locally constant family} by $\fM(L,L')=\cM(L,L')\otimes_{\bQ_p} P$. By Quillen--Suslin theorem for Tate algebras (\cite[Theorem 6.7]{kedlayaoverconvergent}), every finitely generated projective module is free; hence, these two notions coincide.  
\end{exmp}
One can define the convolution of two families. First, recall the convolution of bimodules over $\cB$:
\begin{defn}\label{defn:ordinaryconv}
Let $\cM_1$ and $\cM_2$ be two bimodules over $\cB$. Then, $\cM_1\otimes_\cB \cM_2$	is the bimodule defined by 
\begin{equation}\label{eq:convordin}
(L,L')\longmapsto \bigoplus \cM_1(L_k,L')\otimes \cB(L_{k-1},L_k)\otimes \dots \otimes \cB(L_1,L_2) \otimes \cM_2(L,L_1)[k],
\end{equation}
where the direct sum is over all ordered sets $(L_1,\dots,L_k)$ for all $k\in\bZ_{\geq 0}$. The differential is given by 
\begin{align}
(m_1\otimes b_1\otimes\dots \otimes b_f\otimes m_2)\mapsto \sum\pm \mu_{\cM_1} (m_1\otimes b_1\otimes\dots) \otimes\dots\otimes m_2+\\ \sum\pm m_1\otimes \dots \otimes \mu_{\cM_2} (\dots\otimes m_2)+\sum\pm m_1\otimes \dots \otimes \mu_\cB (\dots)\otimes \dots \otimes m_2\nonumber,
\end{align}
and other structure maps are defined similarly. 
\end{defn}
As we briefly explained in the introduction, if $\fM_1$ is a $p$-adic family and $\cM_2$ is an ordinary bimodule, $\fM_1\otimes_\cB \cM_2$ carries a natural $\bQ_p\langle t\rangle$-linear structure, i.e. it is also a family. To emphasize that it inherits a family structure from $\fM_1$, we denote the convolution by $\fM_1\famotimes_\cB \cM_2$, even though it is the same as $\fM_1\otimes_\cB \cM_2$ as a $\cB$-bimodule. For different orders of tensor products, and iterated tensor products, we use analogous notation (when only one of the bimodules carries a family structure). Similarly,
\begin{defn}\label{defn:reltensor}
Given $p$-adic families $\fM_1$ and $\fM_2$, one can endow $\fM_1\otimes_\cB \fM_2$ with the structure of a family over $\bQ_p\langle t_1,t_2\rangle $. One obtains a family over $\bQ_p\langle t\rangle $ via base change along the (co)diagonal map $\bQ_p\langle t_1,t_2\rangle\to \bQ_p\langle t\rangle, t_1,t_2\mapsto t  $. We denote this family by $\fM_1\relotimes_\cB \fM_2$.
\end{defn}
The family $\fM_1\relotimes_\cB \fM_2$ can also be constructed by performing the construction in Definition \ref{defn:ordinaryconv} $\bQ_p\langle t\rangle$-linearly. In other words, one can see $\fM_1$ and $\fM_2$ as bimodules over $\cB_{\bQ_p\langle t\rangle}:=\cB\otimes \bQ_p\langle t\rangle$, and replace the tensor products in \eqref{eq:convordin} with tensor products over $\bQ_p\langle t\rangle$. In this perspective, $\fM_1\relotimes_\cB \fM_2$ can be seen as the fiberwise convolution of two families, i.e. as the tensor product relative to the base $\bQ_p\langle t\rangle$; hence, the notation.
%
\begin{exmp}\label{exmp:constconvolution}
Let $\fM_1$ and $\fM_2$ be two constant families associated to bimodules $\cM_1$ and $\cM_2$ over $\cB$. Then, $\fM_1\relotimes_\cB \fM_2$ is the constant family associated to $\cM_1\otimes_\cB \cM_2$. In particular, if $\cM_1=h^{L_1}\boxtimes h_{L'_1}$ and $\cM_2=h^{L_2}\boxtimes h_{L'_2}$, then $\fM_1\relotimes_\cB \fM_2$ is the constant family associated to $\cB(L_2,L_1')\otimes (h^{L_1}\boxtimes h_{L'_2})$.
\end{exmp}
By Morita theory, a family of bimodules can be thought of as a family of endomorphisms of the category. Therefore, if the parameter space of the family is a group, one can study the ``actions of this group on the category''. Observe $\bQ_p\langle t\rangle$ is the ring of functions of a group, and is itself a Hopf algebra over $\bQ_p$ with comultiplication  given by $\Delta:t\mapsto t\otimes 1+1\otimes t$ (the counit is given by $e:t\mapsto 0$ and the antipodal map is given by $t\mapsto -t$). 

Let $\pi_i:\bQ_p\langle t\rangle \to \bQ_p\langle t_1,t_2\rangle$ denote the map $t\mapsto t_i$ for $i=1,2$. Given $p$-adic family $\fM$, one can extend the coefficients along $\pi_1$, $\pi_2$ and $\Delta$ to define three $2$-parameter $p$-adic families of bimodules denoted by $\pi_1^*\fM$, $\pi_2^*\fM$ and $\Delta^*\fM$
(we identify $\bQ_p\langle t_1,t_2\rangle$ with a suitable completion of $\bQ_p\langle t\rangle\otimes \bQ_p\langle t\rangle$ such that $t\otimes 1=t_1,1\otimes t=t_2$). The relative tensor product defined in \Cref{defn:reltensor} can also be performed relative to the base $\bQ_p\langle t_1,t_2\rangle$; thus, $\pi_1^*\fM\relotimes_\cB\pi_2^*\fM$ is a well-defined $2$-parameter $p$-adic family.
\begin{defn}\label{defn:grouplike}
A $p$-adic family $\fM$ of bimodules over $\cB$ is called \emph{group-like} if $\Delta^*\fM \simeq \pi_1^*\fM\relotimes_\cB\pi_2^*\fM$ and if the restriction to counit $\fM|_{t=0}$ is quasi-isomorphic to diagonal bimodule.
\end{defn}
\begin{figure}\centering
	\includegraphics[height=4 cm]{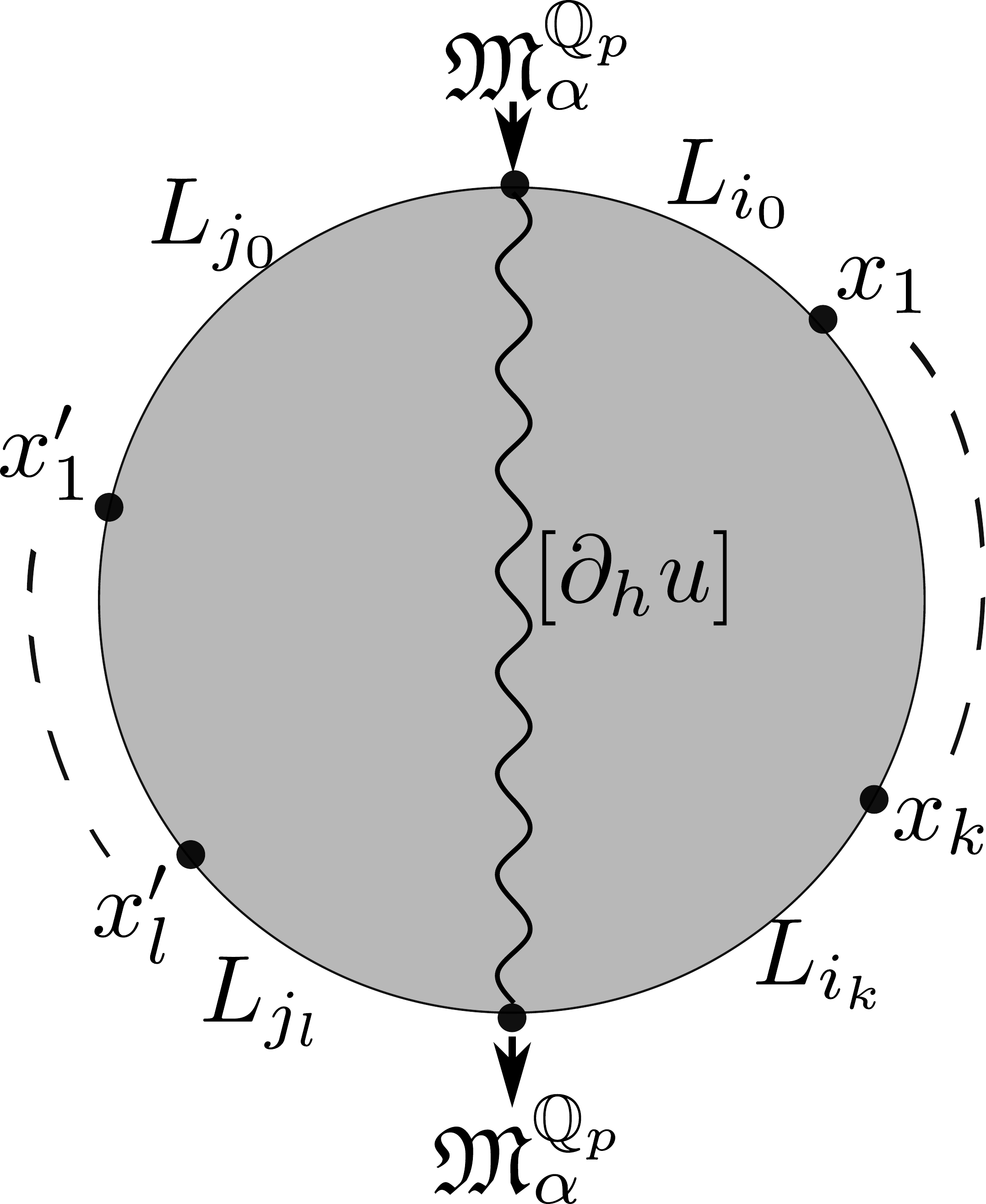}
	\caption{The counts defining $\fM_{\alpha}^{\bQ_p}$}
	\label{figure:padicfamily}
\end{figure}
We observe:
\begin{lem}\label{lem:grouplike}
If $\fM$ is group-like, and $f_1,f_2\in\bZ_p$, then $\fM|_{t=f_1+f_2}\simeq \fM|_{t=f_1}\otimes_\cB \fM|_{t=f_2}$.	
\end{lem}
\begin{proof}
The left-hand side is the restriction of $\Delta^*\fM$ to $(t_1,t_2)=(f_1,f_2)$, and the right-hand side is the restriction of $\pi_1^*\fM\relotimes_\cB\pi_2^*\fM$ to the same point. The conclusion follows from \Cref{defn:grouplike}. 
\end{proof}
\begin{rk}\label{rk:generalcoherent}
Given a group-like family $\fM$, one obtains two morphisms 
\begin{equation}
	\fM|_{t=f_1}\otimes_\cB \fM|_{t=f_2}\otimes_\cB \fM|_{t=f_3}\to \fM|_{t=f_1+f_2+f_3}
\end{equation}
by different orders of composition, and similarly, there are two a priori different compositions of the corresponding three-parameter families. Our definition does not imply that these two are homotopic, and in this sense, it should be considered as a ``weak action'' of the $p$-adic unit disc. In this sense, it should be compared to a homotopy associative action of a topological group, i.e. an action in the homotopy category. This is sufficient for our purposes, although the families we consider also satisfy such coherence conditions. 
\end{rk}

Next, we construct an explicit group-like $p$-adic family $\fM_\alpha^{\bQ_p}$ of bimodules over $\cF(M,\bQ_p)$. Define
\begin{equation}
(L_i,L_j)\mapsto\fM_\alpha^{\bQ_p}(L_i,L_j)= CF(L_i,L_j;\bQ_p)\otimes \bQ_p\langle t\rangle.
\end{equation}
The structure maps are defined via the formula
\begin{equation}\label{eq:padicstructuremaps}
(x_1,\dots, x_k|x|x_1',\dots ,x_l')\mapsto\sum\pm T_\kappa^{E(u)} T_\kappa^{t\alpha([\partial_h u])}.y,
\end{equation}
where the sum ranges over the marked discs with input $x_l'\dots, x,x_k,\dots,x_1$ and output $y$. The class $[\partial_h u]\in H_1(M;\bZ)$ is defined as before (Figure \ref{figure:padicfamily}), and $\alpha([\partial_h u])\in G$ by definition of $G$; therefore, $T_\kappa^{\alpha([\partial_h u])}=\kappa(T^{\alpha([\partial_h u])})\in1+p\bZ_p$ is defined. Let $T_\kappa^{t\alpha([\partial_h u])}\in \bQ_p\langle t\rangle$ be its ``$t^{th}$ power'' as in Definition \ref{defn:padicinterp}. The sum is finite, and 
the bimodule equation is satisfied (a picture explanation is given in Figure \ref{figure:compactificationpadicfamily}). It is immediate that the restriction to $t=0$ is isomorphic to the diagonal bimodule of $\cF(M,\bQ_p)$. 
\begin{figure}\centering
	\includegraphics[height=5.5 cm]{compactificationpadicfamily.pdf}
	\caption{Some degenerations of discs as in Figure \ref{figure:padicfamily}}
	\label{figure:compactificationpadicfamily}
\end{figure}

To prove that this family is group-like, our next task is to write a closed morphism of families 
\begin{equation}\label{eq:convolutiontodiag}
g_{\fM_\alpha^{\bQ_p}}:\pi_1^*\fM_\alpha^{\bQ_p}\relotimes_{\cF(M,\bQ_p)} \pi_2^*\fM_\alpha^{\bQ_p}\to \Delta^*\fM_\alpha^{\bQ_p}
\end{equation}
satisfying the following: the restriction of (\ref{eq:convolutiontodiag}) to $t_1=t_2=0$ is the standard map 
\begin{equation}
	\cF(M,\bQ_p)\otimes_{\cF(M,\bQ_p)}\cF(M,\bQ_p)\to \cF(M,\bQ_p)
\end{equation} of bimodules, which is a quasi-isomorphism. 

Given an $A_\infty$-category $\cB$, it is a general result that $\cB\otimes_{\cB}\cB\simeq\cB$ (\cite[Proposition 2.2]{sheelthesis}). The bimodule quasi-isomorphism from left-hand side to the right is given by
\begin{equation}\label{eq:convolutiontodiagformula}
g^{k|1|l}:(x_1,\dots,x_k|x\otimes b_1\otimes\dots \otimes b_f\otimes x'|x_1',\dots x_l')\mapsto\atop \pm\mu_\cB (x_1,\dots,x_k,x, b_1,\dots b_f, x',x_1',\dots x_l').
\end{equation}
Here, $x\otimes b_1\otimes\dots \otimes b_f\otimes x'$ is an element of $\cB\otimes_\cB\cB$. When $\cB$ is the Fukaya category, this map is geometrically given by the count of marked discs as usual. To deform it, we define the following cohomology classes: given a pseudo-holomorphic disc $u$ with output $y$ and with input given by generators $x_1,\dots,x_k$, $x$, $b_1\dots , b_f$,  $x'$, $x_1',\dots x_l'$ (in counter-clockwise direction after output), define $[\partial_1\alpha]\in H_1(M;\bZ) $ to be the path obtained by concatenating the fixed path from the base point of $M$ to the generator $x$, the image under $u$ of a path from the input marked point for $x$ to the output marked point, and the reverse of the path from the base point to $y$. We think of this class as the portion of the boundary of $u$ from $x$ to $y$. Similarly define $[\partial_2 u]\in H_1(M;\bZ)$ by replacing $x$ with $x'$. The class $[\partial_2 u]$ can be thought of as the portion of boundary from $x'$ to $y$. In other words, the paths $[\partial_1 u]$ and $[\partial_2 u]$ are obtained by concatenating the $u$-image of the respective wavy line in Figure \ref{figure:groupquasi} with chosen paths from the base point to the generator. To define the map (\ref{eq:convolutiontodiag}), fix inputs $x_1,\dots,x_k,x, b_1\dots , b_f, x',x_1',\dots x_l'$ as above. The coefficient of $y$ under the map (\ref{eq:convolutiontodiag}) is given by 
\begin{equation}\label{eq:padicmapformula}
\sum\pm T_\kappa^{E(u)} T_\kappa^{t_1\alpha([\partial_1 u])} T_\kappa^{t_2\alpha([\partial_2 u])}.y,
\end{equation}
where $u$ ranges over the pseudo-holomorphic discs with given input and output (Figure \ref{figure:groupquasi}).  

Recall that $T_\kappa^{\alpha([\partial_1 u])}:=\kappa(T^{\alpha([\partial_1 u])})\in 1+p\bZ_p$. As before $T_\kappa^{t_1\alpha([\partial_1 u])}\in \bQ_p\langle t_1\rangle\subset \bQ_p\langle t_1,t_2\rangle$ is its $t_1^{th}$-power as in Definition \ref{defn:padicinterp}. The element $T_\kappa^{t_2\alpha([\partial_2 u])}\in \bQ_p\langle t_2\rangle\subset \bQ_p\langle t_1,t_2\rangle$ is defined similarly. 
This defines a map of bimodules (\ref{eq:convolutiontodiag}) (Figure \ref{figure:compactificationgroupquasi} provides a picture explanation). Moreover, the map \eqref{eq:convolutiontodiag} of families restricts to the standard quasi-isomorphism of diagonal bimodules at $t_1=t_2=0$ defined by (\ref{eq:convolutiontodiagformula}).
\begin{figure}\centering
	\includegraphics[height=4 cm]{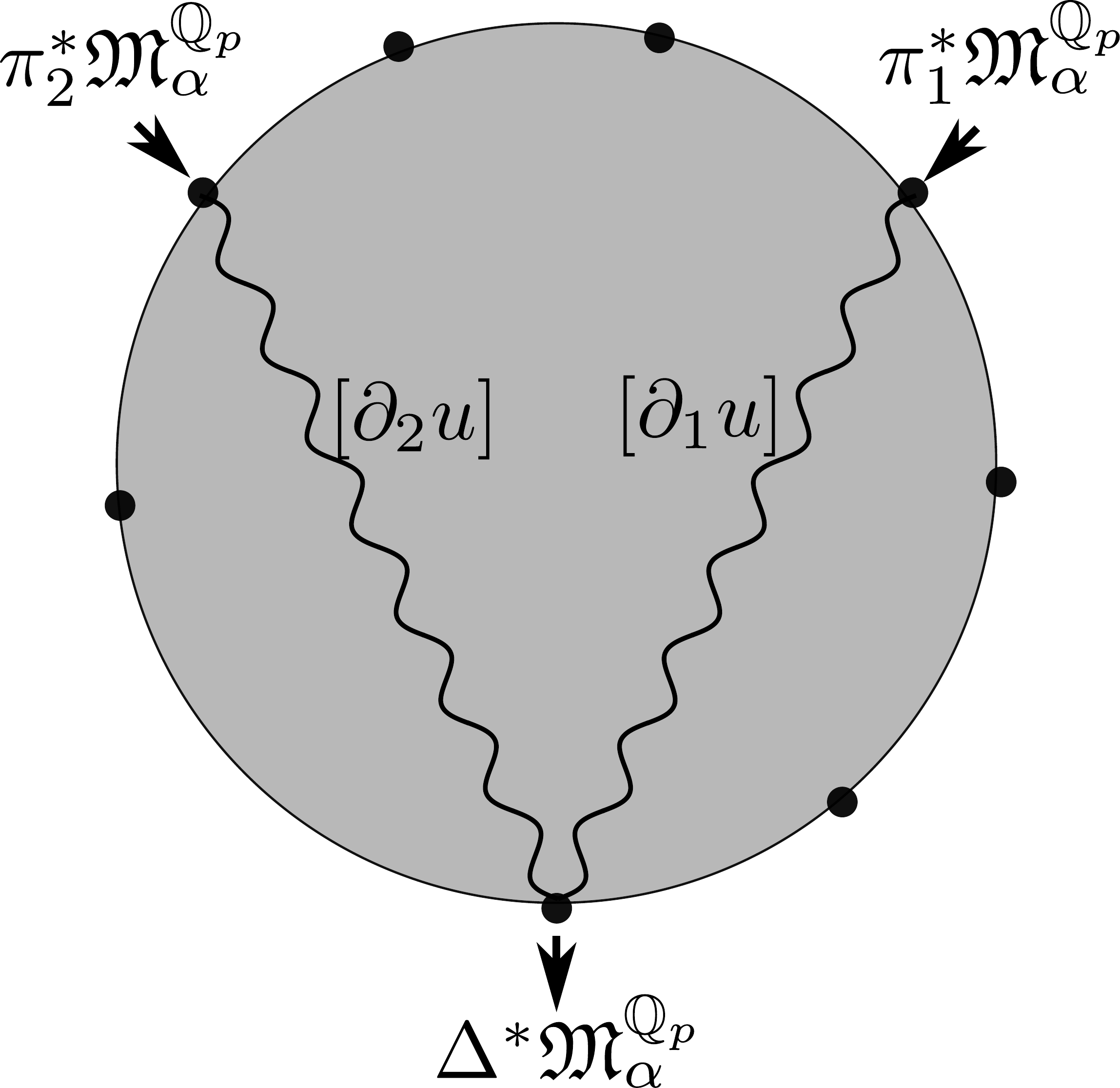}
	\caption{The counts defining (\ref{eq:convolutiontodiag}). The input $x$, resp. $x'$ is associated to the marked point labeled as $\pi_1^*\fM_{\alpha}^{\bQ_p}$, resp. $\pi_2^*\fM_{\alpha}^{\bQ_p}$, and the output $y$ is associated to $\Delta^*\fM_{\alpha}^{\bQ_p}$.}
	\label{figure:groupquasi}
\end{figure}
\begin{figure}\centering
	\includegraphics[height=12 cm]{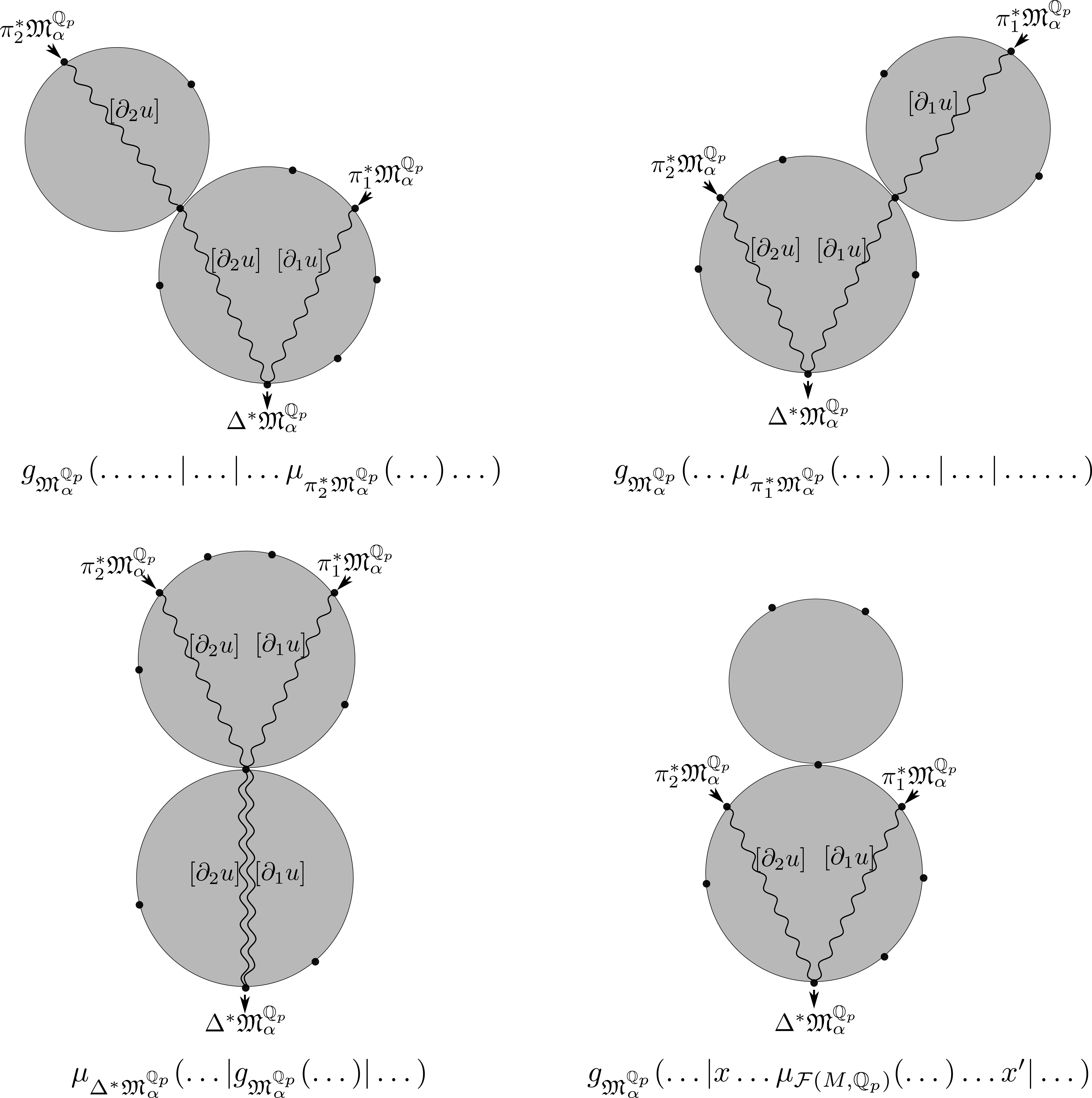}
	\caption{Some degenerations of moduli of discs as in Figure \ref{figure:groupquasi}}
	\label{figure:compactificationgroupquasi}
\end{figure}

Our next task is to prove (\ref{eq:convolutiontodiag}) is a quasi-isomorphism. This relies on a semi-continuity argument for which we need a properness result for the domain of (\ref{eq:convolutiontodiag}), i.e. we need to show that this family is a cohomologically finitely generated $\bQ_p\langle t_1,t_2\rangle$-module at every pair of objects $(L,L')$. We need some technical preparation for this:
\begin{defn}\label{defn:perfectfamily}
A family $\fM$ is called \emph{perfect} if it is quasi-isomorphic to a direct summand of a twisted complex of constant families of Yoneda bimodules in the category of families. 
A family $\fM$ is called \emph{proper} if the cohomology of $\fM(L,L')$ is finitely generated over $\bQ_p\langle t\rangle$ for all $L,L'$.
\end{defn}
\begin{rk}\label{rk:locconst}
Note that our definition of perfect families is a priori more restrictive than the notion of a family of twisted complexes in \cite{flux}. More precisely, in loc.\ cit.\ the author allows \emph{locally constant families of Yoneda bimodules} as defined in \Cref{exmp:constfams}. As we have remarked, in our specific case, this notion is not more general for Tate algebras thanks to the Quillen--Suslin theorem. However, the correct notion of a perfect family over more general rings of functions should allow the locally constant ones. 
\end{rk}
Clearly, 
if $\cB$ is a proper category, 
the perfectness of a family implies its properness. We will see that proper implies perfect for a smooth, proper $A_\infty$-category $\cB$. For simplicity, we start with the following:
\begin{lem}\label{lem:propermodule}
Let $\cB$ be a smooth and proper $A_\infty$-category over a field. 
A proper right/left module or a bimodule over $\cB$ is perfect. 
\end{lem}
\begin{proof}
Let $\cM$ be a proper right module over $\cB$. Then, $\cM\otimes_\cB \cB\simeq \cM $ as right modules. One can represent the diagonal bimodule $\cB$ in terms of Yoneda bimodules $h^{L}\boxtimes h_{L'}$, where $h^{L}$ denotes the left Yoneda module, $h_{L'}$ denotes the right Yoneda module, and $\boxtimes$ denotes the exterior tensor product (Ganatra \cite[(2.83),(2.84)]{sheelthesis}). Observe that
\begin{equation}
\cM\otimes_\cB (h^{L}\boxtimes h_{L'})\simeq \cM(L)\otimes h_{L'}.
\end{equation}
Therefore, $\cM$ can be written as direct a summand of a twisted complex (iterated cone) of  modules of type $\cM(L)\otimes_{\cB} h_{L'}$, but the latter is quasi-isomorphic to finitely many copies of $h_{L'}$ as $\cM$ is proper. This concludes the proof.

The proof is the same for left modules, and for bimodules one uses 
\begin{equation}
\cM\simeq \cB\otimes_{\cB}\cM\otimes_{\cB} \cB
\end{equation}
together with the finite resolution of the diagonal on both sides.
\end{proof}
Assume that $\cB$ is smooth and proper $A_\infty$-category over $\bQ_p$ (or over a subfield of $\bQ_p$). Lemma \ref{lem:propermodule} immediately generalizes to:
\begin{lem}\label{lem:properfamily}
A proper $p$-adic family of bimodules over $\cB$ is perfect. 	
\end{lem}
\begin{proof}
Let $\fM$ be a proper family. The quasi-isomorphism 	
\begin{equation}
\fM\simeq \cB\famotimes_{\cB}\fM\famotimes_{\cB} \cB
\end{equation}
still holds true. By using the representation of the diagonal bimodule as a direct summand of a twisted complex of Yoneda bimodules, we see that $\fM$ is quasi-isomorphic to a direct summand of a twisted complex of families of the form
\begin{equation}
(h^{L_1}\boxtimes h_{L_1'}) \famotimes_{\cB}\fM\famotimes_{\cB} (h^{L_2}\boxtimes h_{L_2'})\simeq \fM(L_2,L_1')\otimes (h^{L_1}\boxtimes h_{L_2'}).
\end{equation}
Thus, it suffices to show that the last type of family is perfect. Note that here we consider $\fM(L_2,L_1')$ as a chain complex over $\bQ_p\langle t\rangle $, and $(h^{L_1}\boxtimes h_{L_2'})$ is the Yoneda bimodule as before. The family structure on their tensor product is obvious. 

By assumption, $\fM(L_2,L_1')$ has finitely generated cohomology over $\bQ_p\langle t\rangle$ (or whichever parameter space we are using). By \Cref{lem:freeres}, 
there exists a complex $C$ of finitely generated, free modules over the Tate algebra quasi-isomorphic to $\fM(L_2,L_1')$. Then, the bimodule 
\begin{equation}
C\otimes (h^{L_1}\boxtimes h_{L_2'})\simeq \fM(L_2,L_1')\otimes (h^{L_1}\boxtimes h_{L_2'})
\end{equation}
is perfect. This completes the proof.
\end{proof}
\begin{rk}
The notions of $p$-adic family of left/right modules can be defined similarly. Then, Lemma \ref{lem:properfamily} still holds for such families.
\end{rk}
\begin{cor}\label{cor:properconv}
Let $\fM_1$ and $\fM_2$ be two proper $p$-adic families. Then the convolution $\pi_1^*\fM_1\relotimes_{\cB}\pi_2^*\fM_2$ is proper.
\end{cor}
\begin{proof}
By Lemma \ref{lem:properfamily}, both families are perfect; therefore, they can be represented as summands of complexes of constant families of Yoneda bimodules. It follows from Example \ref{exmp:constconvolution} that the convolution of two constant families of Yoneda bimodules is perfect; hence, proper. Therefore, $\pi_1^*\fM_1\relotimes_{\cB}\pi_2^*\fM_2$ can be represented as the direct summand of a twisted complex (iterated cone) of proper modules and it is proper itself.
\end{proof}
\begin{cor}\label{lem:Nproper} Let $\fN$ denote the cone of the morphism (\ref{eq:convolutiontodiag}). Then, $H^*(\fN(L_i,L_j))$ is a finitely generated module over $\bQ_p\langle t_1,t_2\rangle$ for all $L_i,L_j$, i.e. $\fN$ is proper. 
\end{cor}
\begin{proof}
By construction, $\pi_1^*\fM_\alpha^{\bQ_p}$ and $\pi_2^*\fM_\alpha^{\bQ_p}$ are both proper modules; therefore, by Corollary \ref{cor:properconv}, $\pi_1^*\fM_\alpha^{\bQ_p}\relotimes_{\cF(M,\bQ_p)} \pi_2^*\fM_\alpha^{\bQ_p}$ is also proper. Since $\Delta^*\fM_\alpha^{\bQ_p}$ is proper too, the cone of a morphism 
\begin{equation}
\pi_1^*\fM_\alpha^{\bQ_p}\relotimes_{\cF(M,\bQ_p)} \pi_2^*\fM_\alpha^{\bQ_p}\to \Delta^*\fM_\alpha^{\bQ_p}
\end{equation}
is proper. This completes the proof. 
\end{proof}
\begin{prop}\label{prop:grouplikepadic}
$H^*(\fN)$ vanishes over the Tate algebra $\bQ_p\langle t_1/p^n,t_2/p^n\rangle$ for a sufficiently large $n$. Therefore, $\fM_\alpha^{\bQ_p}|_{\bQ_p\langle t/p^n\rangle}$ is group-like. 
\end{prop}
\begin{proof}
By Lemma \ref{lem:Nproper}, the cohomology of $\fN$ is finitely generated over $\bQ_p\langle t_1,t_2\rangle$, and it vanishes at $t_1=t_2=0$ (as (\ref{eq:convolutiontodiag}) is a quasi-isomorphism at $t_1=t_2=0$). Each $\fN(L_i,L_j)$ is a $\bZ/2\bZ$-graded complex, and can be seen as a doubly periodic infinite complex with finite cohomology at each degree. For some natural number $k\gg 0$, consider the truncation $\tau_{\leq k}\tau_{\geq -k}\fN(L_i,L_j)$, which is a bounded complex with finitely generated cohomology. As $k\gg 0$ (indeed, $k\geq 2$ is sufficient). The cohomology of $\fN(L_i,L_j)$ at even, resp. odd degrees coincide with the cohomology of $\tau_{\leq k}\tau_{\geq -k}\fN(L_i,L_j)$ at degree $0$, resp. $1$ (due to two periodicity). By \cite[Proposition 6.5]{kedlayaoverconvergent}, every finitely generated $\bQ_p\langle t_1,t_2\rangle$-module has a finite, free resolution, which implies that $\tau_{\leq k}\tau_{\geq -k}\fN(L_i,L_j)$ is quasi-isomorphic to a free finite complex of $\bQ_p\langle t_1,t_2\rangle$-modules (cf. \Cref{lem:freeres}). Then, the result follows from \Cref{lem:semiconqpaff}, applied to a free finite complex quasi-isomorphic to $\bigoplus_{i,j} \tau_{\leq k}\tau_{\geq -k}\fN(L_i,L_j)$.
%
\end{proof}

\begin{rk}
As we will explain in more detail in \Cref{appendix:tatealgebras}, the ring $\Q_p\langle t/p^n\rangle$ is the set of power series in $t$ that converge over the $p$-adic disc $\bD_{p^{-n}}=p^n\bZ_p$ of radius $p^{-n}$. It can be thought of as the ring of analytic functions on $\bD_{p^{-n}}$. On the other hand, it is also a subgroup of the $p$-adic unit disc $\bD_1=\bZ_p$. This is in sharp contrast with the Archimedean geometry: if $E$ is a connected Lie group, and $U\subsetneq E$ is a neighborhood of the identity, then $U$ is never closed under multiplication. Indeed, the union of iterates $U$, $UU\subset E$, $UUU\subset E$, $\dots$ cover $E$ for every neighborhood of the identity. 
%
\end{rk}
\begin{rk}\label{rk:padiccoherent}
As pointed out in \Cref{rk:generalcoherent}, the map \eqref{eq:convolutiontodiag} can be used to produce two maps 
\begin{equation}\label{eq:twomaps}
	\pi_1^*\fM_\alpha^{\bQ_p}\relotimes_{\cF(M,\bQ_p)} \pi_2^*\fM_\alpha^{\bQ_p}\relotimes_{\cF(M,\bQ_p)} \pi_3^*\fM_\alpha^{\bQ_p}\rightrightarrows \Delta^*\fM_\alpha^{\bQ_p}
\end{equation}
of families over $\bQ_p\langle t_1,t_2,t_3\rangle$, which are not a priori the same (here, $\pi_i$ used are the projections on the triple product, and $\Delta$ denotes the ordered addition of the three components). It is natural to ask for these two maps to be homotopic and to fix a homotopy. In this specific case, such a homotopy is provided by a count of rigid discs analogous to \Cref{figure:groupquasi}, but with three wavy lines, as shown in \Cref{figure:associator}. We will not need this property though. 
\end{rk}
\begin{figure}\centering
	\includegraphics[height=4 cm]{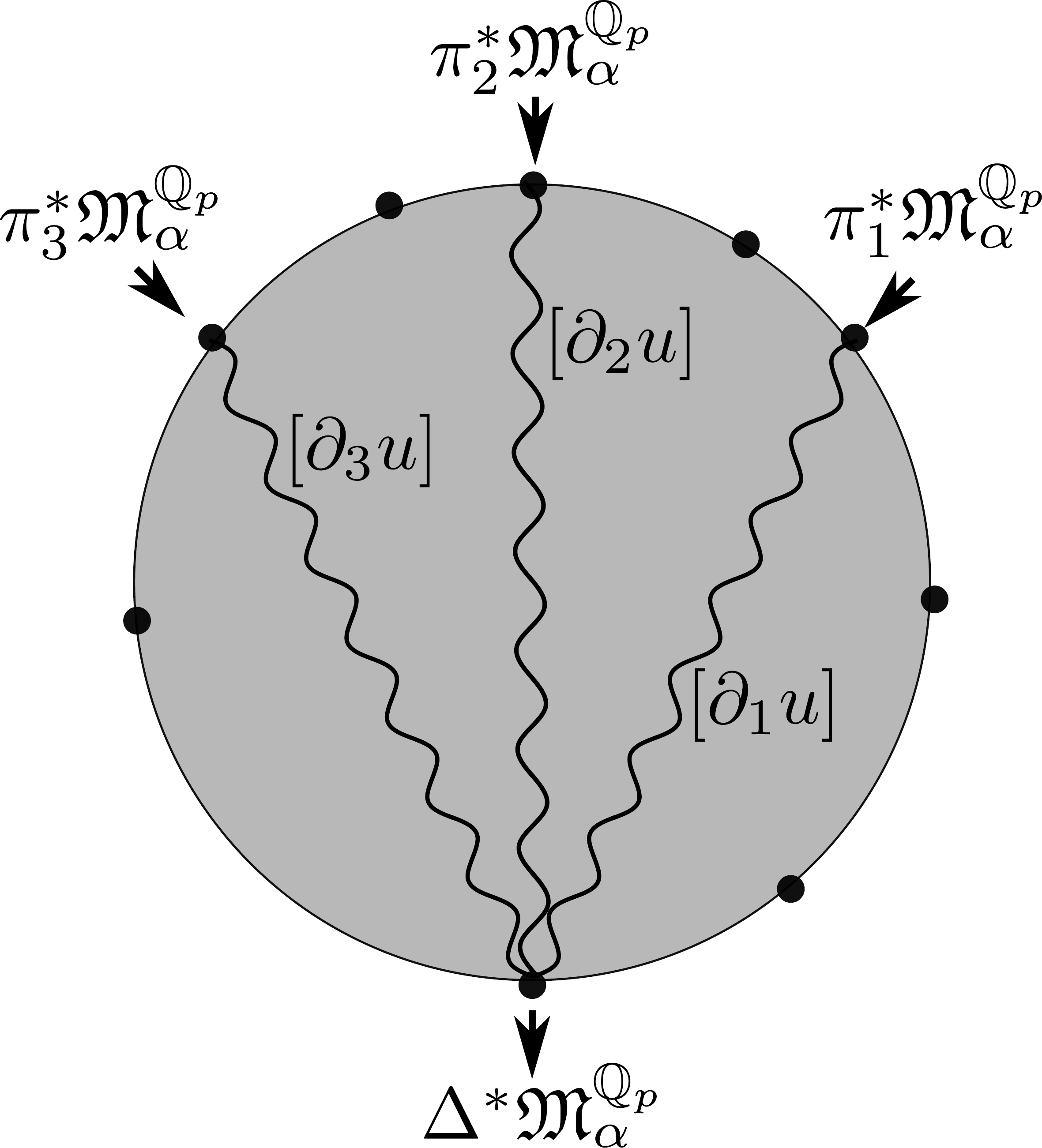}
	\caption{The counts defining the homotopy of two maps \eqref{eq:twomaps}.}
	\label{figure:associator}
\end{figure}
Even though Proposition \ref{prop:grouplikepadic} is about the family $\bQ_p$, it still allows us to conclude a weak group-like statement for the family $\fM_\alpha^\Lambda$. Consider $\faM{f}{K}$, for $f\in\bR$.
\begin{cor}\label{cor:groupovernovikov}
Let $f,f'\in p^n\bZ_{(p)}\subset \bQ$. Then, $\faM{f}{K} \otimes_{\cF(M,K)}  \faM{f'}{K}\simeq \faM{f+f'}{K}$. The same statement holds if $K$ is replaced by $\Lambda$.
\end{cor}
In other words, the bimodules $\faM{f}{K}$ and $\faM{f}{\Lambda}$ behave like a group action for rational numbers $f\in\bQ$ that are $p$-adically close to $0$. 
\begin{proof}
One can define the map 
\begin{equation}\label{eq:mapfromsum}
\faM{f}{K} \otimes_{\cF(M,K)}  \faM{f'}{K}\to \faM{f+f'}{K}
\end{equation}
similar to (\ref{eq:convolutiontodiag}). More precisely, one would need to replace (\ref{eq:padicmapformula}) by 
\begin{equation}\label{eq:rationalmapformula}
\sum\pm T^{E(u)} T^{f\alpha\cdot[\partial_1 u]} T^{f'\alpha\cdot[\partial_2 u]}.y,
\end{equation}
and 
this defines a bimodule map. Moreover, after the base change along $\kappa:K\to \bQ_p$, the map (\ref{eq:mapfromsum}) becomes the same as (\ref{eq:padicmapformula}) evaluated at $t_1=f,t_2=f'$ considered as elements of $p^n\bZ_p\subset\bQ_p$. The same holds for the cone of (\ref{eq:mapfromsum}). By Proposition \ref{prop:grouplikepadic}, this cone vanishes after the base change, implying it is $0$ before the base change as well. Therefore, (\ref{eq:mapfromsum}) is a quasi-isomorphism before the base change as well. One can then apply base change along the inclusion $K\to\Lambda$ to conclude $\faM{f}{\Lambda} \otimes_{\cF(M,\Lambda)}  \faM{f'}{\Lambda}\simeq \faM{f+f'}{\Lambda}$.
\end{proof}
\section{Comparison with the action of $\phi^t_\alpha$ and proof of Theorem \ref{thm:mainthm}}\label{sec:comparisonandmaintheorem}
We need the following proposition to conclude the proof of Theorem \ref{thm:mainthm}:
\begin{prop}\label{prop:compareprop}
For any $f\in p^n\bZ_{(p)}$, \begin{equation}\label{eq:compareprop1}
HF(\phi^f_\alpha (L),L')\cong HF(L,\phi^{-f}_\alpha (L'))\cong H^*(\faM{-f}{\Lambda}(L,L') ),
\end{equation}
and this group has the same dimension over $\Lambda$ as 
\begin{equation}\label{eq:compareprop2}
dim_{K}\big(H^*\big(h_{L'}\otimes_{\cF(M,K)}\faM{-f}{K} \otimes_{\cF(M,K)}h^L\big)\big)= \atop dim_{\bQ_p}\big(H^*\big(h_{L'}\otimes_{\cF(M,\bQ_p)}\fM^{\bQ_p}_{\alpha}|_{t=-f} \otimes_{\cF(M,\bQ_p)}h^L\big)\big).
\end{equation}
\end{prop}
We want to show that the bimodules $\faM{f}{\Lambda}$ act on the perfect modules coming from Lagrangians in the expected way. 

Recall that for a given $\tilde L\subset M$ satisfying some assumptions (such as $\tilde L\pitchfork L_i$ for all $i$), we have a right $\cF(M,\Lambda)$-module $h_{\tilde L}$ that satisfies $h_{\tilde L}(L_i)=CF(L_i,\tilde L)=CF(L_i,\tilde L;\Lambda)$. The differential and structure maps are defined by counting marked discs with one boundary component on $\tilde L$, and with other boundary components on various $L_i$. We will deform this module algebraically to give a simpler description of $h_{\phi_\alpha^f(\tilde L)}$ for small $f$: 
\begin{defn}\label{defn:algyoneda}
Fix homotopy classes of paths from the base point of $M$ to intersection points $\tilde L\cap L_i$. Let $h_{\phi_\alpha^f,\tilde L}^{alg}$ denote the right $\cF(M,\Lambda)$-module defined by \begin{equation}
L_i\mapsto CF(L_i,\tilde L)
\end{equation} and whose differential/structure maps are given by formulae of the form
\begin{equation}\label{eq:structurealgyoneda}
\sum \pm T^{E(u)}T^{f\alpha([\partial_{\tilde L} u])}.y ,
\end{equation}
where $u$ ranges over holomorphic curves with one boundary component on $\tilde L$ (the one on the clockwise direction from the output) and other components on $L_i$. Here, $y$ is the output marked point and $[\partial_{\tilde L} u]\in H_1(M;\bZ)$ denotes the homology class of the path obtained by concatenating the fixed paths from the base point with ($u$-image of) the wavy path in Figure \ref{figure:hphilalg} from the module input to module output $y$. 
\end{defn}
Clearly, $h_{\phi_\alpha^0,\tilde L}^{alg}=h_{\tilde L}$; therefore, one can see $h_{\phi_\alpha^f,\tilde L}^{alg}$ as a deformation of $h_{\tilde L}$. The proof of the following Lemma implies (\ref{eq:structurealgyoneda}) converge $T$-adically, i.e. $h_{\phi_\alpha^f,\tilde L}^{alg}$ is well-defined, for small $|f|$:
\begin{figure}\centering
	\includegraphics[height=4 cm]{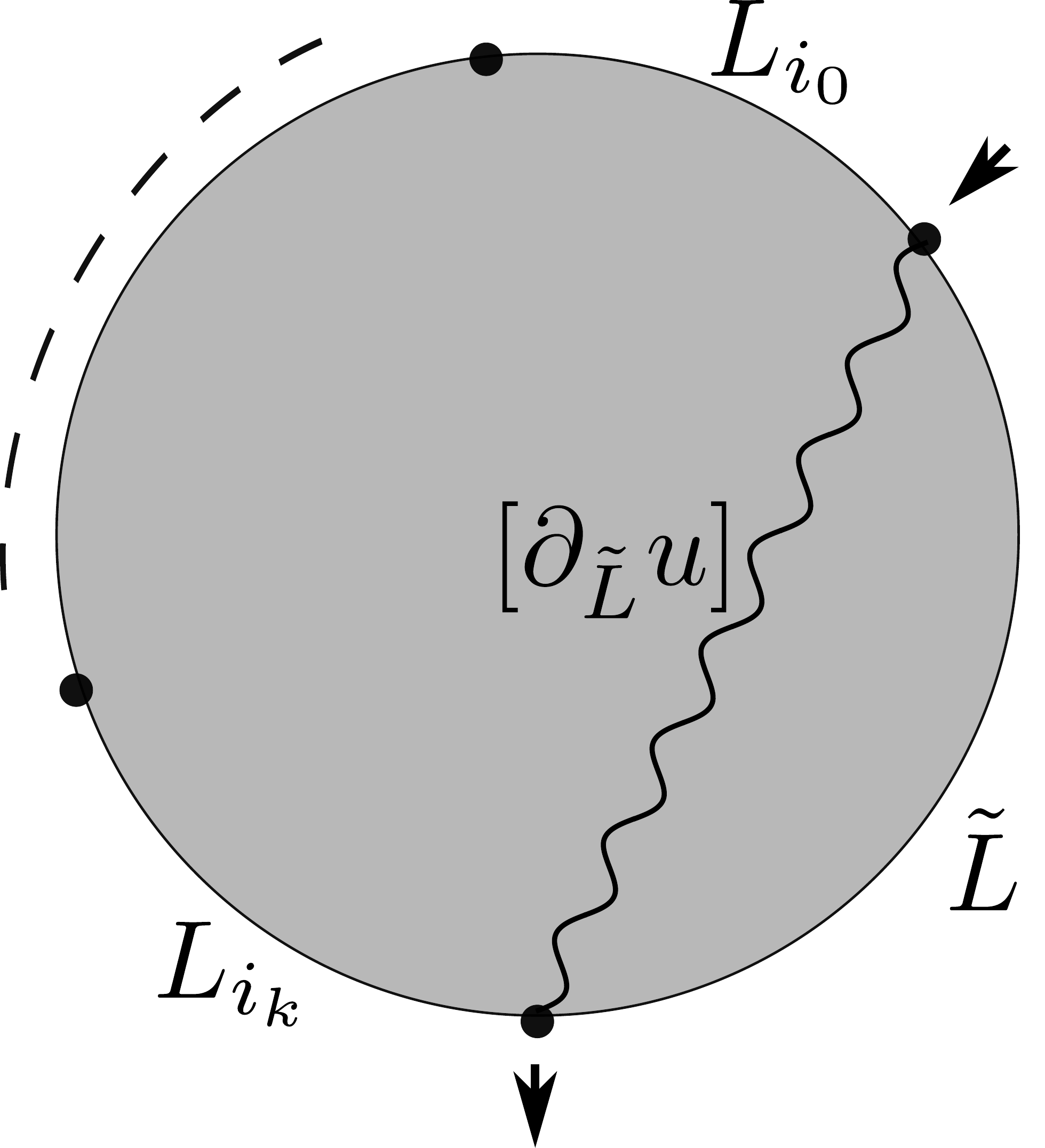}
	\caption{The counts defining $h^{alg}_{\phi_\alpha^f,\tilde L}$, which gives a simpler quasi-isomorphic description of $h_{\phi^f_{\alpha}(\tilde{L})}$ for small $f$}
	\label{figure:hphilalg}
\end{figure}
\begin{lem}\label{lem:halg=h}
For $f\in \bR$ such that $|f|$ is small, we have $h_{\phi_\alpha^f,\tilde L}^{alg}\simeq h_{\phi_\alpha^f(\tilde L)}$. In other words, $h_{\phi_\alpha^f,\tilde L}^{alg}$ gives what is expected geometrically.
\end{lem}
How small $|f|$ should be depends on the Lagrangian $\tilde{L}$.
\begin{proof}
This is an application of Fukaya's trick, as in \cite{abouzaidicm}, for instance. Namely, if $f$ is small enough, the intersection points $\phi^f_\alpha(\tilde L)\cap L_i$ can be identified with $\tilde L\cap L_i$, and this holds throughout the isotopy. Consider a smooth isotopy $\psi^f$ of $M$ fixing all $L_i$ setwise and mapping $\tilde L$ to $\phi^f_\alpha(\tilde L)$. Choose the Floer data (with vanishing Hamiltonian term) for $(L_i,\tilde L)$ and $(L_i,\phi_\alpha^f(\tilde L))$ to be related by $\psi^f$. Similarly, choose the perturbation data for discs with boundary on $(L_{i_0},\dots, L_{i_k}, \tilde L)$ and $(L_{i_0},\dots, L_{i_k}, \phi_\alpha^f(\tilde L))$ to be related by $\psi^f$. For small $|f|$, the tameness and regularity will be preserved. Then, one can identify the moduli of pseudo-holomorphic discs labeled by $(L_{i_0},\dots, L_{i_k}, \tilde L)$ with pseudo-holomorphic discs labeled by $(L_{i_0},\dots, L_{i_k}, \phi^f_{\alpha}(\tilde L))$, where the identification is via the composition by $\psi^f$. One has the energy identity 
\begin{equation}\label{eq:energyidentity}
E(\psi^f\circ u)=E(u)+f\alpha([\partial_{\tilde L} u])-g(f,x)+g(f,y),
\end{equation}
where $x$ is the input, $y$ is the output and $g(f,y)$ is a real number that only depends on $f$ and $y$ (it may depend on the homotopy class of the isotopy, but the isotopy is given in our situation). A version of this identity is given in \cite[Lemma 3.2]{abouzaidicm}. In our situation, this is still a similar application of Stokes theorem, namely if one moves the disc by $\psi^f$, then the energy difference can be measured as the area traced by the part of boundary labeled by $\tilde L$ (which is homotopic to the wavy path in Figure \ref{figure:hphilalg}). This energy difference can be measured by $f\alpha([\partial_{\tilde L} u])$, except one has to correct it by the areas traced by the fixed paths from the base
point of $M$ to intersection points $x,y$. The correction can be written in the form $g(f,x)-g(f,y)$, where $g(f,x)$ is a number that depends on $x$ and continuously on $f$ (we neither need nor attempt to compute it).

After rescaling the generators via $x\mapsto T^{g(f,x)}x$, one can identify the structure maps of $h_{\phi_\alpha^f(\tilde L)}$ with respect to the above perturbation data, and the structure maps of $h_{\phi_{\alpha}^f,\tilde L}^{alg}$ defined in (\ref{eq:structurealgyoneda}). The quasi-isomorphism class of $h_{\phi_\alpha^f(\tilde L)}$ is independent of the perturbation data; therefore, $h_{\phi_{\alpha}^f,\tilde L}^{alg}$ is well-defined and quasi-isomorphic to $h_{\phi_\alpha^f(\tilde L)}$. 
\end{proof}
\begin{rk}\label{rk:seidelmodelftrick}
It is possible to apply Fukaya's trick even if one uses the model of the Fukaya category described in \cite{seidelbook}. In this case, as $\tilde L$ is assumed to be transverse to all $L_i$, we can choose the Floer data for the pair $(L_i,\tilde L)$ with vanishing Hamiltonian term (we do not need Floer data for $(\tilde L,\tilde L)$ in order to define $h_{\tilde L}$). When we choose the smooth isotopy $\psi^f$ as above, we have to make sure it fixes a neighborhood of the intersection points and Hamiltonian chords between various $L_i$. We deform Floer data and perturbation data by $\psi^f$, and we also have to make sure that for a small time regularity is preserved and no new Hamiltonian chords are introduced between various $L_i$. An analogue of the energy identity (\ref{eq:energyidentity}) holds, and the rest of the argument works in the same way.
\end{rk}
The module $h^{alg}_{\phi^f_{\alpha},\tilde{L}}$ comes from a Novikov family of right modules over $\cF(M,\Lambda)$:
\begin{defn}
Let $\fh^{alg}_{\phi_\alpha,\tilde L}$ be the family of right $\cF(M,\Lambda)$-modules defined by \begin{equation}
L_i\mapsto CF(L_i,\tilde L)\otimes \Lambda\{z^\bR\}
\end{equation} and whose differential/structure maps are given by 
\begin{equation}\label{eq:structurealgyonedafamily}
\sum \pm T^{E(u)}z^{\alpha([\partial_{\tilde L} u])}.y ,
\end{equation}where the sum is analogous to (\ref{eq:structurealgyoneda}).
\end{defn}
Recall that $\Lambda\{z^\bR\}$ denotes the ring $\Lambda\{z^\bR\}_{[a,b]}$ for some $a<0<b$, which we introduced in Section \ref{sec:families} briefly, and explain more in \Cref{appendix:semicont}. The series (\ref{eq:structurealgyonedafamily}) belongs to $\Lambda\{z^\bR\}_{[a,b]}$ for small $|a|$ and $|b|$ (this is equivalent to convergence of (\ref{eq:structurealgyoneda}) for $f\in[a,b]$). One can replace this ring with $\Lambda[z^\bR]$ when $\tilde L$ is Bohr-Sommerfeld monotone. Clearly, $h^{alg}_{\phi^f_\alpha,\tilde L}=\fh^{alg}_{\phi_\alpha,\tilde L}|_{z=T^f}$. 

Using Lemma \ref{lem:halg=h}, we establish:
\begin{lem}\label{lem:nearbyfloer}
For $f\in\bR$ with small $|f|$, $h_{\tilde L}\otimes_{\cF(M,\Lambda)}\faM{f}{\Lambda}\simeq h_{\phi^f_\alpha(\tilde L )}$ and $h_{\tilde L}\simeq h_{\phi^f_\alpha(\tilde L )}\otimes_{\cF(M,\Lambda)}\faM{-f}{\Lambda}$.
\end{lem}
\begin{figure}\centering
	\includegraphics[height=4 cm]{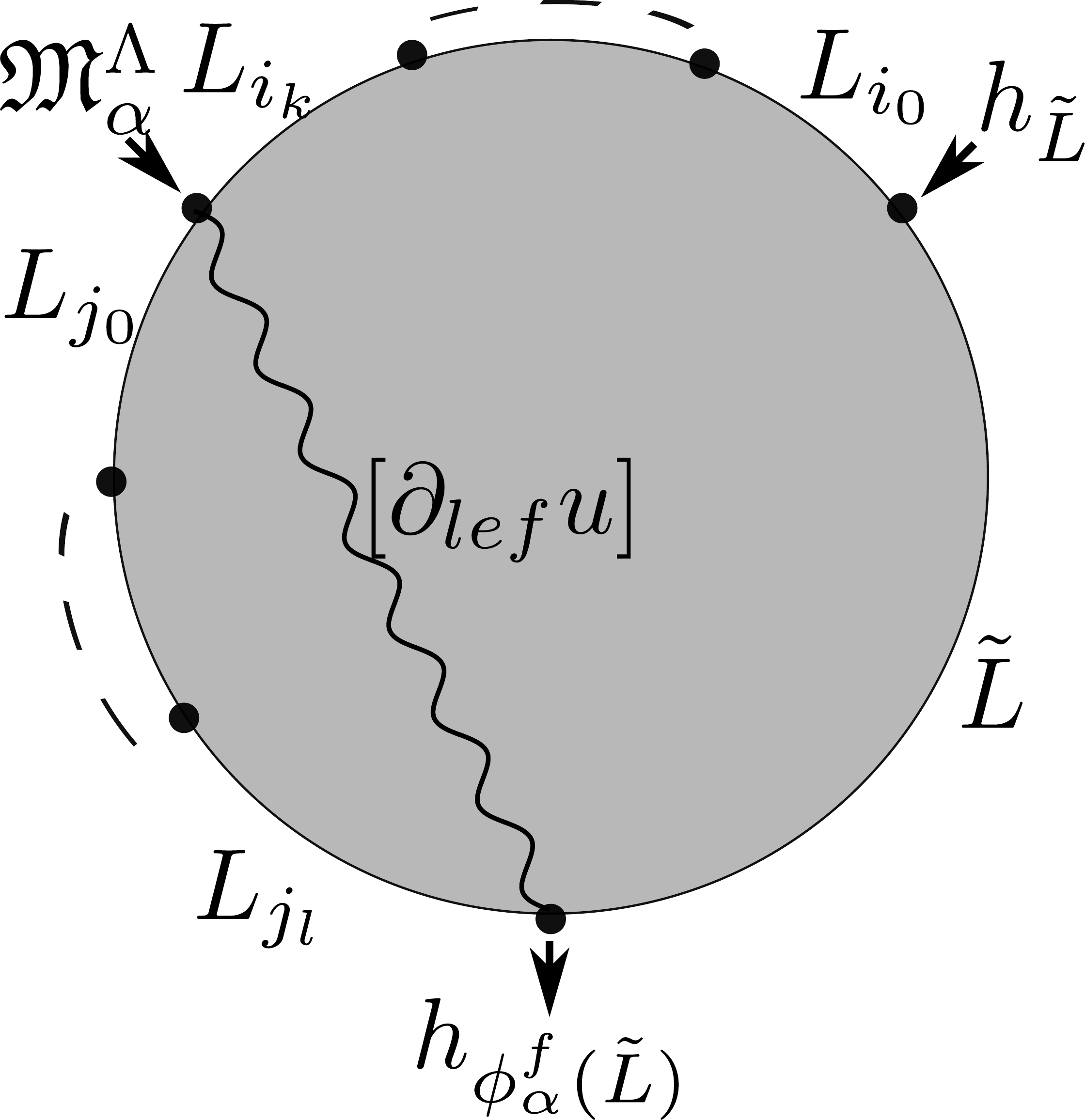}
	\caption{The counts defining (\ref{eq:hlconvtophihl})}
	\label{figure:hlconvtophihl}
\end{figure}
\begin{proof}
By Lemma \ref{lem:halg=h}, it suffices to prove the corresponding statement for $h^{alg}_{\phi_{\alpha}^f,\tilde L}$. We define a map
\begin{equation}\label{eq:hlconvtophihl}
h_{\tilde L}\otimes_{\cF(M,\Lambda)}\faM{f}{\Lambda} \to h^{alg}_{\phi^f_{\alpha},\tilde{L}}
\end{equation}
in a way very similar to (\ref{eq:convolutiontodiagformula}). Namely, send
\begin{equation}
(x\otimes x_1\otimes\dots \otimes x_k\otimes x';x_1',\dots x_l')
\end{equation}
to the sum 
\begin{equation}\label{eq:formulahconvtophihl}
\sum\pm T^{E(u)} T^{f\alpha([\partial_{lef} u])}.y,
\end{equation}
where $[\partial_{lef} u]$ denotes the $u$-image of the part of the boundary of the disc from $\faM{f}{\Lambda}$-input $x'$ to output $y$ (i.e. the wavy line in Figure \ref{figure:hlconvtophihl}, up to homotopy) concatenated with paths from the base point of $M$. One can check that (\ref{eq:hlconvtophihl}) is an $A_\infty$-module homomorphism in a way similar to (\ref{eq:convolutiontodiag}) by considering the degenerations of discs in Figure \ref{figure:hlconvtophihl} analogous to Figure \ref{figure:compactificationgroupquasi}. When $f=0$, (\ref{eq:hlconvtophihl}) is a quasi-isomorphism as its cone is the standard bar resolution of $h_{\tilde L}$. Also, (\ref{eq:hlconvtophihl}) can be thought of as the specialization of a map of families 
\begin{equation}\label{eq:hlconvtophihlfamilyversion}
h_{\tilde L}\famotimes_{\cF(M,\Lambda)}\fM_{\alpha}^\Lambda \to \fh^{alg}_{\phi_{\alpha},\tilde{L}}
\end{equation} at $z=T^f$, which can be defined by replacing (\ref{eq:formulahconvtophihl}) by $\sum\pm T^{E(u)} z^{\alpha([\partial_{lef} u])}.y$.

The proof that (\ref{eq:hlconvtophihl}) is a quasi-isomorphism for small $|f|$ is similar to the proof of \Cref{lem:grouplikenovikov}, which we give in \Cref{appendix:semicont}. More precisely, this claim follows by \Cref{lem:semicontcomplexnovsingle}, i.e. from a semi-continuity argument (cf. the proof of \Cref{prop:grouplikepadic}). The application of \Cref{lem:semicontcomplexnovsingle} requires a lift of the map \eqref{eq:hlconvtophihlfamilyversion} under the ring homomorphism $\Lambda\{z^H\}\hookrightarrow \Lambda\{z^\bR\}$ explained in \Cref{appendix:semicont} in detail. Here, $H=H_1(M,\bZ)/ker(\alpha)$ is a free abelian group, $\Lambda\{z^H\}$ is the ring of series $\sum_{C\in H} a_Cz^C$ with a convergence condition analogous to $\Lambda\{z^\bR\}$. The injection $\Lambda\{z^H\}\to \Lambda\{z^\bR\}$ is given by $z^C\mapsto z^{\alpha(C)}$. The fact that \eqref{eq:hlconvtophihlfamilyversion} lifts is analogous to \Cref{lem:lifts}, namely one replaces $z^{\alpha([\partial_{lef} u])}$ in \eqref{eq:hlconvtophihlfamilyversion} by $z^{[\partial_{lef} u]}$. Let $\fN$ denote the cone of the lift of \eqref{eq:hlconvtophihlfamilyversion}, which is perfect similar to \Cref{cor:properconv}. This cone vanishes at $\mathbf{z=1}$. As a result, analogous to the proof of \Cref{lem:grouplikenovikov}, \eqref{eq:hlconvtophihlfamilyversion} is a quasi-isomorphism at $z=T^f$ for small $|f|$. We omit the details.

Now, define 
\begin{equation}\label{eq:phihlconvtohl}
h_{\phi^f_{\alpha},\tilde{L}}^{alg}\otimes_{\cF(M,\Lambda)}\faM{-f}{\Lambda} \to h_{\tilde{L}}
\end{equation}
similarly by replacing (\ref{eq:formulahconvtophihl}) by $\sum\pm T^{E(u)} T^{f\alpha([\partial_{top} u])}.y$, where we define $[\partial_{top} u]\in H_1(M;\bZ)$ to be the class obtained by concatenating the path from $h_{\phi^f_{\alpha},\tilde{L}}^{alg}$ input ($h_{\tilde{L}}$ input in Figure \ref{figure:hlconvtophihl}) to $\faM{-f}{\Lambda}$ input ($\fM_{\alpha}^\Lambda$ input in Figure \ref{figure:hlconvtophihl}) with the fixed paths to the base point. Then (\ref{eq:phihlconvtohl}) is also an $A_\infty$-module homomorphism, and it is a quasi-isomorphism when $f=0$. Similarly, (\ref{eq:phihlconvtohl}) extends to a map of families
\begin{equation}
\fh_{\phi_{\alpha},\tilde{L}}^{alg}\relotimes_{\cF(M,\Lambda)}\fM_{-\alpha}^\Lambda \to \fh^{const}_{\tilde{L}},
\end{equation} where $\fM_{-\alpha}^\Lambda$ is defined by the obvious modification of $\fM_{\alpha}^\Lambda$ (here, $\fh^{const}_{\tilde{L}}$ denotes the constant family of Yoneda modules, defined analogous to \Cref{exmp:constfams}). The same proof applies to show (\ref{eq:phihlconvtohl}) is a quasi-isomorphism for small $|f|$. 
\end{proof}
An immediate corollary of Lemma \ref{lem:nearbyfloer} is the following:
\begin{cor}\label{cor:definable}
The Yoneda module $h_{\phi^f_\alpha(L')}$ is definable over $\cF(M,K)$ for $f\in \bZ_{(p)}$ such that $|f|$ is sufficiently small. In other words, there exists a perfect module over $\cF(M,K)$ such that one obtains a module quasi-isomorphic to $h_{\phi^f_\alpha(L')}$ after base change to $\Lambda$. 
\end{cor}
\begin{proof}
Consider the module $h_{L'}\otimes_{\cF(M,K)}\faM{f}{K}$ over $\cF(M,K)$, which is well-defined as $f\in\bZ_{(p)}$ and $T^{f\alpha(C) }\in K$. Both $h_{L'}$ and $\faM{f}{K}$ are perfect as a consequence of \Cref{lem:propermodule}, implying the same for $h_{L'}\otimes_{\cF(M,K)}\faM{f}{K}$ (c.f. \Cref{cor:properconv}). Once one extends the coefficients to $\Lambda$, one obtains $h_{L'}\otimes_{\cF(M,\Lambda)}\faM{f}{\Lambda}$, which is quasi-isomorphic to $h_{\phi^f_\alpha(L')}$
by Lemma \ref{lem:nearbyfloer}.
\end{proof}
\begin{rk}
One could use Lemma \ref{lem:halg=h} directly to prove the corollary. Since the sum defining $h_{L'}$ is assumed to be finite, (\ref{eq:structurealgyoneda}) also remains finite. Therefore, if $f\in \bZ_{(p)}$, $h^{alg}_{\phi^f_\alpha,L'}$ is definable over $\cF(M,K)$ in the sense above, and by Lemma \ref{lem:halg=h}, it becomes quasi-isomorphic to $h_{\phi^f_\alpha(L')}$ after extending the coefficients $K\subset \Lambda$. 	
\end{rk}
We can drop the assumption that $|f|$ is small from Lemma \ref{lem:nearbyfloer} after modifying it as follows:
\begin{lem}\label{lem:iteratedconv}
For any $f>0$ and any $\tilde L$, one can find a sequence of numbers $0=s_0<s_1<\dots<s_r=f$ such that
\begin{equation}\label{eq:iteratedconv}
h_{\phi^f_\alpha(\tilde L)}\simeq h_{\tilde L}\otimes_{\cF(M,\Lambda)} \faM{s_1-s_0}{\Lambda} \otimes_{\cF(M,\Lambda)} \faM{s_2-s_1}{\Lambda} \dots \otimes_{\cF(M,\Lambda)} \faM{s_r-s_{r-1}}{\Lambda}.
\end{equation}
Similarly, when $f<0$, there is a sequence $0=s_0>s_1>\dots>s_r=f$ such that (\ref{eq:iteratedconv}) is satisfied. 
Moreover, if $f\in \bZ_{(p)}$, resp. $f\in p^n\bZ_{(p)}$ then one can assume that all $s_i\in \bZ_{(p)}$, resp. $s_i\in p^n\bZ_{(p)}$.
\end{lem}
\begin{proof}
Without loss of generality, assume that $f>0$ and consider the isotopy $\phi_{\alpha}^{s}(\tilde L), s\in[0,f]$. By Lemma \ref{lem:nearbyfloer}, any $s\in[0,f]$ has an open neighborhood $(s-\epsilon,s+\epsilon)$ such that
\begin{equation}
h_{\phi_\alpha^{s}(\tilde L)}\simeq h_{\phi_\alpha^{s'}(\tilde L)}\otimes_{\cF(M,\Lambda)} \faM{s-s'}{\Lambda},\atop
h_{\phi_\alpha^{s''}(\tilde L)}\simeq h_{\phi_\alpha^{s}(\tilde L)}\otimes_{\cF(M,\Lambda)} \faM{s''-s}{\Lambda}
\end{equation}
for $s',s''\in(s-\epsilon,s+\epsilon)$. 

Assume that $\epsilon$ is small enough so that Lemma \ref{lem:grouplikenovikov} is also satisfied for $|f|,|f'|<\epsilon$, i.e. $\faM{f}{\Lambda}\otimes_{\cF(M,\Lambda)} \faM{f'}{\Lambda}\simeq \faM{f+f'}{\Lambda}$ for $|f|,|f'|<\epsilon$. Then 
\begin{align}
h_{\phi_\alpha^{s''}(\tilde L)}\simeq h_{\phi_\alpha^{s}(\tilde L)}\otimes_{\cF(M,\Lambda)} \faM{s''-s}{\Lambda}\simeq \\
h_{\phi_\alpha^{s'}(\tilde L)}\otimes_{\cF(M,\Lambda)} \faM{s-s'}{\Lambda} \otimes_{\cF(M,\Lambda)} \faM{s''-s}{\Lambda}\simeq \\
h_{\phi_\alpha^{s'}(\tilde L)}\otimes_{\cF(M,\Lambda)} \faM{s''-s'}{\Lambda}
\end{align}
for any $s',s''\in(s-\epsilon,s+\epsilon)$. The last equivalence comes from Lemma \ref{lem:grouplikenovikov}. 

Call such $s',s''$ algebraically related. By Heine--Borel theorem, one can cover $[0,f]$ by finitely many such small intervals, and this gives a sequence $0=s_0<s_1<\dots< s_r=f$ such that $s_i$ and $s_{i+1}$ are algebraically related, concluding the proof. We could choose any $(r+1)$-tuple $0=s_0'<\dots<s_r'=f$ sufficiently close to $0=s_0<s_1<\dots< s_r=f$, in particular, we can assume that they are in $\bZ_{(p)}$, resp. $p^n\bZ_{(p)}$.
\end{proof}
\begin{cor}\label{cor:uniterated}
If $f\in p^n\bZ_{(p)}$, then 
\begin{equation}\label{eq:uniteratedconv}
h_{\phi_\alpha^f(L') }\simeq h_{L'}\otimes_{\cF(M,\Lambda)} \faM{f}{\Lambda}.
\end{equation}
\end{cor}
\begin{proof}
Without loss of generality assume that $f>0$. If $f\in p^n\bZ_{(p)}$, we can choose $0=s_0<\dots<s_r=f$ in Lemma \ref{lem:iteratedconv} from $p^n\bZ_{(p)}$. Therefore, $s_{i+1}-s_i\in p^n\bZ_{(p)}$ and $\faM{f}{\Lambda}\otimes_{\cF(M,\Lambda)} \faM{f'}{\Lambda}\simeq \faM{f+f'}{\Lambda}$ holds when $f,f'\in p^n\bZ_{(p)}$ by Corollary \ref{cor:groupovernovikov}. This fact, together with (\ref{eq:iteratedconv}) implies the corollary.
\end{proof}
Observe that if one confines themselves to the case of $f\in p^n\bZ_{(p)}$, one can prove Lemma \ref{lem:iteratedconv} and Corollary \ref{cor:uniterated} by using Corollary \ref{cor:groupovernovikov}, rather than Lemma \ref{lem:grouplikenovikov}.
\begin{proof}[Proof of Proposition \ref{prop:compareprop}]
Let $f\in p^n\bZ_{(p)}$. Then by Corollary \ref{cor:uniterated} 
\begin{equation}\label{eq:unitconv}
h_{\phi_\alpha^{-f}(L') }\simeq h_{L'}\otimes_{\cF(M,\Lambda)} \faM{-f}{\Lambda}.
\end{equation}
Applying $(\cdot)\otimes_{\cF(M,\Lambda)} h^L$ to both sides of (\ref{eq:unitconv}), we obtain 
\begin{equation}\label{eq:yonyon}
h_{\phi_\alpha^{-f}(L') }\otimes_{\cF(M,\Lambda)} h^L\simeq h_{L'}\otimes_{\cF(M,\Lambda)} \faM{-f}{\Lambda}\otimes_{\cF(M,\Lambda)} h^L.
\end{equation}
Moreover, \begin{equation}\label{eq:chyonyon}
CF(L,\phi_\alpha^{-f}(L'))
\simeq  h_{\phi_\alpha^{-f}(L') }\otimes_{\cF(M,\Lambda)} h^L
\end{equation}
by Lemma \ref{lem:rightconvlefteqhf}. Combining these, we get 
\begin{equation}\label{eq:combining}
CF(\phi_\alpha^f(L),L')\simeq CF(L,\phi_\alpha^{-f}(L'))\simeq  h_{L'}\otimes_{\cF(M,\Lambda)} \faM{-f}{\Lambda}\otimes_{\cF(M,\Lambda)} h^L,
\end{equation}
which proves (\ref{eq:compareprop1}) as \begin{equation}
h_{L'}\otimes_{\cF(M,\Lambda)} \faM{-f}{\Lambda}\otimes_{\cF(M,\Lambda)} h^L\simeq \faM{-f}{\Lambda}(L,L').
\end{equation} 

The statement about the dimension is straightforward, namely the proper module \begin{equation}
h_{L'}\otimes_{\cF(M,K)}\faM{-f}{K} \otimes_{\cF(M,K)}h^L
\end{equation}
is well-defined. By extending the coefficients under $K\hookrightarrow \Lambda$, one obtains  \begin{equation}
h_{L'}\otimes_{\cF(M,\Lambda)}\faM{-f}{\Lambda} \otimes_{\cF(M,\Lambda)}h^L
\end{equation}
on the one hand, and by extending coefficients under $K\hookrightarrow\bQ_p$, one obtains
\begin{equation}
h_{L'}\otimes_{\cF(M,\bQ_p)}\fM^{\bQ_p}_{\alpha}|_{t=-f} \otimes_{\cF(M,\bQ_p)}h^L.
\end{equation}
Base change under field extensions do not change the dimensions of cohomology groups, and this finishes the proof.
\end{proof}
\begin{rk}
One can generalize Proposition \ref{prop:compareprop} to all $f\in\bZ_{(p)}$ by using an argument that will also be used in the proof of Theorem \ref{thm:mainthm}. 
\end{rk}
\begin{rk}
	A corollary of the discussion in \Cref{rk:padiccoherent} is that the isomorphisms
	\begin{equation}\label{eq:oneisom}
		\faM{f}{\Lambda}\otimes_{\cF(M,\Lambda)} \faM{f'}{\Lambda}\simeq \faM{f+f'}{\Lambda},f,f'\in p^n\bZ_{(p)}
	\end{equation}
	satisfy a coherence condition (\Cref{rk:generalcoherent}). On the other hand, such a condition is not necessary. More precisely, to conclude \eqref{eq:uniteratedconv} from \eqref{eq:iteratedconv}, we apply the isomorphism \eqref{eq:oneisom} iteratively, $r-1$ times to 
	\begin{equation}
		\faM{s_1-s_0}{\Lambda} \otimes_{\cF(M,\Lambda)} \faM{s_2-s_1}{\Lambda} \dots \otimes_{\cF(M,\Lambda)} \faM{s_r-s_{r-1}}{\Lambda}.
	\end{equation}
	A priori, this iterated multiplication can be done in multiple ways: in $2$ ways if $r=3$, in $5$ ways if $r=4$, in $14$ ways if $r=5$, and so on. 
	Therefore, without a coherence condition, we can obtain several different isomorphisms between $h_{\phi_\alpha^f(L') }$ and $h_{L'}\otimes_{\cF(M,\Lambda)} \faM{f}{\Lambda}$. On the other hand, regardless of the isomorphism \eqref{eq:unitconv} that we use, the conclusions such as \eqref{eq:yonyon}, \eqref{eq:chyonyon} and \eqref{eq:combining} in the proof of \Cref{prop:compareprop} hold. 
\end{rk}
We can now prove the main theorem:
\begingroup
\def\thethm{\ref*{thm:mainthm}}
\begin{thm} Under given assumptions, the rank of $HF(\phi^k(L),L')$ is constant in $k\in\bZ$ except for finitely many $k$.
\end{thm}
\addtocounter{thm}{-1}
\endgroup
\begin{proof}
Recall that we assume that $\phi=\phi_{\alpha}^1$ without loss of generality. By Proposition \ref{prop:compareprop}, the rank of $HF(\phi^k(L),L')$ is equal to the rank of 
\begin{equation}\label{eq:krestr}
H^*\big(h_{L'}\otimes_{\cF(M,\bQ_p)}\fM^{\bQ_p}_{\alpha}|_{t=-k} \otimes_{\cF(M,\bQ_p)}h^L\big)
\end{equation}as long as $k\equiv 0\, (mod\, p^n)$. As $h_{L'}\famotimes_{\cF(M,\bQ_p)}\fM^{\bQ_p}_{\alpha} \famotimes_{\cF(M,\bQ_p)}h^L$ has cohomology that is finitely generated over $\bQ_p\langle t\rangle$, by \Cref{lem:freeres}, there exists a free, finite complex over $\bQ_p\langle t\rangle$ that is quasi-isomorphic to $h_{L'}\famotimes_{\cF(M,\bQ_p)}\fM^{\bQ_p}_{\alpha} \famotimes_{\cF(M,\bQ_p)}h^L$. The structure sheaf of $t=-k$ has a resolution given by $\bQ_p\langle t\rangle \xrightarrow{t+k} \bQ_p\langle t\rangle$. As a result, the restriction \eqref{eq:krestr} of $C\simeq h_{L'}\famotimes_{\cF(M,\bQ_p)}\fM^{\bQ_p}_{\alpha} \famotimes_{\cF(M,\bQ_p)}h^L$ to $t=-k$ is equivalent to  
\begin{equation}
C|_{t=-k}=	C\otimes_{\bQ_p\langle t\rangle} cone(\bQ_p\langle t\rangle \xrightarrow{t+k} \bQ_p\langle t\rangle)\simeq cone(C\xrightarrow{t+k}C).
\end{equation}
One can see $C$ as an unbounded $2$-periodic complex of finitely generated modules, and we have a long exact sequence
\begin{equation}
	\dots\to H^*(C)\xrightarrow{t+k} H^*(C)\to H^*(C|_{t=-k})\to  H^{*+1}(C)\xrightarrow{t+k} H^{*+1}(C)\to \dots  .
\end{equation}
Therefore, $H^*(C|_{t=-k})$ is an extension of $H^*(C)/(t+k)H^*(C)$ and $ker (H^{*+1}(C)\xrightarrow{t+k} H^{*+1}(C))$.

As $\bQ_p\langle t\rangle$ is a PID (\cite[Section 2, Cor 10]{bosch}), the finitely generated module $H^*(C)$ is isomorphic to a direct sum of a finitely generated free module over $\bQ_p\langle t\rangle$ and finitely many modules of the form $\bQ_p\langle t\rangle/(f(t))$, where $f(t)\neq 0$. By Strassman's theorem (\cite{strassmannuber}, \cite[p.62, Theorem 4.1]{casselslocal}, \cite[Theorem 3.38]{padicanalysiskatok}), every such $f(t)$ has finitely many roots. As a result, for all but finitely many $k$, $t+k$ acts invertibly on the torsion part of $H^*(C)$ and $H^{*+1}(C)$. Thus, $ker (H^{*+1}(C)\xrightarrow{t+k} H^{*+1}(C))$ vanishes, and $H^*(C)/(t+k)H^*(C)$ --and therefore $H^*(C|_{t=-k})$-- has the same rank as $H^*(C)$. In other words, the rank of $H^*(C|_{t=-k})$ is constant for all but finitely many $k\in p^n\bZ_{(p)}$. Therefore, the rank of $HF(\phi^k(L),L')$ is constant for $k\in p^n\bZ_{(p)}$ with finitely many possible exceptions.

%

Now choose $f_1,f_2,\dots,f_{p^n-1}\in \bZ_{(p)}$ such that $f_i\equiv i\,(mod\, p^n)$ and $|f_i|$ is small so that Lemma \ref{lem:halg=h} and Lemma \ref{lem:nearbyfloer} hold for $\tilde L=L'$ and $f=-f_i$. In other words, $h^{alg}_{\phi^{-f_i}_\alpha,L'}\simeq h_{\phi^{-f_i}_\alpha(L')}\simeq h_{L'}\otimes_{\cF(M,\Lambda)}\faM{-f_i}{\Lambda}$. In particular, $h_{\phi^{-f_i}_\alpha(L')}$ is definable over $\cF(M,K)$, and therefore over $\cF(M,\bQ_p)$ (Corollary \ref{cor:definable}). Given $f\equiv f_i\,(mod \, p^n)$, one can prove \begin{equation}
h_{\phi^{-f}_\alpha (L') }\simeq h_{L'}\otimes_{\cF(M,\Lambda)}\faM{-f_i}{\Lambda} \otimes_{\cF(M,\Lambda)}\faM{-f+f_i}{\Lambda} 
\end{equation}as a corollary of Lemma \ref{lem:iteratedconv}, similar to Corollary \ref{cor:uniterated}. Then one can simply follow the proof of Proposition \ref{prop:compareprop} to prove that the rank of $HF(\phi_\alpha^{f}(L),L' )$ is the same as the rank of
\begin{equation}
H^*\big(h_{L'} \otimes_{\cF(M,\bQ_p)} 
\fM^{\bQ_p}_{\alpha}|_{t=-f_i} \otimes_{\cF(M,\bQ_p)}\fM^{\bQ_p}_{\alpha}|_{t=-f+f_i}
\otimes_{\cF(M,\bQ_p)}h^L\big),
\end{equation}
i.e. the rank of $\bQ_p\langle t\rangle$-module (or $\bQ_p\langle t/p^n\rangle$-module, after restriction)
\begin{equation}
H^*\big(h_{L'} \famotimes_{\cF(M,\bQ_p)} 
\fM^{\bQ_p}_{\alpha}|_{t=-f_i} \famotimes_{\cF(M,\bQ_p)}\fM^{\bQ_p}_{\alpha}
\famotimes_{\cF(M,\bQ_p)}h^L\big)
\end{equation}
at the point $t=-f+f_i$. This allows one to conclude the rank of $HF(\phi_\alpha^{f}(L),L' )$ is constant among all but finitely many $f$ satisfying $f\equiv f_i\,(mod\,p^n)$. Therefore, the rank of $HF(\phi_\alpha^{f}(L),L' ), f\in\bZ$ is periodic of period (dividing) $p^n$, except for finitely many $f\in\bZ$.

Notice that we can replace $p$ by another prime $p'>2$, and we can conclude that this sequence is periodic of period $(p')^{n'}$ for some $n'$, except for finitely many terms. Since, $p$ and $p'$ are coprime, this implies that the rank of $HF(\phi_\alpha^{f}(L),L' )$ is constant in $ f\in\bZ$ except for finitely many $f$. 
\end{proof}
\begin{rk}\label{rk:denseresult}
The proof actually implies that the rank of $HF(\phi_\alpha^{f}(L),L' )$ is constant in $f\in\bZ_{(pp')}=\{\frac{a}{b}:a,b\in\bZ, p\nmid b,p'\nmid b \}$ except for finitely many $f$. Applying this version to $r\alpha,r\in\bR$ in place of $\alpha$ implies that rank of $HF(\phi_\alpha^{rf}(L),L' )$ is constant except for finitely many $f\in\bZ_{(pp')}$. However, from this, we cannot immediately conclude that the rank of $HF(\phi_\alpha^{f}(L),L' ),f\in\bR$ is constant except for finitely many $f\in\bR$, as the union of all these finite sets corresponding to classes in $\bR/\bZ_{(pp')}$ may still be infinite.
\end{rk}
\begin{rk}
Under the assumption that $L$ and $L'$ are Bohr-Sommerfeld monotone, one can assume that they are two of the fixed split generators $\{L_i\}$ without loss of generality. Therefore, it is possible to calculate the rank of $h_{L'}\famotimes_{\cF(M,\bQ_p)}  \fM^{\bQ_p}_\alpha \famotimes_{\cF(M,\bQ_p)} h^L$ as the dimension of the cohomology of $\fM^{\bQ_p}_{\alpha}(L,L')$ or $\fM^{K}_{\alpha}(L,L')$. In particular, if $L\cap L'=\emptyset$, then this complex is $0$, and Proposition \ref{prop:compareprop} implies that $HF(\phi^k (L),L';\Lambda)=0$ for all $k\in p^n\bZ_{(p)}$. It is possible that the group-like property holds over the base $\bQ_p\langle t\rangle$, which would let one conclude the same for all $k\in \bZ_{(p)}$. Therefore, one would have $HF(\phi^1_\alpha(L),L';\Lambda)=0$ for all $\alpha$. This may seem to contradict the theorem; for instance, by letting $L'=\phi^1_\alpha(L)$. However, $L$ and $\phi^1_\alpha(L)$ cannot be Bohr-Sommerfeld monotone at the same time, i.e. this remark does not apply in this case (unless $[\alpha|_L]=0$, in which case $L$ and $\phi^1_\alpha(L)$ are Hamiltonian isotopic). To see this, assume the converse. Consider a loop $C$ in $L$, and the cylinder $u$ traced by $C$ under the isotopy $\phi_\alpha^s$, $s\in[0,1]$. The area of the cylinder is equal to $-\alpha(C)$, and the calculations in the proof of \cite[Lemma 4.1.5]{wehrheimwoodwardquilted1} imply that it is proportional to the index $I(u)$ (this also follows from of \cite[Definition 4.1.2]{wehrheimwoodwardquilted1} and \cite[Lemma 4.1.5]{wehrheimwoodwardquilted1}). The index is given by the differences of the Maslov indices of totally real bundles $u(0,\cdot)^*T\phi_\alpha^0(L)$ and $u(1,\cdot)^*T\phi_\alpha^1(L)$ with respect to a trivialization of $u^*TM$, and it is easy to see that this index is independent of $s$ for $u(s,\cdot)^*T\phi_\alpha^s(L)$. In other words, $I(u)=0$, which implies $\alpha(C)=0$ for every loop on $L$, i.e. $\alpha|_L$ has a vanishing cohomology class. 
Note that the situation is analogous to the deformations of exact Lagrangians in exact symplectic manifolds: unless $\alpha|_L$ is also exact, such a deformation will not remain exact with respect to the same primitive. As we will explain, one can get rid of the Bohr-Sommerfeld monotonicity assumption on $L$ and $L'$ by passing to a finitely generated extension $K\subset \tilde K$ over which $h_{L'}$ and $h^L$ are definable. However, this does not mean that the modules $h_{L'}$ and $h^L$ (defined in the usual way as in Section \ref{sec:background}, not as Yoneda modules of twisted complexes over $\cF(M,\Lambda)$) have coefficients in $K$ or $\tilde K$.
\end{rk}
\subsection*{Dropping the assumption that $L$ and $L'$ are Bohr-Sommerfeld monotone}
We briefly explain how to drop the assumption that $L$ and $L'$ are Bohr-Sommerfeld monotone, while holding the assumption that they are tautologically unobstructed and have minimal Maslov number at least $3$. The idea is simple: one can represent the modules $h_{L'}$ and $h^L$ as iterated cones/twisted complexes of Yoneda modules of $\{L_i\}$. Since the data to define a twisted complex is finite, these modules are definable over a finitely generated extension of $K$, if not $K$ itself. 

More precisely, represent $L'$ as an element of $tw^\pi (\cF (M,\Lambda))$, i.e. $L'\simeq (\bigoplus_k L_{i_k},\sigma,\pi)$, where $\sigma$ is the differential of the twisted complex and $\pi$ is the idempotent. Write components of $\sigma$ and $\pi$ as linear combinations of the generators of $CF(L_i,L_j)$. 
One needs to add only finitely many elements from $\Lambda$ to the subfield $K$ to include the coefficients of these linear expressions. In other words, there exists a finitely generated extension $K\subset \tilde K$ such that $\sigma$ and $\pi$, and hence this twisted complex is defined over $\cF(M,\tilde K)$--- the category obtained from $\cF(M,K)$ via base change along $K\to\tilde K$. Let $h'_{L'}$ denote the image of this twisted complex under Yoneda embedding. By construction, $h'_{L'}$ turns into a module quasi-isomorphic to $h_{L'}$, when we base change under $\tilde K\subset\Lambda$. Denote this module by $h'_{L'}$ as well. Similarly, $h^L$ is also definable over a finitely generated extension of $K$, i.e. by further extending $\tilde K$ by finitely many elements, one can ensure that there exists a left $\cF(M,\tilde K)$-module $h'^{L}$ that becomes quasi-isomorphic to $h^{L}$ after base change along $\tilde K\subset \Lambda$.
\begin{rk}
Analogous to \Cref{rk:nonisomreduction}, the reductions of two Hamiltonian isotopic Lagrangians to a smaller field are not necessarily quasi-isomorphic. 
\end{rk}
Since $\tilde K$ is finitely generated and countable, one can find a finite (not only finitely generated) extension $\bQ_p'$ of $\bQ_p$ and a map $\tilde \kappa :\tilde K\to\bQ_p'$ extending $\kappa:K\to\bQ_p$. For this, one can first extend $\kappa$ to a map to $\bQ_p$ from a maximal purely transcendental extension of $K$ inside $\tilde K$ (by choosing a set of elements of $\bQ_p$ that are algebraically independent over $\kappa(K)$). Then, $\tilde K$ is finite over this extension of $K$, and there exists a finite extension $\bQ_p'$ of $\bQ_p$ and a map $\tilde{\kappa}:\tilde K\to \bQ_p'$, whose restriction to $K$ is equal to $\kappa$. 

The field $\bQ_p'$ carries a unique discrete valuation extending that on $\bQ_p$ by \cite[Theorem 9.5]{shimuraarith}, and one can use the base change under $\bQ_p\langle t\rangle \to \bQ_p'\langle t\rangle$ to obtain a family $\fM^{\bQ_p'}_{\alpha}$ over the latter that is group-like over $\bQ_p'\langle t/p^n\rangle$. The proof of Proposition \ref{prop:compareprop} still applies and we have
\begin{equation}
dim_\Lambda\big(HF(\phi^k(L),L')\big)=dim_{\bQ_p'}\big( H^*\big(h'_{L'}\otimes_{\cF(M,\bQ_p')}\fM^{\bQ_p'}_{\alpha}|_{t=-f} \otimes_{\cF(M,\bQ_p')}h'^L\big) \big)
\end{equation}
for all $k\in p^n\bZ_{(p)}$. This dimension is constant for all but finitely many $k$ as before.

Replacing $L'$ by $\phi^{-i}(L')$, where $i=0,1,\dots  p^n-1$, we see that $dim_\Lambda\big(HF(\phi^k(L),L')\big)$ is constant for all but finitely many $k\in i+p^n\bZ_{(p)}$. Therefore, $dim_\Lambda\big(HF(\phi^k(L),L')\big)$ is $p^n$ periodic for all but finitely many $k\in \bZ$, and by replacing $p$ by another prime, we see that this dimension is actually constant among  $k\in\bZ$ except for finitely many. This concludes the proof of Theorem \ref{thm:mainthm} without the Bohr-Sommerfeld monotonicity assumption on $L$ and $L'$.
\begin{rk}
Notice that the important difference as one drops the Bohr--Sommerfeld monotone assumption is that we are now able to directly replace $L'$ by any $\phi^{-i}(L')$ (which was not the case before, as $\phi^{-i}(L')$ is not necessarily Bohr--Sommerfeld monotone). Previously we had circumvented this issue by replacing $i$ with $f_i\in i+p^n\bZ_{(p)}$ such that $|f_i|$ is small, and replacing $h_{\phi^{-f_i}(L')}$ by $h^{alg}_{\phi^{-f_i}_\alpha,L'}$, that is definable over $K$ and equivalent to $h_{\phi^{-f_i}(L')}$ over $\Lambda$. 
%
\end{rk}

\section{Generic $\alpha$ and the proof of Theorem \ref{thm:generic}}\label{sec:generic}
In this section, we explain a way to construct group-like $p$-adic families of bimodules (i.e. ``analytic $\bZ_p$-actions'') without the assumption of monotonicity of $M$ at the cost of assuming $\alpha$ is generic. Using this, we will deduce Theorem \ref{thm:generic}.

The main reason we had to assume monotonicity was that, we had no way of ensuring convergence for an infinite series of the form (\ref{eq:padicstructuremaps}) as well as (\ref{eq:padicmapformula}) in $\bQ_p$. One could try to choose the map $\kappa:K\to\bQ_p$ so that $\kappa(T^{E(u)})\in p\bZ_p$ and $\kappa(T^{\alpha([\partial_h u])})\in 1+p\bZ_p$, but such a map may not exist if there are algebraic relations between the elements of the form $T^{E(u)}$ and $T^{\alpha([\partial_h u])}$. On the other hand, for generic $[\alpha]\in H^1(M,\bR)$, there are no such algebraic relations. Therefore, one can define the embedding of the field of definition into $\bQ_p$ such that the images of $T^{E(u)}$ converge in $p$-adic topology, just like $T^{E(u)}$ converging in $T$-adic topology by Gromov compactness.
 
Our construction works in the general setting under \Cref{assumption:generic}. In other words, we assume that $\pi_2(M)=0$, $\cF(M,\Lambda)$ is an (uncurved, $\bZ$ or $\bZ/2\bZ$-graded) smooth and proper $A_\infty$-category over $\Lambda$, , and it is split generated by tautologically unobstructed Lagrangians $L_1,\dots, L_m$. Also, $L,L'$ are tautologically unobstructed Lagrangians with brane structures that bound no discs of Maslov index $2$ or less (so they define objects of the Fukaya category).
We denote the span of $L_1,\dots,L_m$ by $\cF(M,\Lambda)$. Two natural examples of such $M$ are an elliptic curve and the product of two elliptic curves. Let $L,L'$ be two branes satisfying Assumption \ref{assumption:generic}, which are represented by elements of $tw^\pi(\cF(M,\Lambda))$ by the generation assumption. Assume that $L_i,L,L'$ are pairwise transverse. 

Fix perturbation data. Our first goal is to embed the additive semi-group spanned by all $E(u)>0$ into the multiplicative semi-group  $p\bZ_p$. For this, we prove that there exists a discrete free submonoid of $\bR_{\geq 0}$ that contains the topological energy of all pseudo-holomorphic curves in $M$ with boundary on $L_i,L,L'$:
\begin{lem}\label{lem:freeenergymonoid}
For appropriate choices of perturbation data, there exist rationally independent elements 
\begin{equation}
E_1,\dots,E_n\in \bR_{>0}\cap \omega_M (H_2(M, L\cup L'\cup\bigcup_i L_i;\bQ))
\end{equation}	
such that the monoid spanned by them contains the topological energy of every non-constant pseudo-holomorphic curve with boundary on $L\cup L'\cup\bigcup_i L_i$.
\end{lem}
\begin{proof}
For simplicity, we work with curves with boundary on a single Lagrangian $L$. 
Let $\omega_1,\dots ,\omega_{n'}$ denote closed $2$-forms that are obtained by small perturbations of $\omega_M$ and that satisfy:
\begin{enumerate}
\item\label{condbasis:perturbtame} the forms $\omega_i$ are still symplectic and tamed by the almost complex structure term of the perturbation data,
\item\label{condbasis:rationalbasis} $\omega_i|_L=0$ for all $i$ and $\{[\omega_i] \}$ form a basis of $H^2(M,L;\bQ)$ (in particular they are rational),
\item\label{condbasis:convexhull} $[\omega_M]\in H^2(M,L;\bR)$ is in the convex hull of the rays generated by $[\omega_i]$. 
\end{enumerate}
%
To see that such $\{\omega_i\}$ exists, choose closed $2$-forms $\theta_1,\theta_2,\dots, \theta_{n'}$ representing a basis of $H^2(M,L;\bR)$ and containing $\omega_M$ is their span $\bigoplus_i \bR.\theta_i\subset \Omega^2(M)$. The space $\bigoplus_i \bR.\theta_i$ has a canonical identification with $H^2(M,L;\bR)$. Note that small perturbations of $\omega_M$ within $\bigoplus_i \bR.\theta_i$ remain symplectic and an almost complex structure taming $\omega_M$ tames the small perturbation as well. Therefore, assuming the variation of almost complex structures is bounded, we can guarantee that the forms in a small neighborhood of $\omega_M$ in $\bigoplus_i \bR.\theta_i$ are still tamed by the same almost complex structures. Choosing $\{\omega_i\}$ from this neighborhood guarantees \eqref{condbasis:perturbtame}. It is clear that one can choose such $\{\omega_i\}$ that is a basis and such that $\omega_M$ is in the interior of their convex hull. Moreover, these two properties are open conditions, and as the rational forms are dense in $\bigoplus_i \bR.\theta_i\cong H^2(M,L;\bR)$, we can also assume that each $\omega_i$ is rational. Therefore, \eqref{condbasis:rationalbasis} and \eqref{condbasis:convexhull} also hold. 

Note that the topological energy is still bigger than the geometric energy, as we do not change the Hamiltonian term (cf. \cite[(7.6)]{abousei}, which holds for the tame almost complex structures as well). In particular, the topological energy of a pseudo-holomorphic curve is still positive, and as before the topological energy of such a curve representing class $\beta\in H_2(M,L;\bZ)$ can be calculated as  $\omega_i(\beta)$.
%
Furthermore, there exists a natural number $N$ such that $\omega_i (H_2(M,L;\bZ))\subset \frac{1}{N}\bZ$. Choose a common $N$ for all $\omega_i$. 

Consider the set 
\begin{equation}\label{eq:somebetas}
\{\beta\in H_2(M,L;\bQ) :\omega_i(\beta)>0\text{ and } \omega_i(\beta)\in\frac{1}{N}\bZ \text{ for all }i  \}\cup\{0\}.
\end{equation}
If we use $(\omega_i)_i$ to identify $H_2(M,L;\bQ)$ with $\bQ^{n'}$, then (\ref{eq:somebetas}) corresponds to $(\frac{1}{N}\bZ_+)^{n'}\cup\{0\}$. Therefore, (\ref{eq:somebetas}) is a finitely generated submonoid of $H_2(M,L;\bQ)$. It follows from the discussion above that \eqref{eq:somebetas} contains the classes of all pseudo-holomorphic curves. Also, as $\omega_M$ can be represented as a positive linear combination of $\omega_i$, the image of (\ref{eq:somebetas}) under $\omega_M$ is a submonoid of $\bR_{\geq 0}$, and it is finitely generated as (\ref{eq:somebetas}) is. We conclude the proof by applying Lemma \ref{lem:fgmonoid} below to the image of (\ref{eq:somebetas}) under $\omega_M$. 
\end{proof}
\begin{lem}\label{lem:fgmonoid}
Every finitely generated additive submonoid of $\bR_{\geq 0}$ can be extended in its rational span to a submonoid generated by positive, rationally independent elements. 
\end{lem}
\begin{proof}
We proceed by induction on the number of generators. Assume that the statement holds when there are less than $n$ generators. Consider an additive submonoid of $\bR_{\geq 0}$ generated by $x_1,\dots ,x_n> 0$. By induction hypothesis for $x_1,\dots,x_{n-1}$, we can find a monoid in their rational span generated by rationally independent elements $y_1,\dots y_k>0$ that contains $x_1,\dots,x_{n-1}$. If $x_n$ is rationally independent of $\{y_i\}$, then $y_1,\dots,y_k,x_n$ is a rationally independent set of generators of a monoid containing $x_1,\dots,x_{n}$. 

Otherwise, one can write $x_n$ as $x_n=a_1y_{i_1}+\dots +a_qy_{i_q}-b_1y_{j_1}-\dots- b_ry_{j_r}$, where $a_i\in\bQ_+,b_j\in \bQ_{\geq 0}$, such that the sets $\{y_{i_1}\dots,y_{i_q}\}$ and $\{y_{j_1},\dots,y_{j_r}\}$ form a partition of $\{y_1,\dots,y_k\}$ (i.e. they are disjoint, and their union contains all $y_i$). Let $b_1y_{j_1}+\dots+ b_ry_{j_r}$ be denoted by $\eta$. Since $x_n>0$, there exists $\lambda_1,\dots ,\lambda_q\in\bQ_{\geq 0}$ such that $\lambda_1+\dots +\lambda_q=1$, and $a_hy_{i_h}-\lambda_h \eta> 0$ (to see this, assume that $\eta>0$ and $a_hy_{i_h}=\iota_h \eta$, where $\iota_h >0$. Thus, $\sum \iota_h >1$, choose $\lambda_h$ to be a rational partition of $1$ such that $\iota_h-\lambda_h>0$). We can find a new basis that consists of positive rational multiples of $y_{j_g}$ and $a_hy_{i_h}-\lambda_h \eta$. For instance, let $\tau_1\in\bZ_+$ be the product of numerators of all $a_h$, and let $\tau_2\in\bZ_+$ be the product of denominators of all $\lambda_h$ and $b_g$. Then, 
the monoid spanned by positive rationally independent elements $\frac{1}{\tau_1}(a_hy_{i_h}-\lambda_h \eta),h=1,\dots, q$, $\frac{1}{\tau_1\tau_2}y_{j_g}, g=1,\dots,r $ contains $y_1,\dots,y_k,x_n$, and thus it contains $x_1,\dots,x_n$.
This finishes the proof.
\end{proof}
\begin{rk}
Gromov compactness implies that the set of energies of non-constant pseudo-holomorphic curves is a discrete subset of $\bR_+$; however, we are unable to use this statement to give a simpler proof of Lemma \ref{lem:freeenergymonoid}. The main difficulty is the finite generation statement, it is not true that an additive submonoid of $\bR_{\geq 0}$ that lies in a finitely generated subgroup of $\bR$ is always contained in a finitely generated submonoid of $\bR_{\geq 0}$. For instance, consider the monoid generated by $1$ and $\sqrt{2}$. It is a discrete monoid, and we can expand it to another discrete monoid by adding an element of the set $\bZ+\sqrt{2}\bZ$ that is outside the original monoid (such as $3\sqrt{2}-1$). Then, we can go on by adding another element of $\bZ+\sqrt{2}\bZ$ from outside the monoid, to expand it further. We continue in this way, and the union of this nested sequence of monoids is a monoid that is not finitely generated. If we assume that the element added at the $n^{th}$ step is larger than $n$, then this monoid is discrete. 
\end{rk}
Let $\cE_+\subset\bR_{\geq 0}$ denote the monoid spanned by $E_1,\dots,E_n$. Since these generators are rationally independent, there is no algebraic relation between $T^{E_i}$. 
\begin{defn}
$\alpha$ is called \emph{generic} if $\alpha(H_1(M,\bZ))$ intersects the group generated by $\cE_+$ only at $0$.
\end{defn}
This condition is weaker than \begin{equation}
\alpha(H_1(M,\bZ))\cap \omega_M (H_2(M, L\cup L'\cup\bigcup_i L_i;\bQ))=\{0\}.
\end{equation}
As a result, there is an abundance of generic $1$-forms. 

Our next goal is to find a field $K_{g}\subset\Lambda$, that contains all the coefficients of $A_\infty$-structure maps, and that we can embed into $\bQ_p$. Using this embedding, we obtain a category over $\bQ_p$, which we still denote by $\cF(M,\bQ_p)$. We also want to define a similar $p$-adic family of $\cF(M,\bQ_p)$-bimodules over $\bQ_p\langle t\rangle$ such that the restriction to $t=f$ comes from a natural $\cF(M,K_g)$-bimodule. 

Fix a prime $p>2$. For a generic $\alpha$, there are no non-trivial algebraic relations between various $T^{\alpha(C)}$ and $T^{E_i}$. Consider the Novikov series \begin{equation}\label{eq:sumgeneric}
\sum_{k} a_k T^{E(k)+f\alpha(k) },
\end{equation} where $a_k\in\bZ_{(p)}$, $E(k)\in\cE_+$, $\alpha(k)\in \alpha (H_1(M;\bZ))$, $f\in\bZ_{(p)}$ satisfying
\begin{enumerate}
	\item\label{cond:egoesinfty} $E(k)+f\alpha(k) \to \infty$,
	\item $E(k)\to\infty $. 
\end{enumerate}
The first condition is for convergence in $\Lambda$. Denote the set of such series by $R_{big}\subset \Lambda$. Then, 
$R_{big}$ is a subring of $\Lambda$. To write a ring homomorphism $R_{big}\to \bZ_p$, choose a basis $\{\alpha_j:j=1,\dots l \}\subset\alpha(H_1(M;\bZ) )\subset \bR$. Choose a set of algebraically independent elements of $\bZ_p$ corresponding to each $E_i$ and each $\alpha_j$. Call these elements $\kappa_i$ and $\kappa_j'\in\bZ_p$. Define $\kappa_{big}:R_{big}\to\bZ_p$ as follows: we send
\begin{equation}
 T^{\alpha_j}\to 1+p\kappa_j', T^{E_i}\to p\kappa_i, T^{f\alpha_j}\to (1+p\kappa_j')^f.
\end{equation}
The image of (\ref{eq:sumgeneric}) converges in $\bZ_p$ as $E(k)\to\infty$; therefore, one can extend it to $R_{big}$. More precisely, the total number of $E_i$ in $E(k)\in\cE_+$ goes to $\infty$; therefore, the p-adic valuation of their image goes to $\infty$. It is possible to show that the coefficients defining $\cF(M,\Lambda)$ are in $R_{big}$; therefore, one can define the Fukaya category over $R_{big}$ and base change along $\kappa_{big}$, giving us a category over $\bZ_p$ and eventually over $\bQ_p$. Then, it is easy to construct $p$-adic families of bimodules, and prove similar properties such as group-like property. On the other hand, when making dimension comparisons in the previous sections, we implicitly took advantage of the flatness of the map $K\to\bQ_p$. We do not know this for $R_{big}\to \bQ_p$, it is not always true that the map $\kappa_{big}$ is injective, and hence extends to a map from the field of fractions of $R_{big}$ to $\bQ_p$. Presumably, generic choice of $\kappa_i$, $\kappa_j'$ can ensure injectivity; however, we have not been able to prove this. Instead, we will construct a countable subring of $R_{big}$ that satisfies this property. First observe,
\begin{lem}
Fix a non-zero element $w=\sum_{k} a_k T^{E(k)+f\alpha(k) } \in R_{big}$. The set of $(\kappa_1,\dots ,\kappa_n,\kappa_1',\dots ,\kappa_l')\in\bZ_p^{n+l}$ such that $\kappa_{big}(w)=0$ (for $\kappa_{big}:R_{big}\to  \bZ_p$ corresponding to this particular choice of $\kappa_i$, $\kappa_j'$) has measure zero with respect to a Haar measure on $\bZ_{p}^{n+l}$.	
\end{lem}
\begin{proof}
As $E_i, \alpha_j$ are independent, we obtain an element $F_w\in \bQ_p\langle X_1,\dots ,X_n,Y_1,\dots,Y_l\rangle$ by replacing $T^{E_i}$ terms within $w$ by $pX_i$ and $T^{f\alpha_j}$ terms by $(1+pY_j)^f$ (the condition \eqref{cond:egoesinfty} guarantees that $F_w$ converges on the unit polydisc, i.e. $F_w\in \bQ_p\langle X_1,\dots ,X_n,Y_1,\dots,Y_l\rangle$). Moreover, $F_w\neq 0$ as $w\neq 0$. The set of $(\kappa_1,\dots ,\kappa_n,\kappa_1',\dots ,\kappa_l')\in\bZ_p^{n+l}$ such that $\kappa_{big}(w)=0$ is the same as the zero set of $F_w$. 

Zero set of any non-vanishing analytic function on a $p$-adic polydisc has measure zero. This follows from Strassman's theorem for one variable, and the general case can be proven by induction on the number of variables. Assume that the claim holds for functions with less than $N$ variables. Given $0\neq F\in \bQ_p\langle t_1,\dots, t_N\rangle$, $F(t_1,\dots ,t_{N-1},a)\neq 0$ for all but finitely many $a\in\bZ_p$ (by Strassman's theorem applied to non-Archimedean normed ring $\bQ_p\langle t_1,\dots, t_{N-1}\rangle$). By induction hypotheses, the zero set of any non-vanishing $F(t_1,\dots ,t_{N-1},a)$ has measure zero in $\bZ_{p}$. Therefore, by Fubini's theorem (\cite[Theorem 8.8]{rudinanalysis}), the zero set of $F$ has measure zero. This completes the proof. 
\end{proof}
\begin{cor}
Fix a countable subring $R$ of $R_{big}$. For generic $\kappa_i,\kappa_j'$, the function $\kappa_{big}|_R:R\to\bZ_p$ is injective, and therefore extends to a field homomorphism from its fraction field to $\bQ_p$.
\end{cor}
We will let $R_g\subset R$ be a countable ring that contains the series defining the Fukaya category $\cF(M,\Lambda)$, the modules $h^L,h_{L'}$, the analogues of the bimodules $\faM{f}{\Lambda}$, and the map (\ref{eq:novconvto}). This will allow us to define the field of definition $K_g$ as the field of fractions of $R_g$ and $\kappa_g:K_g\to \bQ_p$ as the extension of $\kappa_{big}|_{R_g}$. Then, one can prove analogues of Proposition \ref{prop:compareprop} and Theorem \ref{thm:mainthm}, following the same reasoning. More precisely, 
\begin{defn}\label{defn:genericfield}
Let $R_g$ be the subring of $R_{big}$ containing $\bZ_{(p)}$ and the following series	
\begin{enumerate}
	\item\label{cpcond:variables2} the elements $T^{E_i}$ for $i=1,\dots, n$, and $T^{\alpha_j}$ for $j=1,\dots,l$	,
	\item\label{cpcond:ainfty2} the signed sums $\sum \pm T^{E(u)}$ where $u$ ranges over pseudo-holomorphic curves with fixed inputs and one fixed output, and with boundary on $L_i$, $L$, and $L'$ ,
	\item\label{cpcond:discswithfixed2} the signed sums $\sum \pm T^{E(u)}$ where $u$ ranges over pseudo-holomorphic curves with fixed inputs, one fixed output and fixed $[\partial_h u]\in H^1(M;\bZ)$, and with boundary on $L_i$,
	\item\label{cpcond:halgf2} the signed sums $\sum \pm T^{E(u)+f\alpha([\partial_{\tilde L} u]) }$ where $u$ ranges over the same set of pseudo-holomorphic curves as (\ref{eq:structurealgyoneda}) with fixed inputs and one fixed output, where $\tilde L=L'$, and $f\in\bZ_{(p)}$ is sufficiently small so that the series converge in $\Lambda$,
	\item\label{cpcond:family2} the signed sums $\sum \pm T^{E(u)+f\alpha([\partial_h u]) }$ where $u$ ranges over the same set of pseudo-holomorphic curves as (\ref{eq:structurenovikov}) and (\ref{eq:padicstructuremaps}) with fixed inputs and one fixed output, and $f\in\bZ_{(p)}$ is sufficiently small so that the series converge in $\Lambda$,
	\item\label{cpcond:gplikemap2} the signed sums $\sum\pm T^{E(u)+f_1\alpha([\partial_1 u])+f_2\alpha([\partial_2 u])}$ where $u$ ranges over the same set of pseudo-holomorphic curves as (\ref{eq:padicmapformula}) with fixed inputs and one fixed output, and $f_1,f_2\in\bZ_{(p)}$ are sufficiently small so that the series converge in $\Lambda$. 
\end{enumerate}
Let $K_g$ denote the fraction field of $R_g$ and let $\kappa_g:K_g\to\bQ_p$ denote the extension of $\kappa_{big}|_{R_g}$.
\end{defn}
Condition (\ref{cpcond:ainfty2}) ensures that the coefficients $\sum\pm T^{E(u)}$ of the $A_\infty$-structure, and of the modules $h^L,h_{L'}$ are in $R_g$. Thanks to condition (\ref{cpcond:halgf2}), the module analogous to $h^{alg}_{\phi_{\alpha}^f,L'}$ is definable over $R_g$ and $K_g$. Similarly, condition (\ref{cpcond:family2}) guarantees that the analogous bimodule to $\faM{f}{\Lambda}$ is definable over $R_g$ and $K_g$, and condition (\ref{cpcond:gplikemap2}) is for the definability of the bimodule homomorphisms
\begin{equation}
	\faM{f_1}{\Lambda}\otimes_{\cF(M,\Lambda)} \faM{f_2}{\Lambda}\to \faM{f_1+f_2}{\Lambda}.
\end{equation}
Finally, the purpose of condition (\ref{cpcond:discswithfixed2}) is that the coefficients of the family $\fM^\Lambda_{\alpha}$ lie in $R_g$ and in $K_g$ so that one can define the family $\fM^{K_g}_{\alpha}$ (analogous to $\fM^{K}_{\alpha}$ defined before). This condition is not strictly necessary, but it is more convenient to include it. 
\begin{note}
	The signs in conditions (\ref{cpcond:ainfty2})-(\ref{cpcond:gplikemap2}) come from the orientation of the corresponding moduli space of discs, and they are identical to signs in the analogous sums defined before. For instance, the signs of the summands of (\ref{cpcond:ainfty2}) (and its subsum (\ref{cpcond:discswithfixed2})) are the same as the signs in the sums defining the $A_\infty$-structure. Similarly, the signs in (\ref{cpcond:family2}) are the same as the signs in (\ref{eq:structurenovikov}) and (\ref{eq:padicstructuremaps}). As mentioned in Note \ref{note:signs} these signs are standard and omitted throughout the paper. 
\end{note}
\begin{note}
The convergence assumptions in (\ref{cpcond:halgf2})-(\ref{cpcond:gplikemap2}) are needed, since we are no longer in the monotone setting, and these sums are not finite. However, using Fukaya's trick (or reverse isoperimetric inequalities \cite{gromansolomonreverse,duvalreverse},\cite[Appendix A]{abouzaidhmscorrections}) for $\tilde L$, as well as $\{L_i\}$, one can show that the convergence holds for sufficiently small non-zero $|f|$, resp. $|f_1|$, $|f_2|$. 	
\end{note}

After this setup, it is possible to follow the steps of the proof of \Cref{thm:mainthm}. We describe these steps in less detail. First, we define Novikov and $p$-adic families. Because of the convergence issue, we let the base of a Novikov family be the ring $\Lambda\{z^\bR \}$ defined in \Cref{sec:families} and further analyzed in \Cref{appendix:semicont}. In other words, \begin{equation}
\Lambda\{z^\bR\}:=\bigg\{\sum a_r z^r:\text{, where }r\in\bR, a_r\in\Lambda \bigg\},
\end{equation}
where the series satisfy the convergence condition $val_T (a_r)+r\nu \to\infty$, for all $\nu\in[a,b]$. Here, $a<0<b$ are fixed numbers with small absolute values. Recall that the isomorphism type of the ring does not depend on $a$ and $b$, but the coefficients of series we will encounter will belong to this ring for small $|a|$ and $|b|$ only (c.f. convergence in family Floer cohomology). We like to think of $\Lambda\{z^\bR \}$ as a non-Archimedean analogue of a small closed interval containing $0$. Let $K_g\{z^\bR\}$ be the set of elements of $\Lambda\{z^\bR \}$ where the coefficients of all $z^r$-terms are in $K_g$.
\begin{defn}\label{defn:genericnovikovfamily}
Let $\fM^{K_g}_\alpha$ denote the family of $\cF(M,K_g)$-bimodules defined via 
\begin{equation}\label{eq:familygenericchain}
(L_i,L_j)\mapsto \fM^{K_g}_\alpha (L_i,L_j)=CF(L_i,L_j;K_g)\otimes_{K_g} K_g\{z^\bR\} ,
\end{equation} and with structure maps 
\begin{equation}\label{eq:structuregenericnovikov}
(x_k,\dots,x_1|x|x_1',\dots x_l')\mapsto \sum \pm T^{E(u)}z^{\alpha([\partial_h u])}.y ,
\end{equation}
where the sum ranges over pseudo-holomorphic discs as in (\ref{eq:structurenovikov}). 
Define the family of $\cF(M,\Lambda)$-bimodules $\fM^\Lambda_\alpha$ by replacing $K_g$ with $\Lambda$.
\end{defn}
For the series $\sum \pm T^{E(u)}z^{\alpha([\partial_h u])}$ to be in $\Lambda\{z^\bR \}$, one needs  
\begin{equation}
E(u)+f\alpha([\partial_h u])\to\infty
\end{equation}
for small $|f|$. This holds, as can be shown using Fukaya's trick. When we restrict the family $\fM^{\Lambda}_\alpha$ to $z=T^f$, we implicitly assume that $|f|$ is small and convergence holds. For $f\in\bZ_{(p)}$, the coefficients defining $\fM^{\Lambda}_\alpha|_{z=T^f}$ are in $K_g$ due to condition (\ref{cpcond:family2}), and we denote the corresponding $\cF(M,K_g)$ bimodule by $\fM^{K_g}_\alpha|_{z=T^f}$. In other words, it is the bimodule defined by replacing $z^{\alpha([\partial_h u])}$ in \eqref{eq:structuregenericnovikov} by $T^{f\alpha([\partial_h u])}$.
\begin{note}
Observe that we have not defined $\fM^{K_g}_\alpha|_{z=T^f}$ via base change along a map $K_g\{z^\bR\}\to K_g$. This is because the natural map $\Lambda\{z^\bR\}\to \Lambda$ sending $z^r\mapsto T^{fr}$ does not restrict to a map $K_g\{z^\bR\}\to K_g$ ($K_g$ is too small). We can get around this by restricting to a small subring $K_g\{z^\bR\}_{small}$ of $K_g\{z^\bR\}$ that contains enough elements to define $\fM^{K_g}$, but also small enough so that the evaluation $K_g\{z^\bR\}_{small}\to K_g$ at $z=T^f$ is well-defined. We do not need this, as we work with the individual bimodules $\fM^{K_g}_\alpha|_{z=T^f}$. 
\end{note}

We also define:
\begin{defn}
Let $\fM^{\bQ_p}_\alpha$ be the $p$-adic family of $\cF(M,\bQ_p)$ modules defined via 
\begin{equation}\label{eq:padicfamilygenericchain}
(L_i,L_j)\mapsto \fM^{\bQ_p}_\alpha (L_i,L_j)=CF(L_i,L_j)\otimes_{K_g} \bQ_p\langle t\rangle, 
\end{equation}and with the structure maps 
\begin{equation}\label{eq:padicgenericstructuremaps}
(x_1,\dots, x_k|x|x_1',\dots ,x_l')\mapsto\sum\pm \kappa_g(T^{E(u)}) \kappa_g(T^{\alpha([\partial_h u])})^t.y.
\end{equation}
\end{defn}
The convergence of the series (\ref{eq:padicgenericstructuremaps}) is guaranteed by the fact that $val_p(\kappa_g(T^{E(u)}))\to \infty$. The following lemma follows from the definitions:
\begin{lem}
Given $f\in\bZ_{(p)}$ such that $|f|$ is small, the bimodule $\faM{f}{K_g}$ becomes	$\fM^{\bQ_p}_\alpha|_{t=f}$ under base change along $\kappa_g:K_g\to \bQ_p$. 
\end{lem}
The analogue of Lemma \ref{lem:grouplikenovikov} also holds:
\begin{lem}\label{lem:grouplikegenericnovikov}
For $f_1,f_2\in\bR$ such that $|f_1|,|f_2|$ are small, $\faM{f_1+f_2}{\Lambda}\simeq \faM{f_1}{\Lambda}\otimes_{\cF(M,\Lambda)} \faM{f_2}{\Lambda}$. The same statement holds for $\fM^{K_g}_{\alpha}|_{z=T^{f_i}}$ if $f_1,f_2\in \bZ_{(p)}$.
\end{lem}
\begin{proof}
The proof for $\fM_{\alpha}^\Lambda$ is identical. To prove the statement about $\fM_{\alpha}^{K_g}$, one needs to ensure the coefficients 
$\sum\pm T^{E(u)} T^{f_1\alpha([\partial_1 u])} T^{f_2\alpha([\partial_2 u])}$ defining the map 
\begin{equation}
\faM{f_1}{\Lambda}\otimes_{\cF(M,\Lambda)} \faM{f_2}{\Lambda}\to \faM{f_1+f_2}{\Lambda}
\end{equation}
are in $K_g$. This follows from the construction of $K_g$ (condition \eqref{cpcond:gplikemap2}).
\end{proof}
Similarly, we have the following analogue of Proposition \ref{prop:grouplikepadic}:
\begin{lem}
For $n\gg 0$, the restriction $\fM^{\bQ_p}_\alpha|_{\bQ_p\langle t/p^n\rangle}$	is group-like. 
\end{lem}
\begin{proof}
The proof of Proposition \ref{prop:grouplikepadic} applies verbatim. To write the map (\ref{eq:convolutiontodiag}), one has to ensure convergence of the series analogous to (\ref{eq:padicmapformula}). This also follows from the fact that $val_p(\kappa_g(T^{E(u)}))\to \infty$.
\end{proof}
As observed, the modules $h_{L'}$ and $h^L$ are defined over $K_g$. Similarly, given $f\in\bZ_{(p)}$, the deformation $h^{alg}_{\phi_\alpha^f, L'}$ is defined over $K_g$, whenever $|f|$ is small and the convergence holds for the defining series. 
Therefore, there exists a complex
\begin{equation}
h_{L'}\famotimes_{\cF(M,\bQ_p)} \fM^{\bQ_p}_\alpha|_{\bQ_p\langle t/p^n\rangle} \famotimes_{\cF(M,\bQ_p)} h^L
\end{equation}
of $\bQ_p\langle t/p^n\rangle$-modules, which has finitely generated cohomology and whose cohomological rank at $t=-f\in p^n\bZ_{(p)}$ is the same as $HF(\phi^f_\alpha(L),L';\Lambda)$. The proof of this is identical to Proposition \ref{prop:compareprop}. By the same reasoning, the rank of $HF(\phi^f_\alpha(L),L';\Lambda)$ is constant in $f\in p^n\bZ_{(p)}$ with finitely many possible exceptions. Replacing $h_{L'}$ by $h_{\phi_{\alpha}^{f_i},L'}^{alg}$, where $f_i\in i+p^n\bZ_{(p)}$ for $i=0,\dots,p^n-1$ and $|f_i|$ is small, we conclude the same for $HF(\phi^f_\alpha(L),L';\Lambda)$, for $f\in i+p^n\bZ_{(p)}$. In other words, $HF(\phi_{\alpha}^k(L),L';\Lambda )$ has $p^n$-periodic rank in $k\in\bZ$ with finitely many exceptions, and since we have no restriction on $p$, we can replace it with another prime to conclude $p'^{n'}$-periodicity, and thus constancy in $k$. In other words, we have:
\begingroup
\def\thethm{\ref*{thm:generic}}
\begin{thm} Suppose Assumption \ref{assumption:generic} holds. Given generic $\phi\in Symp^0(M,\omega_M)$, the rank of $HF(\phi^k(L),L';\Lambda )$ is constant in $k\in\bZ$ with finitely many possible exceptions. 
\end{thm}
\addtocounter{thm}{-1}
\endgroup
\begin{rk}
It may be possible to prove a version of Theorem \ref{thm:generic} by the following method: it is expected that $HF(\phi_{\alpha}^f(L),L';\Lambda), f\in\bR$ forms a coherent analytic sheaf over $\bR$. Therefore, one can attempt to check that the rank of such a sheaf jumps only at a discrete set, and for a generic $r\in\bR$, the sequence $\{r.k:k\in\bZ\}$ intersects this set only at finitely many points. Hence, the rank of $HF(\phi_{\alpha}^{r.k}(L),L';\Lambda )$ is constant in $k\in\bZ$, with finitely many possible exceptions, for generic $r$. On the other hand, using this method, we do not see how to describe the generic set more concretely. 
\end{rk}

\section{Applications}\label{sec:applications}
In this section, we prove \Cref{cor:entropy} and \Cref{cor:seidelsconj}. Let $M,L,L'$ and $\phi$ be as in the setting of either of \Cref{thm:mainthm} or \Cref{thm:generic}. First,
\begingroup
\def\thethm{\ref*{cor:entropy}}
\begin{cor}
The action of $\phi$ on the derived Fukaya category $D^\pi\cF(M,\Lambda)$ has vanishing categorical entropy.
\end{cor}
\addtocounter{thm}{-1}
\endgroup
\begin{proof}
As $\cF(M,\Lambda)$ is smooth and proper, and $\{L_i\}$ generates the derived Fukaya category, the categorical entropy of the equivalence induced by $\phi$ it can be computed via the formula 
\begin{equation}
	\limsup_k\frac{1}{k}\log\Bigg(\sum_{i,j} dim HF(\phi^k(L_i),L_j;\Lambda)\Bigg).
\end{equation}
This is a consequence of \cite[Theorem 2.6]{dynsysandcats} (loc.\ cit.\ uses a single generator $G$ to state their theorem, we can let $G=\bigoplus_i L_i$). It follows from \Cref{thm:mainthm}, resp. \Cref{thm:generic} that $HF(\phi^k(L_i),L_j;\Lambda)$ has constant dimension with finitely many exceptions. Therefore, as there are finitely many $i,j$, $\sum_{i,j} dim HF(\phi^k(L_i),L_j;\Lambda)$ is bounded for $k\in\bN$, and the conclusion follows. 	
\end{proof}
\Cref{cor:entropy} is not very surprising: in the case of surfaces the logarithmic growth of fixed point Floer cohomology can be computed as the smallest topological entropy in the mapping class and this confirms a categorical analogue for the identity component. We believe our techniques can be useful in showing the deformation invariance of the categorical entropy for other symplectic mapping classes as well.

Next,
\begingroup
\def\thethm{\ref*{cor:seidelsconj}}
\begin{cor}
Assume that $L$ is connected, and $\pi_1(L)$ is abelian. Then, the set 
\begin{equation}
	\{k\in\bN:\phi^k(L) \text{ and }L'\text{ are Floer theoretically isomorphic} \}
\end{equation}
is either empty, a singleton or the entire $\bN$. In other words, if $\phi^k(L)$ and $L'$ are isomorphic for two different $k\in\bN$, then they are isomorphic for all $k\in\bN$. 	
\end{cor}
\addtocounter{thm}{-1}
\endgroup
To prove \Cref{cor:seidelsconj}, we need
\begin{lem}\label{lem:nonexactdefo}
Assume that $\alpha|_L$ is not exact. Then, for sufficiently small $|f|\neq 0$, the dimension of $HF(\phi_\alpha^f(L),L;\Lambda)$ is strictly less than the dimension of $H^*(L,\Lambda)\cong HF(L,L;\Lambda)$.
\end{lem}
\begin{proof}
Consider the $\Lambda$-local system $\xi_f$ on $L$ corresponding to the rank $1$ representation of $\pi_1(L)$ that is defined by $\pi_1(L)\to \Lambda^*$, $C\mapsto T^{f\alpha(C)}$. The chain complex defined in \Cref{defn:algyoneda} can be interpreted as $CF(\cdot, (L,\xi_f))$, and it follows from the proof of \Cref{lem:halg=h} that $HF(L,\phi_\alpha^f(L))=HF(L,(L,\xi_f))$ for small $|f|$. A twisted version of the PSS-isomorphism implies that $HF(L,(L,\xi_f))\cong H^*(L,\xi_f)$. We prove the claim in three steps: (i) we first show that the dimension of $H^*(L,\xi_f)$ is upper semi-continuous in $f$, i.e. for sufficiently small $|f|\neq 0$, the dimension of $H^*(L,\xi_f)$ is at most $H^*(L,\Lambda)$, (ii) we show that $H^1(L,\xi_f)$ is isomorphic to the group cohomology $H^1(\pi_1(L),\xi_f)$, and (iii) we prove that $H^1(\pi_1(L),\xi_f)$ is $0$ for sufficiently small $|f|$, unless $f=0$. These imply together that the total dimension of $H^*(L,\xi_f)$ is strictly less than that of $H^*(L,\Lambda)$, for small $|f|>0$, concluding the proof.

For the first step, we consider the cohomology with local coefficients as the sheaf cohomology computed by the \v{C}ech complex. Namely, endow $L$ with a finite cover $\cU$ with contractible intersections. Let $H=\pi_1(L)/ker(\alpha)$, and consider the $\pi_1(L)$-module $\Lambda[z^H]$. There is a corresponding locally constant sheaf of $\Lambda[z^H]$-modules, which we temporarily denote by $\xi$. Consider the ring homomorphism $\mathbf{z}_f:\Lambda[z^H]\to \Lambda$ given by $z^C\mapsto T^{f\alpha(C)}$ (we use the notation of \Cref{cor:semicontfnc}), and notice that the base change $\xi|_{\mathbf{z}=\mathbf{z}_f}:=\xi\otimes_{\Lambda[z^H]} \Lambda$ is the same as $\xi_f$. The \v{C}ech complex $C^*(\cU,\xi)$ is a complex of finite rank free $\Lambda[z^H]$-modules, and the cohomology of $\cC|_{\mathbf{z}=\mathbf{z}_f}:=\cC\otimes_{\Lambda[z^H]} \Lambda$ is the same as $H^*(L,\xi_f)$. We show that the dimension of this cohomology is upper semi-continuous analogous to \Cref{lem:semicontcomplexnovsingle}: namely, as in the proof of \Cref{lem:semicontcomplexnovsingle}, the dimension of the cohomology of $(\cC,d)|_{\mathbf{z}=\mathbf{z}_f}$ is given by 
\begin{equation}\label{eq:rankmank}
	rank(\cC)-2dim(im(d|_{\mathbf{z}=\mathbf{z}_f}))=rank(\cC)-2rank(d|_{\mathbf{z}=\mathbf{z}_f}).
\end{equation}
As in the proof of \Cref{lem:semicontcomplexnovsingle}, choose a basis for $\cC$, and consider a square submatrix of $d$ of size $rank(d|_{\mathbf{z}=\mathbf{z}_0=\mathbf{1}})$ whose determinant does not vanish at $\mathbf{z}=\mathbf{z}_0=\mathbf{1}$ (i.e. at $f=0$). Let $F\in \Lambda[z^H]$ denote the determinant of this submatrix (in particular, $F(\mathbf{z}_0)=F(\mathbf{1})\neq0$). By \Cref{cor:semicontfnc}, $F(\mathbf{z}_f)\neq 0$ for small $|f|$; hence, the rank of $d$ is lower semi-continuous at $0$. This implies the claim about the dimension of $H^*(L,\xi_f)$ by \eqref{eq:rankmank}.

Second, we show that $H^1(L,\xi_f)\cong H^1(\pi_1(L),\xi_f)$. As $H^*(\pi_1(L),\xi_f)\cong H^*(B\pi_1(L),\xi_f)$, we have to show $H^1(L,\xi_f)\cong H^1(B,\xi_f)$ for a $K(\pi_1(L),1)$ space $B$. One can obtain such a space by attaching (possibly infinitely many) cells to $L$ of dimension $3$ and higher. This procedure does not change the first cohomology (with local coefficients, which follows from the same Mayer--Vietoris argument). 

Finally, we prove that $H^*(\pi_1(L),\xi_f)\cong 0$ for sufficiently small $|f|$ unless $f=0$.  The group $\pi_1(L)$ admits a splitting $\pi_1(L)\cong\bZ\times \pi$ such that $\alpha$ is non-vanishing on the $\bZ$ component. Indeed, $\alpha$ defines a surjection onto a non-zero additive subgroup $G_0$ of $\bR$, which is necessarily torsion free and finitely generated; hence, free abelian. Consider any surjection $\beta:G_0\to\bZ$, and the composition $\pi_1(L)\xrightarrow{\alpha} G_0\xrightarrow{\beta}\bZ$. Let $\pi=ker(\beta\circ\alpha)$ and let $C\in\pi_1(L)$ be any element such that $\beta(\alpha(C))=1\in \bZ$. Then, $\pi_1(L)=\langle C\rangle\times \pi$ (this is where we use the abelian assumption on $\pi_1(L)$). Clearly, $\alpha(C)\neq 0$, and thus $C$ is not torsion. 

The $\pi_1(L)$-module $\xi_f$ can be seen as an exterior tensor product of the $\bZ=\langle C\rangle$-module $\xi_1=\xi_f|_{\langle C\rangle}$ and the $\pi$-module $\xi_2=\xi_f|_\pi$. As a result, $\xi_f$, considered as a module over the group ring $\Lambda[z^{\pi_1(L)}]\cong \Lambda[z^\bZ]\otimes \Lambda[z^\pi]$ is also an exterior tensor product of the $\Lambda[z^\bZ]$-module $\xi_1$ and the $\Lambda[z^\pi]$-module $\xi_2$. We have 
\begin{equation}\label{eq:something}
H^*(\pi_1(L),\xi_f)\cong RHom_{\Lambda[z^{\pi_1(L)}]}(\Lambda\boxtimes\Lambda,\xi_1\boxtimes\xi_2)\cong RHom_{\Lambda[z^\bZ]}(\Lambda,\xi_1)\otimes RHom_{\Lambda[z^\pi]}(\Lambda,\xi_2) .
\end{equation}
This can be seen using explicit resolutions of $\Lambda[z^\bZ]$-module, resp. $\Lambda[z^\pi]$-module $\Lambda$, and the duality. 
%
On the other hand, $RHom_{\Lambda[z^\bZ]}(\Lambda,\xi_1)$ can be explicitly computed as 
\begin{equation}
Hom_{\Lambda[z^\bZ]}(\Lambda[z^\bZ]\xrightarrow{z-1}\Lambda[z^\bZ],\xi_1 )=\Lambda\xrightarrow{T^{f\alpha(C)}-1}\Lambda.
\end{equation}
This complex is acyclic, unless $f=0$, which implies \eqref{eq:something} is $0$. This concludes the proof. 


\end{proof}
\begin{proof}[Proof of \Cref{cor:seidelsconj}]
If $L'\simeq \phi^{k_1}(L)\simeq \phi^{k_2}(L)$ for some $k_1\neq k_2$, then $L\simeq \phi^{k_1-k_2}(L)$. Therefore, it suffices to prove \Cref{cor:seidelsconj} when $L'=L$. Assume that $L\simeq \phi^{k_0}(L)$ for some $k_0>0$. Therefore, the dimension of $HF(\phi^k(L),L;\Lambda)$ is equal to the dimension of $H^*(L,\Lambda)$ for infinitely many $k$. If $\alpha|_L$ is exact, then $L$ and $\phi(L)$ are Hamiltonian isotopic, and there is nothing to prove. Assume that $\alpha|_L$ is not exact.

Let $p$ be a prime. The proofs of \Cref{thm:mainthm} and \Cref{thm:generic} imply the stronger statement that the dimension of $HF(\phi^f_\alpha (L),L;\Lambda)$ is constant for $f\in p^n\bZ_{(p)}$, with finitely many possible exceptions. On the other hand, we have observed that this dimension is the same as the dimension of $H^*(L,\Lambda)$ for infinitely many $k\in\bN$ (in particular for every $k\in k_0p^n\bN\subset p^n\bZ_{(p)}$). Hence, $dim_\Lambda(HF(\phi^f_\alpha (L),L;\Lambda))=dim_\Lambda(H^*(L,\Lambda))$ for all but finitely many $f\in p^n\bZ_{(p)}$. This contradicts \Cref{lem:nonexactdefo}, as a non-zero $f\in p^n\bZ_{(p)}$ can be arbitrarily small in Euclidean topology, and $dim_\Lambda(HF(\phi^f_\alpha (L),L;\Lambda))<dim_\Lambda(H^*(L,\Lambda))$ for such $f$. This concludes the proof. 
\end{proof}

\appendix
\section{$p$-adic numbers and Tate algebras}\label{appendix:tatealgebras}
In this \namecref{appendix:tatealgebras}, we remind the basics of $p$-adic numbers and Tate algebras, largely following \cite{kedlayaoverconvergent,bosch}. We focus on $\bQ_p$, but the constructions below go through for other complete non-Archimedean fields with a multiplicative norm. In particular, the statements below hold for the Novikov field as well.

For a given prime $p$, there is a discrete valuation $val_p:\bQ^*\to\bZ$, i.e. a function satisfying (i) $val_p(ab)=val_p(a)+val_p(b)$, (ii) $val_p(a+b)\geq \min\{val_p(a),val_p(b)\}$, if $a+b\in\bQ^*$. The value of $val_p(q)$ is defined to be the multiplicity of $p$ in the denominator of $q$ subtracted from that of the numerator. There is an associated norm defined by $|q|_p=p^{-val_p(q)}$, which makes $\bQ$ an incomplete normed field. Its completion is denoted by $\bQ_p$, and is called the field of $p$-adic numbers. This field is non-Archimedean, i.e. $|a+b|_p\leq \max\{|a|_p,|b|_p \}$. Thanks to this property, the unit disc is a unital subring, which is denoted by $\bZ_p$.  

One can define the corresponding Tate algebras as follows: given $\rho=(\rho_1,\dots,\rho_l)\in\bR^l$, let 
\begin{equation}
	\bQ_p\langle t_1,\dots,t_l\rangle_\rho:=\Bigg\{\sum_{I\in \bN^l}a_It^I:\lim\limits_{|I|\to\infty} \rho^I|a_I|_p=0 \Bigg\}.
\end{equation}
Assume from now on that all $\rho_i$ is in the image of $|\cdot|_p$. 
Due to the non-Archimedean feature, the convergence condition is equivalent to the convergence of $\sum_{I\in \bN^l}a_It^I$ over the polydisc $\bD_\rho:=\{(q_1,\dots,q_l)\in\bQ_p:|q_i|_p\leq \rho_i\}$. 
In other words, this ring can be seen as the ring of analytic functions on the polydisc. Note that $\bQ_p\langle t_1,\dots,t_l\rangle_\rho$ itself is a complete normed ring, where the norm is given by $||f||_\rho:=\max\limits_{I\subset\bN^l} \rho^I|a_I|_p$ (\cite[\S6]{kedlayaoverconvergent}). This norm is also non-Archimedean (i.e. $||f+g||_\rho\leq \max\{||f||_\rho,||g||_\rho\}$) and submultiplicative (i.e. $||fg||_\rho\leq ||f||_\rho.||g||_\rho$). The former is immediate from the definition and the latter follows easily by using the former. More precisely, if $f=\sum_{I\in \bN^l}a_It^I$, $g=\sum_{I\in \bN^l}b_It^I$, then the $t^I$-coefficient of $fg$ is given by $c_I=\sum_{I=I_0+I_1}a_{I_0}b_{I_1}$, and \begin{equation}
	\rho^I|c_I|_p\leq \rho^I\max\limits_{I=I_0+I_1}\{|a_{I_0}|_p.|b_{I_1}|_p\}=\max\limits_{I=I_0+I_1}\{\rho^{I_0}|a_{I_0}|_p.\rho^{I_1}|b_{I_1}|_p\}\leq ||f||_\rho.||g||_\rho.
\end{equation}
When $\rho=(1,\dots,1)$, this norm is actually multiplicative (\cite[p. 13]{bosch}).  

Given $\bQ_p\langle t_1,\dots,t_l\rangle_\rho$-valued matrix $M=(M^{ij})$, define $||M||_\rho=\max\limits_{i,j}||M^{ij}||_\rho$. It follows from the non-Archimedean property and the submultiplicativity that $||M_1M_2||_\rho\leq ||M_1||_\rho.||M_2||_\rho$, for any pair of matrices that can be multiplied. The proof is similar: namely, the $(ij)^{th}$ entry of $M_1M_2$ is given by $\sum_k M_1^{ik}M_2^{kj}$, and
\begin{equation}
	 \Bigg|\Bigg| \sum_k M_1^{ik}M_2^{kj}   \Bigg|\Bigg|_\rho\leq \max\limits_{k}\{||M_1^{ik}M_2^{kj}||_\rho \}\leq \max\limits_{k}\{||M_1^{ik}||_\rho.||M_2^{kj}||_\rho \}\leq ||M_1||_\rho.||M_2||_\rho.
\end{equation} 

We are mostly concerned with the case all $\rho_1=\dots =\rho_l=1$, and in this case we will drop the subscript $\rho$. In other words, $\bQ_p\langle t\rangle:=\bQ_p\langle t\rangle_1$, is the ring of analytic functions on the $p$-adic unit disc, and $\bQ_p\langle t_1,t_2\rangle=\bQ_p\langle t_1,t_2\rangle_{(1,1)}$. Observe that when $\rho_i=p^{-n_i}$, 
\begin{equation}\label{eq:discpadicfnc}
	\bQ_p\langle t_1,\dots,t_l\rangle_\rho=\bQ_p\langle t_1/p^{n_1},\dots,t_l/p^{n_l}\rangle:=\bQ_p\langle t_1/p^{n_1},\dots,t_l/p^{n_l}\rangle_{(1,\dots,1)}	.
\end{equation}
An easy way to see this is as follows: the left side is the ring of series that are convergent for $|t_i|_p\leq \rho_i$, whereas the right-hand side consists of series that converge when $|t_i/p^{n_i}|_p\leq 1$. These conditions are equivalent as $|p^{n_i}|_p=p^{-n_i}$. Throughout the paper, we prefer $\bQ_p\langle t_1/p^{n_1},\dots,t_l/p^{n_l}\rangle$ as our notation.

Despite being the ring of functions on a polydisc, the algebras $\bQ_p\langle t_1,\dots,t_l\rangle$ are Noetherian (\cite[Chapter 2, Proposition 14]{bosch}), and $\bQ_p\langle t\rangle$ is a PID (\cite[Section 2, Cor 10]{bosch}). Moreover, $f(t)\in\bQ_p\langle t\rangle$ vanishes only at finitely many elements of $\bQ_p$ by Strassman's theorem (\cite{strassmannuber}, \cite[p.62, Theorem 4.1]{casselslocal}, \cite[Theorem 3.38]{padicanalysiskatok}). These statements exhibit some odd features of the $p$-adic analytic geometry that contrast the complex analysis. 

One can associate a ringed space to algebras as above, denoted by $Sp(\bQ_p\langle t_1,\dots,t_l\rangle)$ (\cite{bosch}). We will avoid formal use of rigid geometry; however, we believe keeping the geometric perspective in mind is helpful in understanding the ideas.

The following semi-continuity result was used to prove \Cref{prop:grouplikepadic}:
\begin{lem}\label{lem:semiconqpaff}
	Let $(C,d)$ be a $\bZ$-graded complex of free $\bQ_p\langle t_1,\dots, t_l\rangle$-modules that is of finite rank at each degree. Assume that the restriction of $(C,d)$ to $t_1=\dots=t_l=0$ has vanishing cohomology at degree $i$, i.e. $H^i(C|_{\mathbf{t=0}},d|_{\mathbf{t=0}})=0$. Then for sufficiently small $\rho_1,\dots\rho_l>0$, 
	\begin{equation}
		H^i(C\otimes_{\bQ_p\langle t_1,\dots,t_l\rangle}\bQ_p\langle t_1,\dots,t_l\rangle_\rho,d)=0.
	\end{equation}
\end{lem}
\begin{proof}
	Let $(C,d)=\{\dots\to C^{i-1}\xrightarrow{d_{i-1}} C^i \xrightarrow{d_i}	C^{i+1}\to\dots \}$. By choosing bases for the free modules $C^{i-1}$, $C^i$ and $C^{i+1}$, one can see $d_{i-1}$ and $d_i$ as matrices such that $d_id_{i-1}=0$. Let $(B,d')=\{\dots\to B^{i-1}\xrightarrow{d_{i-1}'} B^i \xrightarrow{d_i'}	B^{i+1}\to\dots \}$ denote the restriction of $(C,d)$ to $t_1=t_2=0$. As $H^i(B,d')=0$ and $(B,d')$ is defined over a field $\bQ_p$, there exists maps $h:B^i\to B^{i-1}$, $h':B^{i+1}\to B^i$ such that $d_{i-1}'h+h'd_i'$ is the identity matrix on $B^i$. 
	%
	
	Fix extensions of the maps $h,h'$ to $\bQ_p\langle t_1,\dots,t_l\rangle$-valued matrices $h:C^i\to C^{i-1}$, $h':C^{i+1}\to C^i$ (such as the constant extensions, we denote them by the same letters). Let $D=id_{C^i}-d_{i-1}h-h'd_i:C^i\to C^i$. For any $s\in C^i$, $D(s)$ vanishes at $\mathbf{t=0}$. Therefore, the matrix coefficients of $D$ are all multiples of $t_i$, i.e. they are $O(t_1,\dots,t_l)$. 
	
	Given $s\in ker(d_i)$, $s=d_{i-1}h(s)+D(s)$; thus, $D(s)\in ker(d_i)$ as well. As a result, $D(s)=d_{i-1}h(D(s))+D(D(s))$, and $D^2(s)\in ker(d_i)$ as well. By iterating this process, we find that
	\begin{equation}\label{eq:iterates}
		s=d_{i-1}h(s)+d_{i-1}hD(s)+d_{i-1}hD^2(s)+\dots +d_{i-1}hD^k(s)+D^{k+1}(s)
	\end{equation}
	for every $k\in\bN$ and $s\in ker(d_i)$. We will produce a primitive for $s$ by showing that $\sum_{i=0}^{\infty}hD^i(s)$ converges in the operator norm (possibly for a smaller radius), and its differential is equal to $s$, under the assumption that $d_i(s)=0$. 
	
	First, as the matrix entries of $D$ are in $\bQ_p\langle t_1,\dots,t_l\rangle$, the $p$-adic norms of the coefficients of the entries are bounded above (if $f=\sum_{I\in\bN^l} a_It^I\in \bQ_p\langle t_1,\dots,t_l\rangle$, then $|a_I|_p\to 0$). Let $D^{ij}=\sum_{I\in \bN^l} D^{ij}_It^I$ denote the $(ij)^{th}$-entry of $D$, and let $N$ be such that $|D^{ij}_I|_p\leq N$ for all $I,i,j$. Let $0<\rho_i<\epsilon< 1/N$. Then
	\begin{equation}\label{eq:operatornorm}
		||D||_\rho=\max\limits_{I,i,j}|D^{ij}_I|_p\rho^I\leq (\max\limits_{I,i,j}|D^{ij}_I|_p).\epsilon<1.
	\end{equation}
The first inequality $\max\limits_{I,i,j}|D^{ij}_I|_p\rho^I\leq (\max\limits_{I,i,j}|D^{ij}_I|_p).\epsilon$ holds as $D|_{\mathbf{t=0}}=0$; therefore, either $D^{ij}_I=0$ or $I\neq (0,\dots,0)$.

By submultiplicativity of the norm, $||hD^i(s)||_\rho\leq ||D||_\rho^i||h||_\rho||s||_\rho$, which goes to $0$ as $i\to \infty$ by \eqref{eq:operatornorm}. Thus, $\sum_{i=0}^{\infty}hD^i(s)$ converges to a $\bQ_p\langle t_1,\dots ,t_l\rangle_\rho$-valued column matrix. Moreover, as the matrix multiplication is continuous, its differential is equal to $\sum_{i=0}^{\infty}d_{i-1}hD^i(s)$ and it converges to $s$ by \eqref{eq:iterates}, as $||D^{k+1}(s)||_\rho\leq ||D||_\rho^{k+1}||s||_\rho\xrightarrow{k\to\infty}0$. Therefore, every $s\in ker(d_i)$ is a coboundary $C\otimes_{\bQ_p\langle t_1,\dots,t_l\rangle}\bQ_p\langle t_1,\dots ,t_l\rangle_\rho$ (observe that the radius $\rho$ does not depend on $s$). This finishes the proof. 	
\end{proof}
We also have the following result:
\begin{lem}\label{lem:freeres}
Every ($\bZ/2\bZ$-graded) complex $(B,d)$ over the Tate algebra $A=\bQ\langle t_1,\dots,t_l\rangle_\rho$ with finitely generated cohomology is quasi-isomorphic to a ($\bZ/2\bZ$-graded) complex of finitely generated free modules.
\end{lem}
The lemma requires the standing assumption that each $\rho_i$ is in the image of the norm.
\begin{proof}
Let $F$ be a free, finitely generated $A$-module with a surjection $F\to H^0(B)$. Choose a lift $F\to ker(d^0)\subset B^0$, and consider the complex $B'=\{F\rightleftarrows 0\}$ (i.e. the complex is $F$ in even degrees, $0$ in odd degrees, in particular, $H^1(B')=0$). There is a clear map of complexes $B'\to B$, extending $F\to B^0$, and it induces a surjection $H^0(B')\twoheadrightarrow H^0(B)$. 
Let $B''$ denote the cone of this map; hence, there exists a long exact sequence
\begin{equation}
\dots \to H^1(B')=0\to H^1(B)\to H^1(B'')\to H^0(B')\twoheadrightarrow H^0(B)\to H^0(B'')\to 0=H^1(B')\to\dots .
\end{equation}
It follows that $H^0(B'')=0$. As $B'$ is a complex of finitely generated free modules, proving $B$ is quasi-isomorphic to such a complex is equivalent to proving $B''$ is quasi-isomorphic to such a complex. Therefore, we can assume without loss of generality that $B=B''$ and that $H^0(B)=0$. 

By \cite[Proposition 6.5]{kedlayaoverconvergent}, every finitely generated module over the Tate algebra has a finite free resolution. 
Fix such a resolution 
\begin{equation}
	F_k\to \dots F_1\to H^1(B).
\end{equation}
Choose a lift $F_1\to ker(d^1)\subset B^1$. The composition $F_2\to F_1\to B^1$ is not necessarily $0$; however, its image is in $im(d^0)$ (as the composition $F_2\to F_1\to H^1(B)$ is $0$). Therefore, we can fix a lift $F_2\to B^0$ such that
\begin{equation}\label{eq:shortchainmap}
	\xymatrix{ B^0\ar[r]&B^1\ar[r]& B^0 \\
	F_2\ar[u]\ar[r]&F_1\ar[u]\ar[r]& 0\ar[u] }
\end{equation}
commutes. The right square commutes as the image of $F_1\to B^1$ is in $ker(d^1)$. We can inductively extend \eqref{eq:shortchainmap} to a map of $\bZ$-graded complexes
\begin{equation}\label{eq:longchainmap}
	\xymatrix{ \dots\ar[r]&B^0\ar[r]&B^1\ar[r]& B^0\ar[r]&B^1\ar[r]& B^0 \ar[r]&B^1\ar[r]&\dots  \\
	\dots\ar[r]&F_4\ar[r]\ar[u]& F_3\ar[u]\ar[r]&	F_2\ar[u]\ar[r]&F_1\ar[u]\ar[r]& 0\ar[u]\ar[r]& 0\ar[u]\ar[r]&\dots  }
\end{equation}
Denote the bottom complex by $\tilde F$, and the top complex by $\tilde B$, and assume that they are graded so that $F_i$ is the degree $(-i)$-part of $\tilde F$. The chain map \eqref{eq:longchainmap} induces an isomorphism in cohomology in (i) even degrees (as both have vanishing cohomology in even degrees), (ii) at degree $(-1)$ (by construction). Consider the doubly periodic complex $\bigoplus_{m\in\bZ}\tilde F[2m]$ and the natural doubly periodic chain map 
\begin{equation}\label{eq:doublyper}
	\bigoplus_{m\in\bZ}\tilde F[2m] \to\tilde B. 
\end{equation}
Then \eqref{eq:doublyper} is a quasi-isomorphism. Note that the left-hand side is of finite rank in each degree, as we used a finite resolution to begin with. Therefore, \eqref{eq:longchainmap} can be seen as a quasi-isomorphism from a $\bZ/2\bZ$-graded complex of finite rank free modules to $B$. This finishes the proof.
\end{proof}
\begin{rk}
In the $\bZ$-graded case, \cite[Proposition 6.5]{kedlayaoverconvergent} implies the analogue of \Cref{lem:freeres} by a standard method. The subtlety of \Cref{lem:freeres} was due to $\bZ/2\bZ$-gradings. 
\end{rk}
We also work frequently with the non-Archimedean field $\Lambda$, defined by
\begin{equation}
	\Lambda=\bQ((T^\bR))=\bigg\{\sum_{i=0}^\infty a_iT^{r_i}: a_i\in\bQ, r_i\in\bR,r_i\to\infty \bigg\},
\end{equation}
which is called \emph{the Novikov field}. This field carries a valuation $val_T$, defined by $val_T(\sum_{i=0}^\infty a_iT^{r_i})=\min\limits_{i\in\bN}\{r_i\}$, and a complete non-Archimedean norm $|\cdot|_T:=e^{-val_T(\cdot)}$. As we remarked, the constructions and the statements hold verbatim for $(\Lambda,|\cdot|_T)$. Notice that $|\cdot|_T$ is surjective onto $[0,\infty)$; therefore, the assumption that $\rho_i$ is in the image of $|\cdot|_T$ is automatic.

We also need to consider ``annuli'' over $\Lambda$. Consider the ring 
\begin{equation}
\Lambda\langle \zeta^{-1} z,1/z \rangle=	\Lambda\langle w,y \rangle/(wy-\zeta),
\end{equation}
where $\zeta\in\Lambda\setminus\{0\}$. Morally, this is the ring of functions on a subspace of the polydisc corresponding to $\Lambda\langle w,y \rangle$, and its points are given by $|\zeta^{-1} z|_T\leq 1$, $|1/z|_T\leq 1$, i.e. $1\leq |z|_T\leq |\zeta|_T$ (equivalently $val_T(\zeta)\leq val_T(z)\leq 0$). More generally, consider the ring
\begin{equation}
	\Lambda\langle T^{-b}z, T^c/z \rangle =\Lambda\langle w, y \rangle/(wy-T^{c-b}). 
\end{equation}
This ring can be thought of as the ring of functions on the annulus of points $z$ satisfying $b\leq val_T(z)\leq c$. It is naturally isomorphic to the ring $\Lambda\{ z^\pm\}_{[b,c]}$ of series $\sum_{k\in\bZ}a_kz^k$ such that $val_T(a_k)+vk\to \infty$ for every $v\in [b,c]$. 
The quotient map $\Lambda\langle w, y \rangle\to \Lambda\{ z^\pm\}_{[b,c]}$ is given by $w\mapsto T^{-b}z$, $y\mapsto T^c/z$. 
This map is well-defined and surjective. The ring $\Lambda\{ z^\pm\}_{[b,c]}$ is Noetherian (as it is a quotient of a Tate algebra, which is itself Noetherian by \cite[Chapter 2, Proposition 14]{bosch}), regular, and of finite Krull dimension (as the Tate algebra has finite Krull dimension by \cite[p22 Proposition 17]{bosch}). As a result, it has a finite global dimension, and in particular, every finitely generated module over it has finite projective resolutions. 

More generally, given $b_1\leq c_1, \dots ,b_k\leq c_k$, one can define $\Lambda\{z_1^\pm,\dots, z_k^\pm \}$ to be the ring of series $\sum_{I=(i_1,\dots)\in\bZ^k} a_Iz^I$ such that $val_I(a_I)+v_1i_1+\dots+v_ki_k\to\infty$ for all $v_i\in[b_i,c_i]$. Similar statements hold, which we can summarize as 
\begin{lem}\label{lem:propsofring}
The ring $\Lambda\{z_1^\pm,\dots, z_k^\pm \}$ is Noetherian of finite global dimension. In particular, every finitely generated module over it has a finite projective resolution. 
\end{lem}
\begin{proof}
Analogous to before, one identifies $\Lambda\{z_1^\pm,\dots, z_k^\pm \}$ with the quotient of $\Lambda\langle w_1,y_1,\dots, w_k,y_k \rangle$ (which is Noetherian by \cite[Chapter 2, Proposition 14]{bosch}) by the ideal generated by $w_iy_i-T^{c_i-b_i}$. The rest of the statements follow analogously. 	
\end{proof}
\begin{cor}\label{lem:freeresannulus}
Every ($\bZ/2\bZ$-graded) complex $(B,d)$ over the algebra $\Lambda\{z_1^\pm,\dots, z_k^\pm\}$ with finitely generated cohomology is quasi-isomorphic to a ($\bZ/2\bZ$-graded) complex of finitely generated projective modules.
\end{cor}
\begin{proof}
The proof of \Cref{lem:freeres} works almost verbatim. One needs to replace the free resolutions with the projective resolutions, which exist by \Cref{lem:propsofring}. 
\end{proof}

\section{Semi-continuity statements and the proof of \Cref{lem:grouplikenovikov}}\label{appendix:semicont}
In this \namecref{appendix:semicont}, we collect the semi-continuity statements required for Lemma \ref{lem:grouplikenovikov} and Lemma \ref{lem:nearbyfloer}, and we give the proof of the former. 
We start with the following definition:
\begin{defn}
Given $b\leq c$, let $\Lambda\{z^\bR\}_{[b,c]}$ denote the ring consisting of series 
\begin{equation}
\sum a_rz^r,
\end{equation}
where $r\in\bR$, $a_r\in \Lambda$, 
and 
\begin{equation}
val_T(a_r)+r\nu\to \infty
\end{equation} 
for any $\nu\in[b,c]$. Only countably many $a_r$ can be non-zero, but we do not impose any other condition on the set of $r$ satisfying $a_r\neq 0$ (e.g. they can accumulate).
\end{defn}
We think of this ring as the ring of functions on the universal cover of $\{x\in\Lambda: b\leq val_T(x)\leq c \}$. Another heuristic is that one can think of this ring as a non-Archimedean analogue of the interval $[b,c]$. Observe that if $b<c$, $\Lambda\{z^\bR\}_{[b,c]}$ is independent of $b$ and $c$ up to isomorphism. The isomorphism is given by
\begin{equation}
\Lambda\{z^\bR\}_{[b,c]} \to \Lambda\{z^\bR\}_{[b',c']}\atop z^r\mapsto T^{\alpha r}z^{\beta r},
\end{equation}	where $\beta,\alpha$ are such that $\beta\nu+\alpha$ is the increasing linear bijection $[b',c']\to[b,c]$. We will often omit $b,c$ from the notation and denote this ring by $\Lambda\{z^\bR\}$. We will assume that $0\in(b,c)$ unless stated otherwise.

One can replace $\Lambda[z^\bR]$ in Definition \ref{defn:novikovfamily} with $\Lambda\{z^\bR\}$. Throughout this appendix, we will refer to this modified notion. In particular, we use the notation $\fM_{\alpha}^\Lambda$ to mean a Novikov family in the latter sense. For instance, $\fM_\alpha^\Lambda(L,L')=CF(L,L';\Lambda)\otimes_\Lambda \Lambda\{z^\bR\} $.

The ring $\Lambda\{z^\bR\}$ is not Noetherian. However, for our purposes, we only need a restricted set of exponents for $z$. In other words, our examples will be defined over a Noetherian subring of $\Lambda\{z^\bR\}$. More precisely, let $H:=H_1(M,\bZ)/ker(\alpha)$ denote the integral homology modulo the cycles that vanish under $\alpha$. If $C\in H_1(M,\bZ)$ has a multiple in $ker(\alpha)$, then $C\in ker(\alpha)$; therefore, $H$ is torsion free (thus also free). Consider the group ring $\Lambda[z^H]$, and the ring homomorphism 
\begin{equation}\label{eq:ringhom}
	\Lambda[z^H] \to \Lambda\{z^\bR\}
\end{equation}
given by $z^C\mapsto z^{\alpha(C)}$. As $C\mapsto \alpha(C)$ is injective, so is \eqref{eq:ringhom}. Therefore, $\Lambda[z^H]$ inherits a topology from $\Lambda\{z^\bR\}$, and its completion consists of series $\sum_{C\in H}a_Cz^C$ satisfying $val_T(a_C)+\alpha(C)v\to \infty$ for every $v\in[b,c]$. The closure of the image can be written as $\Lambda\{z^{\pm r_1},\dots z^{\pm r_k} \}$ for a fixed set $r_i$ of real numbers (not to be confused with $\Lambda\{z_1^\pm,\dots,z_k^\pm\}$ from \Cref{appendix:tatealgebras}). Our families are defined over this ring; however, this ring is still too big for our purposes. 

Instead, let $C_1,\dots,C_k\in H$ be a basis such that $\alpha(C_i)>0$. Define $b_i:=b.\alpha(C_i)$, $c_i:=c.\alpha(C_i)$, and $z_i:=z^{C_i}$ (therefore, $\Lambda[z^H]=\Lambda[z_1^\pm,\dots,z_k^\pm]$). Consider the ring $\Lambda\{z^H\}=\Lambda\{z^H\}_{[b,c]}$ consisting of series $\sum_{I=(i_1,\dots)\in \bZ^k}a_Iz^I$ such that $val_T(a_I)+i_1v_1+\dots +i_kv_k\to\infty$ for every $v_1\in[b_1,c_1],\dots v_k\in[b_k,c_k]$. Clearly, $\Lambda\{z^H\}$ is identified with the ring $\Lambda\{z_1^\pm,\dots,z_k^\pm\}$ from \Cref{appendix:tatealgebras}. 

The condition that $val_T(a_I)+i_1v_1+\dots +i_kv_k\to\infty$ for every $v_1\in[b_1,c_1],\dots v_k\in[b_k,c_k]$ implies the convergence condition for $\sum_{C\in H}a_Cz^C$ inherited from $\Lambda\{z^\bR\}$ (i.e. $val_T(a_C)+\alpha(C)v\to \infty$ for every $v\in[b,c]$) by letting $v_i=v.\alpha(C_i)$. 
As a result, we have a map 
\begin{equation}
	\Lambda\{z^H\} \to \Lambda\{z^\bR\}	.
\end{equation}
Clearly, this map is injective. 
%
\begin{lem}\label{lem:lifts}
The family $\fM_\alpha^\Lambda$ parametrized by $\Lambda\{z^\bR\}$ lifts to a family parametrized by $\Lambda\{z^H\}$. 
\end{lem}
\begin{proof}
We assign to $(L_i,L_j)$ the $\Lambda\{z^H\}$-algebra $CF(L_i,L_j;\Lambda)\otimes_\Lambda \Lambda\{z^H\}$, and define the structure maps similar to 	\eqref{eq:structurenovikov} and \eqref{eq:structuregenericnovikov}, with the coefficient $T^{E(u)}z^{\alpha([\partial_h u])} $ replaced by $T^{E(u)}z^{[\partial_h u]}$ (where we abuse the notation to denote the image of $[\partial_h u]$ in $H=H_1(M,\bZ)/ker(\alpha)$ by $[\partial_h u]$ as well).

One still needs to check that the coefficients belong to $\Lambda\{z^H\}$. In the monotone case, this is immediate as there are only finitely many terms. The general case can be achieved by using Fukaya's trick (cf. \cite{abouzaidicm,abouzaidfamilyfloerfaithful}), or alternatively the reverse isoperimetric inequalities (\cite{gromansolomonreverse,duvalreverse},\cite[Appendix A]{abouzaidhmscorrections}). 
More precisely, the length of the boundary of a disc is bounded above by its area (up to multiplication by a uniform constant). Moreover, the size of the class $[\partial_h u]$ is also bounded above by the length of the boundary (up to multiplying and adding with a uniform constant). Therefore, (for any choice of a metric on $H\otimes_\bZ\bR$) the size of $[\partial_h u]$ is bounded above by the energy (up to a constant), which guarantees the convergence of the series in $\Lambda\{z^H\}$ for $b_i,c_i$ sufficiently close to $0$. 
\end{proof}
\begin{defn}
Let $b\leq a\leq c$. Then, there is a ring homomorphism 
\begin{equation}
ev_{T^a}:\Lambda\{z^\bR\}_{[b,c]}\to \Lambda \atop \sum a_r z^r\longmapsto \sum a_rT^{a r}.
\end{equation}	
We call this map \emph{evaluation map at $z=T^a$}. Given $f\in\Lambda\{z^\bR\}_{[b,c]}$, $ev_{T^a}(f)$ is also denoted by $f(T^a)$, and the base change of a $\Lambda\{z^\bR\}_{[b,c]}$-module $M$ along $ev_{T^a}$ is denoted by $M|_{z=T^a}$. One can also think of $ev_{T^a}$ as a $\Lambda$-point of $\Lambda\{z^\bR\}_{[b,c]}$. 
\end{defn}
Even though we use real exponents for both $\Lambda$ and $\Lambda \{z^\bR \}$, there is no a priori relation between the topology of these rings and the topology in $\bR$. For instance, $t\mapsto T^t$ is not a continuous assignment. Nevertheless, we can still prove some semi-continuity results as $t\in\bR$ varies continuously. First, we prove:  
\begin{lem}\label{lem:uppersemicontfnc}
Let $f(z)\in\Lambda\{z^\bR\}$ be such that $f(1)\neq 0$. Then there exists $\delta>0$ such that $f(T^t)\neq 0$ for $|t|<\delta$.
\end{lem}
\begin{proof}
Let $f(z)=\sum a_rz^r$. Without loss of generality, assume that $f(1)=1$. By the convergence condition, there exists an $\epsilon>0$ such that $val_T(a_r)-\epsilon |r|\to\infty$. Therefore, $val_T(a_r)- \epsilon |r|\geq 1$ for almost every $r$. We split $f$ as $f_{1}+f_{2}$ where $f_{1}$ is a finite sub-sum $\sum_{r\in F} a_rz^r$ of $\sum a_rz^r$, and $f_2$ is the sum of remaining terms such that the coefficients of $f_2$ have $val_T(a_r)-\epsilon |r|\geq 1$. In particular, $f_2(1)=O(T)$, i.e. $val_T(f_2(1))\geq 1$, and $f_1(1)=1+O(T)$. Also, by assumption, $f_2(T^t)=O(T)$ for $t\in[-\epsilon,\epsilon]$. 

Let $a_r=\sum_\alpha a_r(\alpha)T^\alpha$ be the expansion of $a_r$. Then, $f_1(z)=\sum_{r\in F,\alpha}a_r(\alpha)T^\alpha z^r$ and $\sum_{r\in F} a_r(0)=1$ as $f_1(1)=1+O(T)$. Consider $f_1(T^t)=\sum_{r\in F,\alpha}a_r(\alpha)T^\alpha T^{tr}$. As $F$ is finite, there is a positive gap between the valuations of $a_r(0)T^{tr}$ terms of the sum and the terms $a_r(\alpha)T^{\alpha+ tr}$ with $\alpha\neq 0$, as long as $|t|$ is small. Therefore, these two types of terms of the sum cannot cancel out each other. Furthermore, for small $|t|$, $a_r(0)T^{tr}$ have valuation less than $1$, so these terms cannot cancel out with terms of $f_2(T^t)$. As $\sum_{r\in F} a_r(0)T^{tr}$ remains non-zero and has a different valuation than the remaining terms of $f(T^t)$, for a small variation of $t$, $f(T^t)$ remains non-zero.
\end{proof}
\begin{cor}\label{cor:semicontfnc}
Let $f\in\Lambda\{z^H\}$ be such that $f(\mathbf{1})=f(1,\dots,1)\neq 0$. Define $\mathbf{z}_t:=(T^{t\alpha(C_1)},\dots ,T^{t\alpha(C_k)})$. Then, $f(\mathbf{z}_t)\neq 0$ for $|t|$ sufficiently small.
\end{cor}
The point $\mathbf{z}_t$ corresponds to the composition of the map $\Lambda\{z^H\}\to\Lambda\{z^\bR\}$ such that $ z_j=z^{C_j}\mapsto z^{\alpha(C_j)}$ with the map $\Lambda\{z^\bR\}\to\Lambda$ such that $z^r\mapsto T^{tr}$. \Cref{cor:semicontfnc} follows by applying \Cref{lem:uppersemicontfnc} to the image $f(z^{\alpha(C_1)},\dots,z^{\alpha(C_k)})$ of $f$ under the former map. 
\begin{lem}\label{lem:semicontcomplexnovsingle}
Let $(C,d)$ be a (possibly unbounded) complex of finite rank projective $\Lambda\{z^H\}$-modules. Assume that $(C|_{\mathbf{z=1}},d|_{\mathbf{z=1}})$ have vanishing cohomology in degree $i$, i.e. $H^i(C|_{\mathbf{z=1}},d|_{\mathbf{z=1}})$. Then, for sufficiently small $|t|$, $(C,d)$ is acyclic at degree $i$ at $\mathbf{z}=\mathbf{z}_t:=(T^{t\alpha(C_1)},\dots,T^{t\alpha(C_k)})$, i.e. \begin{equation}
	H^i(C|_{\mathbf{z}=\mathbf{z}_t},d|_{\mathbf{z}=\mathbf{z}_t})=0. 
\end{equation}
More generally, the rank of this group is upper semi-continuous in $t$. 
\end{lem}
\begin{proof}
The statement is concerned with degree $i$; therefore, without loss of generality we can assume that $C$ is bounded (or even supported at degrees $i-1,i,i+1$). 
Recall that projective modules over local rings are free (\cite[\href{https://stacks.math.columbia.edu/tag/058Z}{Tag 058Z}]{projlocalstacks-project}). As a result, the localization of such a module at the ideal of $\mathbf{z=1}=(1,\dots,1)$ is free, which implies that there exists an element $g\in \Lambda\{z^H\}$ such that $g(\mathbf{1})\neq 0$ and the localization $C_g$ of $C$ at $g$ (obtained by inverting $g$) is a complex of free modules (we use boundedness here). By \Cref{lem:uppersemicontfnc}, $g(\mathbf{z}_t)\neq 0$ for sufficiently small $|t|$. In particular, the evaluation at $\mathbf{z}_t$ is well-defined over $\Lambda\{z^H\}_g$, for $|t|$ sufficiently small. By \Cref{cor:semicontfnc}, if $f\in \Lambda\{z^H\}_g$ is such that $f(\mathbf{1})\neq 0$, then $f(\mathbf{z}_t)\neq 0$ for sufficiently small $t$. 

Choose trivializations for every $C_g^j$ and consider $C^{i-1}_g\xrightarrow{d_{i-1}} C^i_g\xrightarrow{d_{i}} C^{i+1}_g$. We can see $d_{i-1}$ and $d_i$ as matrices with coefficients in $\Lambda\{z^H\}_g$ such that $d_id_{i-1}=0$. The rank of $H^i(C|_{\mathbf{z}=\mathbf{z}_t},d|_{\mathbf{z}=\mathbf{z}_t})$ is given by 
\begin{equation}\label{eq:rankhom}
rank(C_i)-dim(d_{i-1}|_{\mathbf{z}=\mathbf{z}_t})-rank(d_{i}|_{\mathbf{z}=\mathbf{z}_t}).	
\end{equation}
Choose a square sub-matrix of $d_{i-1}$, resp. $d_i$, such the restriction to $\mathbf{z=1}$ is non-singular and of size equal to $rank(d_{i-1}|_{\mathbf{z=1}})$, resp. $rank(d_i|_{\mathbf{z=1}})$. Let $f\in \Lambda\{z^H\}_g$ denote the product of determinants of these matrices. As $f(\mathbf{1})\neq 0$, $f(\mathbf{z}_t)\neq 0$ for sufficiently small $|t|$. As a result \eqref{eq:rankhom} has a local maximum at $t=0$. This finishes the proof. 
%
%
%
\end{proof}
It is possible to define multivariable versions of $\Lambda\{ z^\bR\}$, and prove analogues of Lemma \ref{lem:uppersemicontfnc}. For instance, let $b_1<c_1$ and $b_2<c_2$ be real numbers. Define $\Lambda\{z_1^\bR,z_2^\bR \}$ to be the series of the form 
\begin{equation}
\sum a_{\mathbf{r}} z_1^{r_1}z_2^{r_2},
\end{equation}
where $\mathbf{r}=(r_1,r_2)\in\bR^2$, $a_{\mathbf{r}}\in\Lambda$, and satisfying 
\begin{equation}
val_T(a_{\mathbf{r}})+r_1\nu_1+r_2\nu_2\to\infty
\end{equation} 
for any $\nu_1\in[b_1,c_1],\nu_2\in[b_2,c_2]$. As before, this ring does not depend on $b_1<c_1,b_2<c_2$, and we omit these from the notation. We will keep the assumption $b_1<0<c_1,b_2<0<c_2$. Lemma \ref{lem:uppersemicontfnc} and Lemma \ref{lem:semicontcomplexnovsingle} have their obvious generalizations. 
\begin{rk}
There is a ring homomorphism $\Lambda\{z^\bR \}\otimes \Lambda\{z^\bR \}\to \Lambda\{z_1^\bR,z_2^\bR \}$ such that $f\otimes g\mapsto fg$. One can presumably topologize the rings $\Lambda\{z^\bR \}$, $\Lambda\{z_1^\bR,z_2^\bR \}$ and prove that the map induces an isomorphism from the completed tensor product. We do not need this for our current purposes.
\end{rk}
We have given the proof of Lemma \ref{lem:nearbyfloer}. Finally, we have:
\begin{proof}[Proof of Lemma \ref{lem:grouplikenovikov}]
One can write the morphism (\ref{eq:novconvto}) of bimodules, which comes from a morphism of families
\begin{equation}\label{eq:novconvtofam}
\pi_1^*\fM_{\alpha}^\Lambda\relotimes_{\cF(M,\Lambda)} \pi_2^*\fM_{\alpha}^\Lambda\to \Delta^*\fM_{\alpha}^\Lambda.
\end{equation}
This is a map of families over $\Lambda\{z_1^\bR,z_2^\bR\}$; however, analogous to \Cref{lem:lifts}, both sides of \eqref{eq:novconvtofam} as well as the map in between lifts to $\Lambda\{z_1^H,z_2^H\}$. One defines the lift of (\ref{eq:novconvtofam}) by a formula similar to (\ref{eq:padicmapformula}), namely
\begin{equation}
\sum\pm T^{E(u)} z_1^{[\partial_1 u]} z_2^{[\partial_2 u]}.y
\end{equation}
is the coefficient of $y$, and the sum is over the same set of discs as in (\ref{eq:padicmapformula}). At $\mathbf{z_1=z_2=1}$, this defines the canonical quasi-isomorphism (\ref{eq:convolutiontodiagformula}). As remarked in \Cref{rk:locconst}, the notion of perfect families generalize to more general rings, and \Cref{lem:properfamily} holds to families over $\Lambda\{z_1^H,z_2^H\}$ (the proof requires  \Cref{lem:freeresannulus} instead of \Cref{lem:freeres}). In other words, proper families are perfect; therefore, so are their convolution (cf. \Cref{cor:properconv}). This implies that the lift of the cone of \eqref{eq:novconvtofam} to $\Lambda\{z_1^H,z_2^H\}$ is perfect and proper. Let $\fN$ denote this lift. Consider $\bigoplus_{i,j}\fN(L_i,L_j)$, and let $C$ be a projective replacement for it (which exists by \Cref{lem:freeresannulus}). As before, the (derived) restriction of $\bigoplus_{i,j}\fN(L_i,L_j)$ to $\mathbf{z=1}$ vanishes; therefore, $C|_{\mathbf{z=(1,1)}}\simeq 0$. One can see $C$ as a $2$-periodic unbounded complex. By a two-variable version of \Cref{lem:semicontcomplexnovsingle}, 
\begin{equation}
	H^0(C|_{\mathbf{z}=(\mathbf{z}_{t_1},\mathbf{z}_{t_2})})=H^1(C|_{\mathbf{z}=(\mathbf{z}_{t_1},\mathbf{z}_{t_2})})=0
\end{equation}
for sufficiently small $|t_1|$, $|t_2|$. Since $C$ is $2$-periodic, this means $C|_{\mathbf{z}=(\mathbf{z}_{t_1},\mathbf{z}_{t_2})}$ is acyclic and the lift of \eqref{eq:novconvtofam} to $\Lambda\{z_1^H,z_2^H\}$ is an isomorphism at $(\mathbf{z}_{t_1},\mathbf{z}_{t_2})$. This is just a restatement of \Cref{lem:grouplikenovikov}.

\end{proof}

\bibliographystyle{alpha}
\bibliography{bibliogeneral}	


\end{document}